\documentclass[12pt]{article}

\usepackage[utf8]{inputenc} 
\usepackage{amsfonts}       
\usepackage{amssymb}
\usepackage{amsmath}
\usepackage[mathscr]{euscript}
\usepackage{xcolor}         
\usepackage{stmaryrd}       
\usepackage{wasysym}
\usepackage{empheq}
\usepackage{hyperref}       
\usepackage{enumitem}

\usepackage[colors]{optsys}

\renewcommand{\UNCERTAIN}{{\mathcal W}}
\renewcommand{\PRIMAL}{{\mathcal X}}
\renewcommand{\DUAL}{{\mathcal Y}}
\renewcommand{\SPHERE}{S}
\renewcommand{\BALL}{B}

\newcommand{\MinkowskiFunctional}[1]{m_{#1}}

\newcommand{\UppTimes}{\stackops{\cdot}{\times}{-0.2ex}}
\newcommand{\LowTimes}{\stackops{\cdot}{\times}{-2.5ex}}

\newcommand{\Group}{{\mathcal A}} 
\newcommand{\barGroup}{\overline{\Group}}
\newcommand{\group}{a}
\newcommand{\groupbis}{b}
\newcommand{\groupter}{c}
 
\newcommand{\GroupOperation}{*}
\newcommand{\UppGroupOperation}{\stackops{\cdot}{\GroupOperation}{-0.2ex}}
\newcommand{\LowGroupOperation}{\stackops{\cdot}{\GroupOperation}{-2.0ex}}
\newcommand{\UnitElement}{e}

\newcommand{\valley}{v}

\renewcommand{\TripleNorm}{\DoubleNorm}
\renewcommand{\TripleNormDual}{\DoubleNormDual}

\newcommand{\TopSupscript}{\top}
\newcommand{\SupportSupscript}{\TopSupscript\!\star} 
\renewcommand{\TopNorm}[2]{{#1}_{(#2)}^{\TopSupscript}}

\renewcommand{\SupportNorm}[2]{{#1}_{(#2)}^{\SupportSupscript}}

\newcommand{\BipolarFunction}{bipolar function}
\newcommand{\BipolarSet}{bipolar set}

\newcommand{\GSubset}{\Lambda}
\newcommand{\CCSubset}{\Gamma}
\newcommand{\PHCCSubset}{\CCSubset_{\mathtt{H}}}

\newcommand{\GLower}[3]{#1\nc{#2}}
\newcommand{\GreatestIn}[3]{{#1\nc{#2}^{\top}}}
\newcommand{\closedconvexcone}{\overline{\mathrm{cone}}}

\newcommand{\BIPOLARSET}{\mathcal{B}p[\PRIMAL]}
\newcommand{\BIPOLARFUNC}{\mathcal{B}p[\barRR^\PRIMAL]} 

\usepackage{authblk}
\usepackage{fullpage}

\usepackage{graphicx}
\graphicspath{{./FIGURES/}}

\title{A Unified View of Polarity for Functions}

\author[1]{Jean-Philippe Chancelier}
\author[1]{Michel De Lara}
\affil[1]{CERMICS, \'Ecole nationale des ponts et chaussées, IP Paris, France}

\begin{document}

\maketitle


\begin{abstract}
We propose a unified view of the polarity of functions,
that encompasses all specific definitions, generalizes several well-known properties
and provides new results.
We show that bipolar sets and bipolar functions are isomorphic lattices.
Also, we explore three possible notions
of polar subdifferential associated with a nonnegative function,
and we make the connection with the notion of alignement of vectors. 
\end{abstract}

\section{Introduction}
\label{Introduction}

The introduction of~\cite[Chapter~11, Sect.~E]{Rockafellar-Wets:1998} reads as
follows: \emph{While most of the major duality correspondences, like convex sets
  versus sub-linear functions, or polarity of convex cones, ﬁt directly within
  the framework of conjugate convex functions as in 11.4, others, like polarity
  of convex sets that aren’t necessarily cones but contain the origin, ﬁt
  obliquely.}  We feel that this ``obliquely'' has to do with the many different
ways one finds to define the polar of a function in the literature, especially
by restricting the definitions to special classes of functions.  In this paper,
we propose a unified view of the polarity of functions,
that encompasses all specific definitions, generalizes several well-known properties
and provides new results, especially an isomorphism between the lattices of bipolar sets
and of bipolar functions. 

Given a pair $(\PRIMAL, \DUAL)$ of (real) vector spaces equipped with a bilinear
functional \( \nscal{\,}{} \colon \PRIMAL \times \DUAL \to \RR \), the polar operation
is defined without ambiguity over subsets (see
\cite[Section~14]{Rockafellar:1970}, \cite[Chapter~11,
Sect.~E]{Rockafellar-Wets:1998}, \cite[\S~5.16]{Aliprantis-Border:2006}).  The
(negative, or one-sided) {polar set} of a subset~$\Primal\subset\PRIMAL$ is the closed
convex set
\( \PolarSet{\Primal}= \defset{ \dual \in \DUAL }{ \nscal{\primal}{\dual} \leq 1 \eqsepv \forall
  \primal \in \Primal } \).  The situation is not as clear-cut for the polar operation
defined over functions. We present the main approaches now.

The polar of a function~$\fonctionprimal \colon \RR^{\spacedim} \to \barRR$ is
defined in \cite[Section~15]{Rockafellar:1970}, entitled \emph{Polars of convex
  functions}, as follows: first, for nonnegative positively $1$-homogeneous
convex 
functions vanishing at the origin (so-called gauges)
in~\cite[p.~128]{Rockafellar:1970} by the formula
\( \Polarity{\fonctionprimal}\np{\dual}= \inf \bset{\lambda \ge 0}{\nscal{\primal}{\dual} \leq \lambda
  \fonctionprimal\np{\primal}, \forall \primal\in\RR^{\spacedim}} \) and, second, extended to
nonnegative convex functions vanishing at the origin
in~\cite[p.~136]{Rockafellar:1970} by the formula
\( \Polarity{\fonctionprimal}\np{\dual}= \inf \big\{ \lambda \ge 0 \,\vert\,
\nscal{\primal}{\dual} \leq 1+\lambda \fonctionprimal\np{\primal},\forall \primal\in\RR^{\spacedim} \big\}
\).  It is shown that the two
definitions are equivalent for gauges, and that, if
$\fonctionprimal \colon \RR^{\spacedim} \to \barRR$ is nonnegative convex and
vanishes at the origin, then the bipolar~$\biPolarity{\fonctionprimal}$ is the
greatest nonnegative lsc (lower semicontinuous) convex function vanishing at the origin which is
majorized by~$\fonctionprimal$
($\biPolarity{\fonctionprimal} = \mathrm{cl} \fonctionprimal$).
Other formulas are related to polarity, like
the transform~${\cal A}$ for nonnegative convex lsc
functions vanishing at zero \cite[Equation~(2) and
after]{Artstein-Avidan-Milman:2011}, with expression
\( \np{{\cal A}\fonctionprimal}\np{\dual}= \sup_{\primal\in\RR^{\spacedim}}
\bp{\nscal{\primal}{\dual}-1} / {\fonctionprimal\np{\primal}} \), and the perspective-polar
transform, for closed proper convex nonnegative functions
\cite[Equation~(4.1)]{Aravkin-Burke-Drusvyatskiy-Friedlander-MacPhee:2018}
(see other examples in \cite[Remark~5.1]{Martinez-Legaz-Singer:1994}). 
In the formulas defining polar functions, and in derived formulas, one is often
embarrassed with the treatment of $0$ and $+\infty$ values, especially in fraction
terms. For instance, the expression \( \Polarity{\fonctionprimal}\np{\dual}=
\sup_{\primal\neq 0}
\nscal{\primal}{\dual} / {\fonctionprimal\np{\primal}} \)
in~\cite[p.~128]{Rockafellar:1970} is valid when $\fonctionprimal$ is finite everywhere
and positive except at the origin.
By contrast, there is no problem with the treatment of $0$ and $+\infty$ values
with the way that the polar operation is defined on
nonnegative functions in~\cite[\S~4]{Martinez-Legaz-Singer:1994} by the formula
\( \Polarity{\fonctionprimal}\np{\dual} = \sup_{\primal \in \PRIMAL } \Bp{
  {\nscal{\primal}{\dual}}_+ \LowTimes \bpConverse{\fonctionprimal\np{\primal} } } \) (we
will explain the~$\LowTimes$ later).  To our knowledge, this is the most general
formula as it is defined for any nonnegative function (without restriction to
positively $1$-homogeneous, or convex or vanishing at zero).
\medskip

In this paper, we propose a unified view of the polarity of functions by
proposing a definition of the polar of any function, and then revisit the polar
operation on functions as one would do following the tracks of the Fenchel
transform.  Indeed, the Fenchel transform is defined for any function, and then
closed convex lsc functions (we follow the terminology
in~\cite[p.~15]{Rockafellar:1974}) appear as the class of functions that are
equal to their biconjugate (bi-Fenchel transform).
We do the same by defining the polar transfom for any function,
and then bipolar functions will be defined as those equal to their bipolar
transform. 
\medskip

The paper is organized as follows.
In~Section~\ref{Background_on_set_polarity_and_on_Minkowski_functional},
we present set polarity and the Minkowski functional.
We also introduce \BipolarSet s and show that they form a lattice.
In~Section~\ref{sec:Polar_operation_on_functions},
we develop the polar operation on functions.
First, 
in~\S\ref{Polar_operation_on_nonnegative_functions},
we follow the approach taken in~\cite[\S4]{Martinez-Legaz-Singer:1994}:
we recall the definition of the polar of any nonnegative function
(the embarrassement with $0$ and $+\infty$ values is systematically
handled by means of lower~$ \LowTimes$ and upper~$\UppTimes$ multiplications);
we recall that the polar operation is a $\times$-duality as
in~\cite[\S~4]{Martinez-Legaz-Singer:1994};
we show that the polar of any nonnegative function is a support function
(of the polar set of the 0-level set of the Fenchel conjugate).
Second,
in~\S\ref{Polar_operation_on_functions}, we provide a definition of the polar of any function,
we present several properties and 
we show that the polar of any function is a Minkowski functional
(of the 0-level set of the Fenchel conjugate).
In~Section~\ref{Generalized_circ-convex_functions}, we present equivalent
expressions of the set of bipolar functions, that is,
those equal to their bipolar.
We show that the lattices of bipolar sets
and of bipolar functions are isomorphic. 
In~Section~\ref{Subdifferential_of_the_circ-polarity}, we add
to~\cite[\S~4]{Martinez-Legaz-Singer:1994} by exploring three possible notions
of polar subdifferential associated with a nonnegative function.
We make the connection with the notion of alignement of vectors. 
%
%
In Appendix~\ref{app:best_cvx_general}, we study best lsc convex lower
approximations of a function. 
In Appendix~\ref{Background_on_*-dualities}, we provide 
background on *-dualities, as defined and studied in
\cite{Martinez-Legaz-Singer:1994}.

\section{Set polarity and the Minkowski functional}
\label{Background_on_set_polarity_and_on_Minkowski_functional} 

In~\S\ref{Definitions_and_background_on_functions},
we recall notions related to functions\footnote{%
  Adopting usage in mathematics, we follow Serge Lang and use ``function'' only to
  refer to mappings in which the codomain is a set of numbers (i.e. a subset
  of~$\RR$ or $\CC$, or their possible extensions with $\pm \infty$),
  and reserve the term mapping for more general
  codomains.} and to the Fenchel conjugacy.
In~\S\ref{Background_on_set_polarity},
we provide background on set polarity, and we also introduce the notion of \BipolarSet,
and show that \BipolarSet s form a lattice.
In~\S\ref{Background_on_Minkowski_functional}, 
we provide background on $1$-homogeneous functions and on the Minkowski functional.

\subsection{Background on functions}
\label{Definitions_and_background_on_functions}

We denote $\barRR = \ClosedIntervalClosed{-\infty}{+\infty} $,
$\RR_{+} = \ClosedIntervalOpen{0}{+\infty} $,
$\RR_{++}=\OpenIntervalOpen{0}{+\infty}$.
The \emph{positive part}~$z_+$ of a real number~$z$ is \( z_+=\max\na{z,0} \).

For any set~$\UNCERTAIN$ and any 
function~$\fonctionuncertain \colon \UNCERTAIN \to \barRR$,
we introduce different possible notations for 
\begin{subequations}
  \begin{align}
    \text{the \emph{level sets}}\qquad 
    \MidLevelSet{\fonctionuncertain}{r}
    =  \LevelSetUp{\fonctionuncertain}{r} 
    &= 
      \defset{ \uncertain \in \UNCERTAIN }{ \fonctionuncertain\np{\uncertain} \leq r}
      \eqsepv \forall r \in \barRR
      \eqfinv
      \label{eq:fonctionuncertain_level_set}
    \\
    \text{the \emph{strict level sets}}\qquad 
    \MidStrictLevelSet{\fonctionuncertain}{r} 
    =  \StrictLevelSetUp{\fonctionuncertain}{r} 
    &= 
      \defset{ \uncertain \in \UNCERTAIN }{ \fonctionuncertain\np{\uncertain} < r}
      \eqsepv \forall r \in \barRR
      \eqfinv
      \label{eq:fonctionuncertain_strict_level_set}
    \\
    \text{the \emph{level curves}}\qquad 
    \MidLevelCurve{\fonctionuncertain}{r}
    =  \LevelCurveUp{\fonctionuncertain}{r} 
    &= 
      \defset{ \uncertain \in \UNCERTAIN }{ \fonctionuncertain\np{\uncertain} = r }
      \eqsepv \forall r \in \barRR
      \eqfinv
      \label{eq:fonctionuncertain_level_curve}
    \\
    \text{and also}\qquad 
    \ba{s < \fonctionuncertain < r}
    &= 
      \defset{ \uncertain \in \UNCERTAIN }{s < \fonctionuncertain\np{\uncertain} < r }
      \eqsepv \forall s, r \in \barRR
      \eqfinp       
  \end{align}
\end{subequations}

For any function \( \fonctionuncertain \colon \UNCERTAIN \to \barRR \), its
\emph{epigraph} is
\( \epigraph\fonctionuncertain= \defset{
  \np{\uncertain,t}\in\UNCERTAIN\times\RR}%
{\fonctionuncertain\np{\uncertain} \leq t} \), its \emph{strict epigraph} is
\( \epigraph_{s}\fonctionuncertain= \defset{
  \np{\uncertain,t}\in\UNCERTAIN\times\RR}%
{\fonctionuncertain\np{\uncertain} < t} \), its \emph{effective domain} is
\( \dom\fonctionuncertain= \defset{\uncertain\in\UNCERTAIN}{
  \fonctionuncertain\np{\uncertain} <+\infty} \).  A function
\( \fonctionuncertain \colon \UNCERTAIN \to \barRR \) is said to be \emph{convex}
if its epigraph is a convex set, \emph{proper} if it never takes the
value~$-\infty$ and that \( \dom\fonctionuncertain \not = \emptyset \), \emph{lower semi
  continuous (lsc)} if its epigraph is closed.

For any set~$\UNCERTAIN$ and subset \( \Uncertain \subset \UNCERTAIN \), we denote by
$\Indicator{\Uncertain} \colon \UNCERTAIN \to \barRR $ the \emph{indicator
  function} of the set~$\Uncertain$, defined by
\( \Indicator{\Uncertain}\np{\uncertain} = 0 \) if
\( \uncertain \in \Uncertain \), and
\( \Indicator{\Uncertain}\np{\uncertain} = +\infty \) if
\( \uncertain \not\in \Uncertain \).
We denote by $\chi_{\Uncertain} \colon \UNCERTAIN \to \barRR_{+} $ the
\emph{generalized indicator function} of the set~$\Uncertain$
\cite[Definition~2.2]{Martinez-Legaz-Singer:1994}, defined by
\( \Characteristic{\Uncertain}\np{\uncertain} = 1 \) if
\( \uncertain \in \Uncertain \), and
\( \Characteristic{\Uncertain}\np{\uncertain} = +\infty \) if
\( \uncertain \not\in \Uncertain \).  Thus, we have that
\( \Characteristic{\Uncertain}=\Indicator{\Uncertain}+1 \).

\subsection{Set polarity}
\label{Background_on_set_polarity}

In~\S\ref{Dual_pair,_paired_vector_spaces}, we recall the notions of dual pair
and paired vector spaces.
In~\S\ref{Polar_of_a_set,_bipolar_set}, we recall the
definition of the polar of a set, and we introduce the notion of \BipolarSet.
In~\S\ref{The_lattice_of_BipolarSet_s}, we show that \BipolarSet s form a
lattice.

\subsubsection{Dual pair, paired vector spaces, Fenchel conjugacy}
\label{Dual_pair,_paired_vector_spaces}

We refer the reader to \cite{Aliprantis-Border:2006} and~\cite{Rockafellar:1974}
for the following backgrounds.  We consider a pair $(\PRIMAL, \DUAL)$ of (real)
vector spaces equipped with a bilinear functional
\( \nscal{\,}{} \colon \PRIMAL \times \DUAL \to \RR \). 
Following~\cite[p.~13]{Rockafellar:1974}, we say that
$\PRIMAL$ and $\DUAL$ are \emph{paired spaces}, when they have been equipped with
topologies that are {compatible} with respect to the pairing
(hence Hausdorff and locally convex topologies).
Details on how to generate consistent (or compatible) topologies 
\cite[Definition~5.96]{Aliprantis-Border:2006} from a pairing is developed
in~\cite[\S~5.15]{Aliprantis-Border:2006}.  More precisely, $\PRIMAL$ and $\DUAL$ is
called a \emph{dual pair} in~\cite[Definition 5.90]{Aliprantis-Border:2006} when
the bilinear functional \( \nscal{\,}{}\) separates the points of $\PRIMAL$ and
$\DUAL$. Then, it is proved
in~\cite[Theorem~5.93]{Aliprantis-Border:2006} that, from a dual pair $\PRIMAL$ and $\DUAL$,
one makes paired spaces
when $\PRIMAL$ (resp. $\DUAL$) is equipped with the weak topology
$\sigma(\PRIMAL,\DUAL)$ (resp. with the weak$*$ topology
$\sigma(\DUAL,\PRIMAL)$).
%
\medskip

Now, we review concepts and 
notations related to the Fenchel conjugacy
(we refer the reader to \cite[Sect.~3]{Rockafellar:1974}).
We consider $\PRIMAL$ and $\DUAL$ two paired vector spaces.
For any functions \( \fonctionprimal \colon \PRIMAL  \to \barRR \)
and \( \fonctiondual \colon \DUAL \to \barRR \), 
the different conjugates are defined by\footnote{%
  In convex analysis, one does not use~\( \LFMr{} \) and~\( \LFMbi{} \), 
  but simply~\( \LFM{} \) and~\( ^{\Fenchelcoupling\Fenchelcoupling} \).
  We use~\( \LFMbi{} \) to be consistent with the notation for general
  conjugacies.
\label{ft:to_be_consistent_with_the_notation_for_general_conjugacies}}
\begin{subequations}
  \begin{align}
    \LFM{\fonctionprimal}\np{\dual} 
    &= 
      \sup_{\primal \in \PRIMAL} \bp{ \nscal{\primal}{\dual}
      -\fonctionprimal\np{\primal} }
      \eqsepv \forall \dual \in \DUAL
      \eqfinv
      \label{eq:Fenchel_conjugate}
    \\
    \LFMr{\fonctiondual}\np{\primal} 
    &= 
      \sup_{ \dual \in \DUAL} \bp{ \nscal{\primal}{\dual} 
      -\fonctiondual\np{\dual} } 
      \eqsepv \forall \primal \in \PRIMAL
      \eqfinv
      \label{eq:Fenchel_conjugate_reverse}
    \\
    \LFMbi{\fonctionprimal}\np{\primal} 
    &= 
      \sup_{\dual \in \DUAL} \bp{ \nscal{\primal}{\dual} 
      -\LFM{\fonctionprimal}\np{\dual} }
      \eqsepv \forall \primal \in \PRIMAL
      \eqfinp
      \label{eq:Fenchel_biconjugate}
      \intertext{We also recall the classic \emph{(Rockafellar-Moreau) subdifferential}}
      \Lowsubdifferential{\Fenchelcoupling}{\fonctionprimal}\np{\primal}
    &=
      \defset{ \dual \in \DUAL }{ %
      \LFM{\fonctionprimal}\np{\dual} 
      = \nscal{\primal}{\dual} -\fonctionprimal\np{\primal} }
       \eqsepv
      \forall\primal\in\dom\fonctionprimal 
  \eqfinp
  \label{eq:Rockafellar-Moreau-subdifferential_a}
  \end{align}
\end{subequations}
A function \( \fonctionprimal : \PRIMAL \to \barRR \),
or \( \fonctiondual \colon \DUAL \to \barRR \), 
is said to be \emph{closed}\footnote{%
  We follow the terminology in  \cite[p.~15]{Rockafellar:1974}, although it can be misleading.
  Indeed, anticipating on the notion of valley
  function~\cite{Penot:2000},
  a function taking the value $-\infty$ on a closed subset
  (neither the empty set nor the whole set) and $+\infty$ 
  outside is lsc but not closed.
  Some authors \cite{Borwein-Lewis:2006} use closed in the sense of lsc.
  \label{ft:closed_function}}
if it is either lsc and nowhere having the value $-\infty$,
or is the constant function~$-\infty$.
Closed convex functions are the two constant functions~$-\infty$ and~$+\infty$
united with all proper convex lsc functions.\footnote{%
  In particular, any closed convex function that takes at least one finite value
  is necessarily proper convex~lsc.
  Notice that a function taking the value $-\infty$ on a closed convex subset
  (neither the empty set nor the whole set) and $+\infty$
  outside is convex~lsc, but is not closed convex (see Footnote~\ref{ft:closed_function}).   
  \label{ft:closed_convex_function}}
It is proved that the Fenchel conjugacy --- 
indifferently 
\( \fonctionprimal \mapsto \LFM{\fonctionprimal} \)
or
\( \fonctiondual \mapsto \LFMr{\fonctiondual} \) --- 
induces a one-to-one correspondence
between the closed convex functions and themselves
\cite[Theorem~5]{Rockafellar:1974}.

\subsubsection{Support function, polar of a set, \BipolarSet}
\label{Polar_of_a_set,_bipolar_set}

For any subset \( \Primal\subset\PRIMAL \), we denote by $\convexhull \Primal$ (or
$\convexhull\np{\Primal}$) the \emph{convex hull} of~$\Primal$ --- that is, the
smallest convex set in~$\PRIMAL$ containing~$\Primal$ --- by
$\closedconvexhull \Primal$ (or $\closedconvexhull\np{\Primal}$) the
\emph{closed convex hull} of~$\Primal$ --- that is, the smallest closed convex set
in~$\PRIMAL$ containing~$\Primal$.
A subset $\Cone\subset\PRIMAL$ is said to be a \emph{cone}\footnote{%
  Hence, a cone does not necessarily contain the origin~$0$.}  if
$\RR_{++}\Cone \subset \Cone$.

For any subset\footnote{%
  We use the letter~$\Primal$ for a primal subset, 
  and the letter~$\primal$ for a primal vector.
  We use the letter~$\Dual$ for a dual  subset, 
  and the letter~$\dual$ for a dual vector.
}
\( \Dual\subset\DUAL \),
\( \SupportFunction{\Dual} \colon \PRIMAL \to \barRR\) denotes the 
\emph{support function of the set~$\Dual$} --- 
defined by \( \SupportFunction{\Dual}\np{\primal} = 
\sup_{\dual\in\Dual} \nscal{\primal}{\dual} \), for any \( \primal \in \PRIMAL
\) \cite[\S~7.10, p.~288]{Aliprantis-Border:2006}.
In the same way, we define \( \SupportFunction{\Primal} \colon \DUAL \to \barRR\)
for any subset~$\Primal\subset\PRIMAL$.

The (negative) \emph{polar cone}~$\PolarCone{\Primal}$ of the subset \( \Primal\subset\PRIMAL \) 
is the closed convex cone
\cite[p.~122, Equation~(6.28)]{Bauschke-Combettes:2017}
\begin{subequations}
  \begin{equation}
    \PolarCone{\Primal}=
    \defset{ \dual \in \DUAL }{  \nscal{\primal}{\dual} \leq 0
      \eqsepv \forall \primal \in \Primal }
    =  \MidLevelSet{\SupportFunction{\Primal}}{0}
    \eqfinv
    \label{eq:(negative)polar_cone}
  \end{equation}
  and the same definition holds for \( \Dual\subset\DUAL \),
  so that we define
  \begin{equation}
    \biPolarCone{\Primal}= \PolarCone{\np{\PolarCone{\Primal}}}
    \eqfinp
    \label{eq:(negative)bipolar_cone}
  \end{equation}
\end{subequations}
The (negative) or (one-sided) \emph{polar set}~$\PolarSet{\Primal}$ 
of the subset~$\Primal\subset\PRIMAL$ is the closed convex set
\cite[Definition~5.101, p.~216]{Aliprantis-Border:2006}
\begin{subequations}
  \begin{equation}
    \PolarSet{\Primal}=
    \defset{ \dual \in \DUAL }{  \nscal{\primal}{\dual} \leq 1
      \eqsepv \forall \primal \in \Primal }
    =  \MidLevelSet{\SupportFunction{\Primal}}{1}
    \eqfinv
    \label{eq:(negative)polar_set}
  \end{equation}
  and the same definition holds for \( \Dual\subset\DUAL \).
    For any subset $\Dual \subset \DUAL$, the
  effective domain of the support function~\(\SupportFunction{\Dual} \) is
  $\RR_{++}\PolarSet{\Dual}$.

We define the (negative) or (one-sided) \emph{bipolar set}~$\PolarSet{\Primal}$ 
of the subset~$\Primal\subset\PRIMAL$ as the closed convex set
  \begin{equation}
    \biPolarSet{\Primal}= \PolarSet{\np{\PolarSet{\Primal}}}
    \eqfinv
    \label{eq:(negative)bipolar_set}
  \end{equation}
\end{subequations}
and the same definition holds for \( \Dual\subset\DUAL \).
By the bipolar Theorem~\cite[Theorem~5.103]{Aliprantis-Border:2006},
we have that 
\begin{equation}
  \biPolarSet{\Primal}=
  \closedconvexhull\np{\Primal \cup \na{0} }
  \eqfinp
  \label{eq:biPolarSet}
\end{equation}

We have not found the following definition of \BipolarSet\ in the literature, as
most authors simply say that a set is closed convex and contains~$0$.  However,
putting a name on this well-known notion will be quite practical for our
purposes, especially for the connection with bipolar functions.  The three
equivalences below are a straightforward consequence of the bipolar Theorem.
The notion of polar pair can be found in \cite[Theorem~14.6,
p.~126]{Rockafellar:1970}.

\begin{definition}
  \label{de:bipolar_set_dual_pair}
  A subset \( \Primal\subset\PRIMAL \) is said to be a \emph{\BipolarSet} if
  any of the following three equivalent conditions is satisfied:
  \begin{enumerate}
  \item
    \label{it:bipolar_set_def_1}
    \( \Primal\)  is closed convex and contains~$0$, 
  \item
    \( \Primal= \biPolarSet{\Primal} \), that is, 
    \( \Primal\)  is equal to its bipolar,
  \item
    there exists a nonempty set~\( Z\subset\PRIMAL \) such that 
    \( \Primal= \biPolarSet{Z} \), that is, 
    \( \Primal\)  is equal to the bipolar of~$Z$.
  \end{enumerate}
  The same definition holds for a subset \( \Dual\subset\DUAL \).
  We denote by $\BIPOLARSET$ the set of \BipolarSet s of~$\PRIMAL$.
  \medskip

  Let \( \Primal\subset\PRIMAL \) and \( \Dual\subset\DUAL \) be two subsets.
  We say that \( \Primal \) and \( \Dual\) form a \emph{polar pair}
  if       \( \Dual=\PolarSet{\Primal} \) and \( \Primal=\PolarSet{\Dual} \).
  The two elements of a {polar pair} are necessarily bipolar sets.
\end{definition}
As an example, given a subset \( \Primal\subset\PRIMAL \), we easily check that
$\PolarSet{\Primal}$ is a \BipolarSet\ --- as
\( \triPolarSet{\Primal} = \PolarSet{\Primal} \) follows
from~\eqref{eq:biPolarSet} --- and that \( \biPolarSet{\Primal} \) and
\( \PolarSet{\Primal} \) form a polar pair.

We will need the following properties.
\begin{proposition}
  Let \( \Primal, \Primal' \subset\PRIMAL \) be two (primal) subsets
  (or be two (dual) subsets of~$\DUAL$).
  We have that
  \begin{subequations}
    \begin{align}
      &  \PolarSet{\np{\Primal\cup \Primal'}} =   \PolarSet{\Primal} \cap \PolarSet{\Primal'}
        \eqfinv
        \label{eq:polar_of_union}
      \\
      \Primal, \Primal' \mtext{ \BipolarSet s~} \implies
      &
        \PolarSet{\np{\Primal \cap \Primal'}} =
        \closedconvexhull{\np{\PolarSet{\Primal} \cup \PolarSet{\Primal'}}}
        \label{eq:polar_of_intersection}
        \eqfinp
    \end{align}
  \end{subequations}
\end{proposition}

\begin{proof}
  Equation~\eqref{eq:polar_of_union} easily follows from the
  definition~\eqref{eq:(negative)polar_set} of a (negative) or (one-sided) polar set
  (see also \cite[Item~3, Lemma~5.102]{Aliprantis-Border:2006}). 

  Suppose that \( \Primal, \Primal' \) are both \BipolarSet s. Then, 
  Equation~\eqref{eq:polar_of_intersection} follows from
  \begin{align*}
    \PolarSet{\np{\Primal \cap \Primal'}}
    &=
      \PolarSet{\np{\biPolarSet{\Primal} \cap \biPolarSet{\Primal'}}}
      \tag{as both $\Primal$ and $\Primal'$ are \BipolarSet s}
    \\
    &=
      \biPolarSet{\np{\PolarSet{\Primal} \cup \PolarSet{\Primal'}}}
      \tag{by~\eqref{eq:polar_of_union}}
    \\
    &=
      \closedconvexhull{\np{\PolarSet{\Primal} \cup \PolarSet{\Primal'}}}
      \tag{by the bipolar Theorem in~\eqref{eq:biPolarSet}, as $\PolarSet{\Primal} \cup \PolarSet{\Primal'}$ contains $0$}
      \eqfinp
  \end{align*}
  This ends the proof.
\end{proof}

\subsubsection{The lattice of \BipolarSet s}
\label{The_lattice_of_BipolarSet_s}

The following Proposition~\ref{pr:BipolarSet_lattice} is easy to show.
To our knowledge, it is new (see \cite[p.~291-292]{Aliprantis-Border:2006} that points out that the set of
closed convex subsets is a lattice).
\begin{proposition}
  \label{pr:BipolarSet_lattice}  
  The set of \BipolarSet s of~$\PRIMAL$, when ordered by inclusion~$\subset$, is a lattice
  \( \np{\BIPOLARSET,\wedge,\vee}\)
  with bottom~$\na{0}$ and with top~$\PRIMAL$.
  The greatest lower bound~$\vee$ and the least upper bound~$\wedge$
  operations are given, for any family \( \sequence{\Primal_j}{j\in J} \) of
  \BipolarSet s of~$\PRIMAL$, by
  \begin{subequations}
    \begin{align}
      \Bwedge_{j\in J} \Primal_j
      &=
        \bigcap_{j\in J} \Primal_j
        \eqfinv
        \label{eq:lattice_set_wedge}
      \\
      \Bvee_{j\in J} \Primal_j
      &=
        \closedconvexhull \bp{ \bigcup_{j\in J} \Primal_j }
        \label{eq:lattice_set_vee}
        \eqfinp
    \end{align}
  \end{subequations}
\end{proposition}

\begin{proof}
  First, we prove~\eqref{eq:lattice_set_wedge}.  Let~$\Primal$ be a \BipolarSet\
  such that $\Primal \subset \Primal_j$ for all ${j\in J}$ (the set~$\na{0}$ is
  always a possibility).  Then it is immediate that
  $\Primal \subset \bigcap_{j\in J} \Primal_j$ and~\eqref{eq:lattice_set_wedge} follows as
  $\bigcap_{j\in J} \Primal_j$ is a \BipolarSet, being closed convex and
  containing~$0$ as the intersection of closed convex sets containing~$0$
  (Item~\ref{it:bipolar_set_def_1} of Definition~\ref{de:bipolar_set_dual_pair}).
  
  Second, we prove~\eqref{eq:lattice_set_vee}.  Let~$\Primal$ be a \BipolarSet\
  such that $\Primal_j \subset \Primal$ for all ${j\in J}$ (the set~$\PRIMAL$ is
  always a possibility).  Then we have
  $\bigcup_{j\in J} \Primal_j \subset \Primal$ and, as $\Primal$ is closed convex (by
  Item~\ref{it:bipolar_set_def_1} of Definition~\ref{de:bipolar_set_dual_pair}),
  we obtain that
  $\bigcup_{j\in J} \Primal_j \subset \closedconvexhull \bp{ \bigcup_{j\in J} \Primal_j } \subset \Primal$
  by the very definition of the closed convex hull.  Now,
  $\closedconvexhull \bp{ \bigcup_{j\in J} \Primal_j }$ is closed convex and
  contains~$0$ as all $\Primal_j$ contain~$0$.  Using
  Item~\ref{it:bipolar_set_def_1} of Definition~\ref{de:bipolar_set_dual_pair}, we
  immediately obtain that $\closedconvexhull \bp{ \bigcup_{j\in J} \Primal_j }$ is a
  \BipolarSet\ and~\eqref{eq:lattice_set_vee} follows.

  Obviously, $\na{0}$ is the bottom and $\PRIMAL$ is the top of the lattice. 
\end{proof}

\subsection{Background on the Minkowski functional}
\label{Background_on_Minkowski_functional}

We define $1$-homogeneous functions, present the Minkowski functional
and some of their properties. 


%
%

\begin{definition}
  \label{de:1-homogeneous_function}
  Let $\PRIMAL$ be a (real) vector space, and $\Cone\subset\PRIMAL$ be a nonempty cone.
  We say that a
  {function}~$\fonctionprimal \colon \Cone \to {\barRR}$ is
  \emph{(strictly positively)\footnote{%
      The definition of homogeneous function is not stabilized in the literature.
      For instance, in \cite[\S~5.8, p.~190]{Aliprantis-Border:2006}, a real
      function deﬁned on a cone is positively homogeneous if
      Equation~\eqref{eq:1-homogeneous_function} holds for all
      $\lambda \in \RR_{+}$ (thus including $\lambda=0$);
      in \cite[p.~30]{Rockafellar:1970}, a function
      on~$\RR^\spacedim$ is positively homogeneous if
      Equation~\eqref{eq:1-homogeneous_function} holds for all
      $\lambda \in \RR_{++}$ (thus excluding $\lambda=0$).
      This is why, we prefer to avoid all ambiguity and speak of
      (strictly positively) when \( \lambda \in {\RR_{++}} \).  We may sometimes omit the
      (strictly positively) and speak of a $1$-homogeneous function.}
    $1$-homogeneous}, or homogeneous of degree~$1$, (on the cone~$\Cone$) if
  \begin{equation}
    {\fonctionprimal\np{\lambda\primal} = \lambda \fonctionprimal\np{\primal}}
    \eqsepv \forall \lambda \in {\RR_{++}}
    \eqsepv \forall \primal \in \Cone
    \eqfinp
    \label{eq:1-homogeneous_function}
  \end{equation}
\end{definition}

Following \cite[Section~15, p.~130]{Rockafellar:1970}, and as recalled in
Sect.~\ref{Introduction}, a
function~$\fonctionprimal \colon \PRIMAL \to \barRR_{+}$ is said to be a
\emph{gauge} if it is nonnegative (strictly positively) $1$-homogeneous convex
and vanishing at zero (\( \fonctionprimal\np{0}=0 \)).
By \cite[Theorem~7.51, p.~288]{Aliprantis-Border:2006} and
\cite[Definition~5.45, p.~190]{Aliprantis-Border:2006}, the support function of
a nonempty set is convex, lsc, (strictly positively) $1$-homogeneous and
vanishes at the origin, hence is a lsc gauge.

For a nonnegative function, we will need the following result (although it is
well-known that a {function} is (strictly positively) $1$-homogeneous if and
only if its epigraph is a cone if and only if its strict epigraph is a cone, we
give a proof and an explicit expression of the cone in
Item~\ref{it:strict_epigraph_of_a_1-homogeneous_function}).

\begin{proposition}
  Let $\Cone\subset\PRIMAL$ be a nonempty cone and
  \( \fonctionprimal \colon \Cone \to {\barRR}_{+} \) be a function.  The
  following statements are equivalent:
  \begin{enumerate}
  \item
    \label{it:1-homogeneous_function}
    the function \( \fonctionprimal \colon \Cone \to {\barRR}_{+} \)
    is (strictly positively) $1$-homogeneous,
  \item
    \label{it:strict_epigraph_of_a_1-homogeneous_function}
    the  strict epigraph of~$\fonctionprimal$ has the expression
    \( \epigraph_{s}\fonctionprimal= \RR_{++}
    \bp{ \np{\Cone \cap \MidStrictLevelSet{\fonctionprimal}{1}}\times\na{1}
    } \),
  \item
    \label{it:strict_epigraph_of_a_1-homogeneous_function_is_a_cone}
    the  strict epigraph of~$\fonctionprimal$ is a cone included
    in $\Cone\times \RR_{++}$.
  \end{enumerate}
  \label{pr:1-homogeneous_function_strict_epigraph}
\end{proposition}

\begin{proof} We have
  \begin{enumerate}
  \item
    We prove that Item~\ref{it:1-homogeneous_function}
    implies Item~\ref{it:strict_epigraph_of_a_1-homogeneous_function}.
    We have that
    \begin{align*}
      (\primal,\alpha) \in \epigraph_{s}\fonctionprimal
      &\iff
        \primal\in\Cone \text{ and }
        \fonctionprimal\np{\primal} < \alpha \text{ and } \alpha >0
        \intertext{by definition of the strict epigraph of~$\fonctionprimal$,
        and using the assumption that
        \( \fonctionprimal \geq 0 \), hence that $\alpha >0$,}
      &\iff
        \primal\in\Cone \text{ and }
        \frac{1}{\alpha}\fonctionprimal\np{\primal} =\fonctionprimal\np{\frac{\primal}{\alpha}}
        < 1 \text{ and } \alpha > 0
        \tag{by the assumption that the function \( \fonctionprimal \)
        is (strictly positively) $1$-homogeneous in Item~\ref{it:1-homogeneous_function}}
        \eqfinv
      \\
      &\iff
        \frac{\primal}{\alpha} \in \Cone\cap\MidStrictLevelSet{\fonctionprimal}{1} \text{ and } \alpha > 0
        \tag{as $\Cone$ is a cone}
        \eqfinv
      \\
      &\iff
        (\frac{\primal}{\alpha},1) \in \np{\Cone\cap\MidStrictLevelSet{\fonctionprimal}{1}} \times\na{1} \text{ and } \alpha > 0
        \eqfinv
      \\
      &\iff
        (\primal,\alpha) \in \RR_{++}\np{\Cone\cap\MidStrictLevelSet{\fonctionprimal}{1}} \times\na{1}
        \eqfinp
    \end{align*}
  \item
    It is straightforward that Item~\ref{it:strict_epigraph_of_a_1-homogeneous_function}
    implies Item~\ref{it:strict_epigraph_of_a_1-homogeneous_function_is_a_cone}.
  \item
    We prove that Item~\ref{it:strict_epigraph_of_a_1-homogeneous_function_is_a_cone}
    implies  Item~\ref{it:1-homogeneous_function}.
    For any $ \primal \in \Cone$ and $\alpha\in \RR_{++}$, we have that
    \begin{align*}
      \fonctionprimal\np{\alpha\primal}
      &=
        \inf_{\np{\alpha\primal,t}\in\epigraph_{s}\fonctionprimal}t
      \\
      &=
        \inf_{\np{\primal,t/\alpha}\in\frac{1}{\alpha}\epigraph_{s}\fonctionprimal}t
      \\
      &=
        \alpha \inf_{\np{\primal,t/\alpha}\in\epigraph_{s}\fonctionprimal}
        \frac{t}{\alpha}
        \tag{as \(
        \frac{1}{\alpha}\epigraph_{s}\fonctionprimal=\epigraph_{s}\fonctionprimal\)
        since this latter is a cone}
      \\
      &=
        \alpha \inf_{\np{\primal,t'}\in\epigraph_{s}\fonctionprimal} t' 
      \\
      &=
        \alpha \fonctionprimal\np{\primal}
        \eqfinp
    \end{align*}
  \end{enumerate}
  This ends the proof.
\end{proof}
  
We follow \cite[Definition~5.48, p.~191]{Aliprantis-Border:2006} to introduce Minkowski
functionals\footnote{%
  Also called gauges in \cite[Definition~5.48]{Aliprantis-Border:2006}.
  The definition of gauge is not stabilized in the literature.  For
  instance, in \cite[Section~15, p.~130]{Rockafellar:1970}, gauges are functions
  of the form
  \( k\np{\primal} = \inf\bset{\lambda \geq 0}{\primal\in\lambda\Primal} \) for some nonempty convex
  set~$\Primal$.
In \cite[p.~4]{Zalinescu:2002}, Minkowski gauges are defined like in~\eqref{eq:MinkowskiFunctional}
but with \( \lambda \geq 0 \) and for an absorbing set~\( \Primal \).}
\begin{definition}
  \label{de:MinkowskiFunctional}
  Let \( \Primal \subset \PRIMAL \). The \emph{Minkowski functional}
  associated with the subset~\( \Primal \subset \PRIMAL \)
  is the function \( \MinkowskiFunctional{\Primal} \colon \PRIMAL \to \barRR_{+} \)
  defined by (with the convention that $\inf\emptyset=+\infty$) 
  \begin{equation}
    \MinkowskiFunctional{\Primal}\np{\primal}
    = \inf\bset{\lambda>    0}{\primal\in\lambda\Primal}
    \eqsepv \forall \primal \in\PRIMAL
    \eqfinp
    \label{eq:MinkowskiFunctional}
  \end{equation}
\end{definition}

We will need the following properties.




\begin{proposition}
  \label{pr:MinkowskiFunctional}
  Let \( \Primal \subset \PRIMAL \).
  \begin{enumerate}
  \item
    \label{it:Minkowski_is_a_Functionalpositively_1-homogeneous_function}  
    The Minkowski functional \( \MinkowskiFunctional{\Primal} \colon \PRIMAL \to \barRR_{+} \) 
    is a nonnegative (strictly positively) $1$-homo\-ge\-ne\-ous function. 
    
  \item
    \label{it:Minkowski_strict_epigraph}
    The strict epigraph~\( \epigraph_{s}\MinkowskiFunctional{\Primal} \)
    of the Minkowski functional~\( \MinkowskiFunctional{\Primal} \) is the cone
    \begin{equation}
      \epigraph_{s}\MinkowskiFunctional{\Primal}
      = 
      \RR_{++}\np{ \MidStrictLevelSet{\MinkowskiFunctional{\Primal}}{1}
        \times \na{1} }
      \eqfinp 
      \label{eq:Minkowski_strict_epigraph}
    \end{equation}
    
  \item
    \label{it:strict_epigraph_is_a_cone_implies_Minkowski}
    For any 
    function \( \fonctionprimal \colon \PRIMAL \to\barRR_{+}\), 
    we have the implication
    \begin{equation}
           \epigraph_{s}\fonctionprimal = \RR_{++}\np{ \Primal \times \na{1} }
        \implies 
        \fonctionprimal= \MinkowskiFunctional{\Primal}
        \eqfinp
      \label{eq:strict_epigraph_is_a_cone_implies_Minkowski}
    \end{equation}
    
  \item
    \label{it:positively_1-homogeneous_function_is_a_MinkowskiFunctional}  
    Conversely, any nonnegative (strictly positively) $1$-homogeneous function
    \( \fonctionprimal \colon \PRIMAL \to\barRR_{+}\) is the Minkowski
    functional~\( \MinkowskiFunctional{\Primal} \) of a subset
    \( \Primal \subset \PRIMAL \), which can be chosen as the 
    subset \( \Primal = \MidStrictLevelSet{\fonctionprimal}{1} \),
    %
    that is,
    \begin{subequations}
      \begin{equation}
        \fonctionprimal = \MinkowskiFunctional{\MidStrictLevelSet{\fonctionprimal}{1}}
        \eqfinp
        \label{eq:positively_1-homogeneous_function_is_a_MinkowskiFunctional_strict}  
      \end{equation}
      When \( \fonctionprimal\np{0}=0 \), we have that
      \begin{equation}
          \fonctionprimal = \MinkowskiFunctional{\MidStrictLevelSet{\fonctionprimal}{1}}
        = \MinkowskiFunctional{\MidLevelSet{\fonctionprimal}{1}} 
        \eqfinp
        \label{eq:positively_1-homogeneous_function_is_a_MinkowskiFunctional_both}  
      \end{equation}
    \end{subequations}
  \item
    \label{it:The_Minkowski_functional_satisfies}
    The Minkowski functional satisfies 
    \begin{subequations}
      \begin{align}
        \MinkowskiFunctional{\bp{\bigcup_{i\in I}\Primal_i}}
        &=
          \inf_{i\in I}\MinkowskiFunctional{\Primal_i}
          \eqfinv 
          \label{eq:MinkowskiFunctional_union}
          \intertext{ for any family $\sequence{\Primal_i}{i\in I}$
          of subsets of~$\PRIMAL$,}
          \domain\np{\MinkowskiFunctional{\Primal}}
          &= 
            \RR_{++} \Primal
            \eqfinv
            \label{eq:MinkowskiFunctional_effective_domain}
        \\
        \Primal  \text{ is a convex set}
          & \implies
            \MinkowskiFunctional{\Primal} \text{ is a convex function} 
            \eqfinp
            \label{eq:MinkowskiFunctional_is_convex}
            %
      \end{align}
    \end{subequations}
  \end{enumerate}
\end{proposition}

\begin{proof}
  \begin{enumerate}
  \item
    Item~\ref{it:Minkowski_is_a_Functionalpositively_1-homogeneous_function} is
    well-known and     easy to prove.

  \item
    Item~\ref{it:Minkowski_strict_epigraph} is implied by
    Item~\ref{it:Minkowski_is_a_Functionalpositively_1-homogeneous_function}
    as a consequence of Proposition~\ref{pr:1-homogeneous_function_strict_epigraph}
    (more precisely, a consequence of the fact that
    Item~\ref{it:1-homogeneous_function} implies
    Item~\ref{it:strict_epigraph_of_a_1-homogeneous_function}
    in Proposition~\ref{pr:1-homogeneous_function_strict_epigraph}).
    
 \item
    We prove
    Item~\ref{it:strict_epigraph_is_a_cone_implies_Minkowski}.
    We consider a subset \( \Primal \subset \PRIMAL \) and a
    function \( \fonctionprimal \colon \PRIMAL \to\barRR_{+}\) such that
    \( \epigraph_{s}\fonctionprimal = \RR_{++}\np{ \Primal \times \na{1} } \).
    For any $ \primal \in \PRIMAL$, we have that
    \begin{align*}
      \fonctionprimal\np{\primal}
      &=
        \inf_{\np{\primal,t}\in\epigraph_{s}\fonctionprimal}t
      \\
      &=
        \inf_{\np{\primal,t}\in\RR_{++}\np{ \Primal \times \na{1} }}t
        \tag{by assumption}
      \\
      &=
        \inf_{t>0, \primal\in t\Primal}t
        \tag{as \( \np{\primal,t}\in\RR_{++}\np{ \Primal \times \na{1} } \iff
        t>0 \) and \( \primal\in t\Primal \)}
      \\
      &=
        \MinkowskiFunctional{\Primal}\np{\primal}
        \tag{by definition~\eqref{eq:MinkowskiFunctional} of the Minkowski functional}
        \eqfinp 
    \end{align*}

  \item
    We prove
    Item~\ref{it:positively_1-homogeneous_function_is_a_MinkowskiFunctional}.
    Let \( \fonctionprimal \colon \PRIMAL \to\barRR_{+}\) be
    a nonnegative (strictly positively) $1$-homogeneous function.
    The implication of
    Item~\ref{it:strict_epigraph_of_a_1-homogeneous_function} by
    Item~\ref{it:1-homogeneous_function}
    in Proposition~\ref{pr:1-homogeneous_function_strict_epigraph}
    gives that (with \( \Cone=\PRIMAL \))
    \[ \epigraph_{s}\fonctionprimal= \RR_{++} \np{
        \MidStrictLevelSet{\fonctionprimal}{1}\times \na{1}}
      \eqfinv
    \]
    hence~\eqref{eq:positively_1-homogeneous_function_is_a_MinkowskiFunctional_strict}
    follows from implication~\eqref{eq:strict_epigraph_is_a_cone_implies_Minkowski},
    proved in Item~\ref{it:strict_epigraph_is_a_cone_implies_Minkowski}
    of this very Proposition~\ref{pr:MinkowskiFunctional}.
    
    Now, the epigraph of the nonnegative (strictly positively) $1$-homogeneous function
    \( \fonctionprimal \colon \PRIMAL \to\barRR_{+}\) is given by
    \begin{equation}
      \epigraph\fonctionprimal
      = \np{\MidLevelCurve{\fonctionprimal}{0}\times \na{0}}
      \cup
      \RR_{++}\np{\MidLevelSet{\fonctionprimal}{1}\times\na{1}}      
      \eqfinp
      \label{eq:epigraph_of_a_nonnegative_(strictly_positively)_1-homogeneous_function}
    \end{equation}
    Indeed, we have that
    \begin{align*}
      (\primal,\alpha) \in \epigraph\fonctionprimal
      &\iff
           \fonctionprimal\np{\primal} \leq \alpha \text{ and }  \alpha \geq 0 
        \tag{by definition of the strict epigraph of~$\fonctionprimal$, and using the assumption that
        \( \fonctionprimal \geq 0 \)}
        \eqfinv
     \\
      &\iff
        \begin{cases}
          \text{either }
          & \alpha=0 \text{ and } \fonctionprimal\np{\primal}=0, \text{ as }
            0 \leq \fonctionprimal\np{\primal} \leq 0
            \eqfinv
          \\
          \text{or }
          & \frac{1}{\alpha}\fonctionprimal\np{\primal} =\fonctionprimal\np{\frac{\primal}{\alpha}}
            \leq 1 \text{ and } \alpha > 0,
            \text{ as the function~$\fonctionprimal$}
          \\
          &  \text{is (strictly positively) $1$-homogeneous by assumption in Item~\ref{it:1-homogeneous_function},}
        \end{cases}
      \\
      &\iff
        \begin{cases}
          \text{either }
          & (\primal,\alpha) \in \np{\MidLevelCurve{\fonctionprimal}{0}\times \na{0}} \eqfinv
          \\
          \text{or }
          &  (\primal,\alpha) \in
            \RR_{++}\np{\MidLevelSet{\fonctionprimal}{1}\times\na{1}} \eqfinv
        \end{cases}
      \\
      &\iff
        (\primal,\alpha) \in \MidLevelCurve{\fonctionprimal}{0}\times \na{0}
        \cup
        \RR_{++}\np{\MidLevelSet{\fonctionprimal}{1}\times\na{1}}      
        \eqfinp 
    \end{align*}
    Thus, we have proved~\eqref{eq:epigraph_of_a_nonnegative_(strictly_positively)_1-homogeneous_function}.
    As a consequence, for any $ \primal \in \PRIMAL$, we get that
    \begin{align*}
      \fonctionprimal\np{\primal}
      &=
        \inf_{\np{\primal,t}\in\epigraph\fonctionprimal}t
        \\
      &=
        \inf\ba{\inf_{\np{\primal,t}\in\MidLevelCurve{\fonctionprimal}{0}\times\na{0}}t,
        \inf_{\np{\primal,t}\in\RR_{++}\np{\MidLevelSet{\fonctionprimal}{1}\times\na{1}}}t}
        \tag{by~\eqref{eq:epigraph_of_a_nonnegative_(strictly_positively)_1-homogeneous_function}}
        \\
      &=
        \inf\ba{\inf_{\np{\primal,t}\in\MidLevelCurve{\fonctionprimal}{0}\times\na{0}}t,
        \MinkowskiFunctional{\MidLevelSet{\fonctionprimal}{1}}\np{\primal}}
        \intertext{by definition~\eqref{eq:MinkowskiFunctional} of the Minkowski
        functional (see the details in the proof of
        Item~\ref{it:strict_epigraph_is_a_cone_implies_Minkowski})}
      &=
        \begin{cases}
          \inf\ba{\inf_{t=0}t,
            \MinkowskiFunctional{\MidLevelSet{\fonctionprimal}{1}}\np{\primal}}
           & \text{if } \primal\in\MidLevelCurve{\fonctionprimal}{0}
            \eqfinv
          \\
          \inf\ba{\inf_{\np{\primal,t}\in\emptyset}t,
            \MinkowskiFunctional{\MidLevelSet{\fonctionprimal}{1}}\np{\primal}}
           & \text{if } \primal\not\in\MidLevelCurve{\fonctionprimal}{0}
            \eqfinv
        \end{cases}
      \\
      &=
        \begin{cases}
          \inf\ba{0,\MinkowskiFunctional{\MidLevelSet{\fonctionprimal}{1}}\np{\primal}}
          = 0 
          & \text{if } \primal\in\MidLevelCurve{\fonctionprimal}{0},
          \text{ as } \MinkowskiFunctional{\MidLevelSet{\fonctionprimal}{1}}\np{\primal}\geq 0
            \eqfinv
          \\
          \inf\ba{+\infty,\MinkowskiFunctional{\MidLevelSet{\fonctionprimal}{1}}\np{\primal}}
          = \MinkowskiFunctional{\MidLevelSet{\fonctionprimal}{1}}\np{\primal}
          & \text{if } \primal\not\in\MidLevelCurve{\fonctionprimal}{0}
            \eqfinp
        \end{cases}
    \end{align*}
    To
    prove~\eqref{eq:positively_1-homogeneous_function_is_a_MinkowskiFunctional_both},
    there remains to show that
    \( \fonctionprimal(\primal)= 0 \implies
    \MinkowskiFunctional{\MidLevelSet{\fonctionprimal}{1}}\np{\primal}=0\).  Now, for
    any $\primal \in\MidLevelCurve{\fonctionprimal}{0}$ and any $\alpha >0$, we have that
    $\primal \in \alpha \MidLevelSet{\fonctionprimal}{1}$ --- because
    $\fonctionprimal\np{\frac{\primal}{\alpha}}= \frac{1}{\alpha} \fonctionprimal(\primal)= 0 \leq 1$ ---
    and we deduce that
    $\MinkowskiFunctional{\MidLevelSet{\fonctionprimal}{1}}\np{\primal}=\inf\RR_{++}=0$.
    We conclude that
    $\fonctionprimal\np{\primal}=0=\MinkowskiFunctional{\MidLevelSet{\fonctionprimal}{1}}\np{\primal}$.
    Thus, we have
    obtained~\eqref{eq:positively_1-homogeneous_function_is_a_MinkowskiFunctional_both}.
        

  \item
    The proofs of Item~\ref{it:The_Minkowski_functional_satisfies} are left to the reader
    (see \cite[Lemmas~5.49,~5.50]{Aliprantis-Border:2006}). 
  \end{enumerate}

  This ends the proof.
\end{proof}




\section{Polar operation on functions}
\label{sec:Polar_operation_on_functions}

In this Sect.~\ref{sec:Polar_operation_on_functions},
we consider $\PRIMAL$ and $\DUAL$ two (real) vector spaces that are {paired}
(see \S\ref{Dual_pair,_paired_vector_spaces}).
In~\S\ref{Polar_operation_on_nonnegative_functions},
we define the polar of any nonnegative function,
and study properties of the polar operation.
In~\S\ref{Polar_operation_on_functions},
we define the polar of any function,
and study properties of the polar operation.
%

\subsection{Polar operation on nonnegative functions}
\label{Polar_operation_on_nonnegative_functions}
%
In~\S\ref{Upper_and_lower_multiplications},
we provide background on upper and lower multiplications.
In~\S\ref{Definition_of_the_polar_of_a_nonnegative_function},
we follow \cite{Martinez-Legaz-Singer:1994} to define the polar (transform) of any
nonnegative function, and
we recall that, thus defined, the polarity operation is a $\times$-duality.
Then, in~\S\ref{Polar_transform_as_a_support_function},
we provide several results about the polar of nonnegative functions,
some well-known (and scattered in the literature) and some new.
Finally, in~\S\ref{Examples_of_polar_transforms_as_support_functions}, we
provide examples of polar transforms as support functions.

\subsubsection{Background on upper and lower multiplications} 
\label{Upper_and_lower_multiplications}

Following Appendix~\ref{Background_on_*-dualities},
we consider $\barRR_{+}= \ClosedIntervalClosed{0}{+\infty} $ as the canonical enlargement
\( \np{\barRR_{+},\leq,\UppTimes,\LowTimes} \)
of the complete totally ordered group
\( \np{\RR_{++},\leq,\times} \) with the two elements~$0$ and~$+\infty$
by \cite[\S~4]{Martinez-Legaz-Singer:1994}, that is,
\( \barRR_{+} = \RR_{++} \cup \na{0}  \cup \na{+\infty} \) 
with order extended by \( 0 \leq \alpha \leq +\infty \), for all \( \alpha \in \barRR_{+} \),
and with \emph{upper multiplication}~$\UppTimes$
and \emph{lower multiplication}~$\LowTimes$ given by
(see \cite[Equations~(14.8)-(14.9)]{Moreau:1966-1967},
\cite[Equations~(1.4)-(1.8)]{Martinez-Legaz-Singer:1994})
\begin{subequations}
  \begin{align}
    \alpha\UppTimes\beta
    &=
      \alpha\LowTimes\beta
      =       \alpha\times\beta
      \eqsepv \forall \alpha,\beta \in \RR_{++} 
      \eqfinv
    \\
    \np{+\infty} \UppTimes \alpha
    &= \alpha \UppTimes \np{+\infty} = +\infty 
      \eqsepv \forall \alpha \in \barRR_{+} 
      \eqfinv
            \label{eq:lower_multiplication_real_numbers}
    \\
    0 \UppTimes \alpha
    &= \alpha \UppTimes 0 = 0 
      \eqsepv \forall \alpha \in \RR_{++} \cup \na{0} 
      \eqfinv
    \\
    \np{+\infty} \LowTimes \alpha
    &= \alpha \LowTimes \np{+\infty} = +\infty 
      \eqsepv \forall \alpha \in \RR_{++} \cup \na{+\infty} 
      \eqfinv
      \label{eq:lower_multiplication_infty}
    \\
    0 \LowTimes \alpha
    &= \alpha \LowTimes 0 = 0 
      \eqsepv \forall \alpha \in \barRR_{+} 
      \eqfinv
      \label{eq:lower_multiplication_0}
      \intertext{and the inverse operation extended as}
      \Converse{0}
    &= +\infty \eqsepv \npConverse{+\infty}=0
      \eqfinp
      \label{eq:and_the_inverse_operation_extended_as}
  \end{align}
  Both upper and lower multiplications are associative and commutative
  \cite[Remark~1.2]{Martinez-Legaz-Singer:1994}, and isotone in the following sense
  \begin{equation}
    \beta\leq\gamma
    \implies
    \alpha\UppTimes\beta \leq
    \alpha\UppTimes\gamma
    \eqsepv
    \alpha\LowTimes\beta \leq
    \alpha\LowTimes\gamma
    \eqfinp
    \label{eq:lower_upper_Times_isotony}
  \end{equation}
\end{subequations}


\subsubsection{Definition of the polar of a nonnegative function}
\label{Definition_of_the_polar_of_a_nonnegative_function}

We follow \cite[\S~4,\S~5~B)]{Martinez-Legaz-Singer:1994} (see the background in
\S\ref{*-duality}) to define the polar (transform) of any nonnegative function.

%

\begin{definition}
  \label{de:circ-polar_transform}
  For any function~$\fonctionprimal \colon \PRIMAL \to \barRR_{+}$,
  the \emph{polar (transform)} 
  \( \Polarity{\fonctionprimal} \colon \DUAL \to \barRR_{+} \)
  of the function~$\fonctionprimal$ is defined by
  \begin{subequations}
    \begin{equation}
      \Polarity{\fonctionprimal}\np{\dual}
      =
      \sup_{\primal \in \PRIMAL } \Bp{ {\nscal{\primal}{\dual}}_+ 
        \LowTimes \bpConverse{\fonctionprimal\np{\primal} } }
      \eqsepv \forall \dual \in \DUAL
      \eqfinp
      \label{eq:circ-polar_transform}
    \end{equation}
    For any function~$\fonctiondual \colon \DUAL \to \barRR_{+}$,
    the \emph{reverse polar transform} 
    \( \PolarityReverse{\fonctiondual} \colon \PRIMAL \to \barRR_{+} \)
    of the function~$\fonctiondual$ is defined by
    \begin{equation}
      \PolarityReverse{\fonctiondual}\np{\primal}
      =
      \sup_{\dual \in \DUAL } \Bp{ {\nscal{\primal}{\dual}}_+ 
        \LowTimes \bpConverse{\fonctiondual\np{\dual} } }
      \eqsepv \forall \primal \in \PRIMAL
      \eqfinp
      \label{eq:reverse_circ-polar_transform}
    \end{equation}
    For any function~$\fonctionprimal \colon \PRIMAL \to \barRR_{+}$,
    the \emph{bipolar transform} 
    \( \Polaritybi{\fonctionprimal} \colon \PRIMAL \to \barRR_{+} \)
    of the function~$\fonctionprimal$ is defined by\footnote{%
      We adopt the notation~$\Polaritybi{\fonctionprimal}$,
      and not~$\biPolarity{\fonctionprimal}$,
      to be consistent with the notation for general
      conjugacies (see also
      Footnote~\ref{ft:to_be_consistent_with_the_notation_for_general_conjugacies}). 
    }
    \begin{equation}
      \Polaritybi{\fonctionprimal}
      =
      \PolarityReverse{\np{\Polarity{\fonctionprimal}}} 
      \eqfinp
      \label{eq:circ-polar_bitransform}
    \end{equation}
  \end{subequations}
\end{definition}


The following Proposition~\ref{pr:the_circ-polarity_is_a_times-duality}
is a direct application of \cite[\S4]{Martinez-Legaz-Singer:1994}.
Surprisingly, the Inequality~\eqref{eq:bipolar_is_smaller} is not
stated in \cite{Martinez-Legaz-Singer:1994}.

\begin{proposition}(\cite[Theorem~4.1]{Martinez-Legaz-Singer:1994})
  \label{pr:the_circ-polarity_is_a_times-duality}
  The polarity mapping 
  \( \barRR_{+}^{\PRIMAL} \ni \fonctionprimal \mapsto
  \Polarity{\fonctionprimal} \in \barRR_{+}^{\DUAL} \) is a $\times$-duality, that
  is, it satisfies \cite[Definition~2.3]{Martinez-Legaz-Singer:1994}:
  \begin{subequations}
    \label{eq:the_circ-polarity_is_a_times-duality}
    \begin{align}
      \Polarity{\np{\inf_{i\in I}\fonctionprimal_i}}
      &=
        \sup_{i\in I}\Polarity{\fonctionprimal_i}
        \eqsepv \forall \sequence{\fonctionprimal_i}{i\in I} \subset \barRR_{+}^{\PRIMAL}
        \eqfinv
           \label{eq:the_circ-polarity_is_a_times-duality_duality}
      \\
      \Polarity{\np{\fonctionprimal\UppTimes\alpha}}
      &=
        \Polarity{\fonctionprimal}
        \LowTimes \Converse{\alpha}
        \eqsepv \forall \alpha \in \barRR_{+}
        \eqsepv \forall \fonctionprimal\in\barRR_{+}^{\PRIMAL}
        \eqfinv
      \\
      \Polaritybi{\fonctionprimal}
      & \leq {\fonctionprimal}
        \eqsepv \forall \fonctionprimal\in\barRR_{+}^{\PRIMAL}
        \eqfinp
        \label{eq:bipolar_is_smaller}
    \end{align}
  \end{subequations}
\end{proposition}  

\begin{proof}
We follow \cite[\S4]{Martinez-Legaz-Singer:1994} and the background in
\S\ref{*-duality}.
  We define the coupling~$\coupling$ between $\PRIMAL$ and $\DUAL$ 
  by \( \coupling\np{\primal,\dual}={\nscal{\primal}{\dual}}_+ \),
  for all \( \primal\in\PRIMAL, \dual \in \DUAL \), and also
  (see \cite[Equation~(4.11)]{Martinez-Legaz-Singer:1994} and Equation~\eqref{eq:*-duality_definition})
  \begin{equation}
    \SFM{\fonctionprimal}{\Duality\np{\coupling}}\np{\dual} = 
    \sup_{\primal \in \PRIMAL} \Bp{ {\nscal{\primal}{\dual}}_+ 
      \LowTimes \bpConverse{\fonctionprimal\np{\primal} } }
    \eqsepv \forall \dual \in \DUAL
    \eqfinp
   \label{eq:times-duality_definition}
  \end{equation}
  By \cite[Theorem~4.1]{Martinez-Legaz-Singer:1994},
  the mapping 
  \( \barRR_{+}^{\PRIMAL} \ni \fonctionprimal \mapsto  \SFM{\fonctionprimal}{\Duality\np{\coupling}}
  \in \barRR_{+}^{\DUAL} \) is a $\times$-duality, that is,
  it satisfies~\eqref{eq:the_circ-polarity_is_a_times-duality}
  (which corresponds to \cite[Equations~(2.7)-(2.8)]{Martinez-Legaz-Singer:1994}).
  The last Inequation~\eqref{eq:bipolar_is_smaller} follows from Equation~\eqref{eq:bipolar-appendix}
  in \S\ref{*-duality}.
\end{proof}

\subsubsection{Polar transform as a support function}
\label{Polar_transform_as_a_support_function}

The following
Proposition~\ref{pr:properties_circ-polarity_nonnegative_functions} gathers 
properties of the polar transform of a nonnegative function.
\begin{enumerate}
\item Item~\ref{it:Polar_inequality} (polar inequality) is not stated in
  \cite{Martinez-Legaz-Singer:1994} (although it can be easily deduced); it is
  established in \cite[p.~130]{Rockafellar:1970}, but only for functions that
  are themselves gauges, and for vectors in the respective domains.
\item
  Item~\ref{it:properties_circ-polarity_nonnegative_functions_support_function}
  --- which expresses the polar transform of a nonnegative
  function~$\fonctionprimal \colon \PRIMAL \to \barRR_{+}$ as the support
  function~$\SupportFunction{\PolarSet{\MidLevelSet{\LFM{\fonctionprimal}}{0}}}$
  of the polar set~${\PolarSet{\MidLevelSet{\LFM{\fonctionprimal}}{0}}}$ of the
  $0$-level set~${\MidLevelSet{\LFM{\fonctionprimal}}{0}}$ of the Fenchel
  conjugate\footnote{%
    The classic Fenchel conjugacy
    \( \fonctionprimal \mapsto \LFM{\fonctionprimal} \) is outlined
    in~\S\ref{Dual_pair,_paired_vector_spaces}.
  }~$\LFM{\fonctionprimal}$ --- is stated neither in
  \cite{Martinez-Legaz-Singer:1994} nor \cite{Rockafellar:1970} (the set
  \( {\MidLevelSet{\LFM{\fonctionprimal}}{0}} \) appears in \cite[Theorem~13.5,
  p.~118]{Rockafellar:1970}, \cite[Theorem~14.3, p.~123]{Rockafellar:1970}).
\item Item~\ref{it:properties_circ-polarity_nonnegative_functions_convex_lsc} is
  not stated in \cite{Martinez-Legaz-Singer:1994} (although it can be deduced
  from the proof of \cite[Theorem~5.2]{Martinez-Legaz-Singer:1994} which,
  however, lacks some details).  It is established in \cite[Theorem~15.1,
  p.~128]{Rockafellar:1970}, but only for functions that are themselves gauges.
\item Item~\ref{it:properties_circ-polarity_nonnegative_functions_infimum} is
  stated and proved in \cite[Corollary~4.1]{Martinez-Legaz-Singer:1994}.
\item
  As for
  Item~\ref{it:properties_circ-polarity_nonnegative_functions_support_function},
  we suspect that
  Item~\ref{it:properties_circ-bipolarity_nonnegative_functions_support_function}
  is new.
\item
  Finally, Item~\ref{it:bipolarity_nonnegative_functions_greatest_minorant}
  is related to \cite[Theorem~15.4, p.~137]{Rockafellar:1970}, but our
  assumptions are weaker.
\end{enumerate}   
  
\begin{proposition}
  \label{pr:properties_circ-polarity_nonnegative_functions}
  For any function~$\fonctionprimal \colon \PRIMAL \to \barRR_{+}$,
  the following statements hold true.
  \begin{enumerate}
  \item
    \label{it:Polar_inequality}
    \emph{Polar inequality}
    \begin{equation}
      \nscal{\primal}{\dual} \leq
      \fonctionprimal\np{\primal} \UppTimes \Polarity{\fonctionprimal}\np{\dual}
      \eqsepv \forall \primal \in \PRIMAL \eqsepv \forall \dual \in \DUAL
      \eqfinp
      \label{eq:Polar_inequality}
    \end{equation}
    
  \item
    \label{it:properties_circ-polarity_nonnegative_functions_support_function}
    \begin{subequations}
      The following set 
      \begin{equation}
        \Primal_{\fonctionprimal} = 
        \na{0} \cup 
        \RR_{+}\MidLevelCurve{\fonctionprimal}{0} \cup
        \bigcup_{\primal \in \na{0<\fonctionprimal <+\infty} }
        \na{\frac{\primal}{\fonctionprimal\np{\primal}}} 
        \eqfinv
        \label{eq:Primal_fonctionprimal}
      \end{equation}
      is such that
      \begin{equation}
        \closedconvexhull{\Primal_{\fonctionprimal}} =
        \PolarSet{\MidLevelSet{\LFM{\fonctionprimal}}{0}}
        \eqsepv
        \label{eq:closedconvexhull_Primal_fonctionprimal}
      \end{equation}
    \end{subequations}
     where \( \MidLevelSet{\LFM{\fonctionprimal}}{0} \)
      is a \BipolarSet. The polar transform~$\Polarity{\fonctionprimal}$
      is a support function as follows
      \begin{equation}
        \Polarity{\fonctionprimal}=\SupportFunction{\Primal_{\fonctionprimal}}
        =\SupportFunction{\PolarSet{\MidLevelSet{\LFM{\fonctionprimal}}{0}}}
        \eqfinp 
        \label{eq:properties_circ-polarity_nonnegative_functions_support_function}
      \end{equation}     
 
  \item
    \label{it:properties_circ-polarity_nonnegative_functions_convex_lsc}
    The polar transform~$\Polarity{\fonctionprimal}$ is convex lsc
    (strictly positively) $1$-homogeneous and vanishes at the origin
    (\( \Polarity{\fonctionprimal}\np{0}=0 \)) --- that is,
    the function~\( \Polarity{\fonctionprimal} \) is a lsc gauge,
    with effective domain \( \domain\Polarity{\fonctionprimal}
    = \RR_{++} \MidLevelSet{\LFM{\fonctionprimal}}{0} \).
  \item
    \label{it:properties_circ-polarity_nonnegative_functions_infimum}
    The polar transform~$\Polarity{\fonctionprimal}$ is also given as an
    infimum by
    \begin{subequations}
      \begin{align}
        \Polarity{\fonctionprimal}\np{\dual}
        &=
          \inf\defset{\lambda \in \OpenIntervalOpen{0}{+\infty} }{ {\nscal{\primal}{\dual}}_+ \leq 
          \lambda\fonctionprimal\np{\primal} 
          \eqsepv \forall \primal \in \PRIMAL }
          \eqsepv \forall \dual \in \DUAL
          \eqfinv
          \label{eq:duality_coupling_polar_transform_+}
        \\
        &=
          \inf\defset{\lambda \in \OpenIntervalOpen{0}{+\infty} }{ \nscal{\primal}{\dual} \leq 
          \lambda\fonctionprimal\np{\primal} 
          \eqsepv \forall \primal \in \PRIMAL }
          \eqsepv \forall \dual \in \DUAL
          \eqfinp
          \label{eq:duality_coupling_polar_transform}
      \end{align}
    \end{subequations}

  \item
    \label{it:properties_circ-bipolarity_nonnegative_functions_support_function}
    The bipolar transform~\( \Polaritybi{\fonctionprimal} \colon \PRIMAL\to\barRR_{+} \)
    satisfies\footnote{%
      Equation~\eqref{eq:bipolarity_nonnegative_functions_support_function} is valid even if
      $\fonctionprimal$ is not proper, that is, even when $\fonctionprimal\equiv + \infty$.}
    \begin{equation}
      \Polaritybi{\fonctionprimal}
      =\SupportFunction{\MidLevelSet{\LFM{\fonctionprimal}}{0}}
      \eqfinp
      \label{eq:bipolarity_nonnegative_functions_support_function}
    \end{equation}
    As a consequence, the bipolar
    transform~\( \Polaritybi{\fonctionprimal} \colon \PRIMAL \to\barRR_{+} \) is
    convex lsc (strictly positively) $1$-homogeneous and vanishes at the
    origin (\( \Polaritybi{\fonctionprimal}\np{0}=0 \)) --- that is, the
    function~\( \Polaritybi{\fonctionprimal} \) is a lsc gauge.
  \item
    \label{it:bipolarity_nonnegative_functions_greatest_minorant}
    If $0\in\dom \fonctionprimal$,
    the bipolar transform~\( \Polaritybi{\fonctionprimal} \colon \PRIMAL\to\barRR_{+} \)
    is the greatest (strictly positively) $1$-homogeneous proper convex lsc function
    below~$\fonctionprimal$.
  \end{enumerate}
\end{proposition}  

\begin{proof}
  We consider a function~$\fonctionprimal \colon \PRIMAL \to \barRR_{+}$.

  As a preliminary result, observe that
  \begin{equation}
    0 \in \MidLevelSet{\LFM{\fonctionprimal}}{0}
    \iff \sup_{\primal \in \PRIMAL}
    \bp{ \nscal{\primal}{0}-\fonctionprimal\np{\primal} } \leq 0
    \iff \sup_{\primal \in \PRIMAL} \bp{-\fonctionprimal\np{\primal} } \leq 0
    \iff \fonctionprimal \geq 0
    \eqfinp 
    \label{eq:pr:properties_polarity_nonnegative_functions_preliminary_result}
  \end{equation}
  
  \begin{enumerate}
  \item
    We prove~\eqref{eq:Polar_inequality}.
    For any \( \primal \in \PRIMAL \), \( \dual \in \DUAL \), we have that
    \begin{align*}
      \fonctionprimal\np{\primal} \UppTimes
      \Polarity{\fonctionprimal}\np{\dual}
      &\geq
        \fonctionprimal\np{\primal} \UppTimes
        \Bp{ {\nscal{\primal}{\dual}}_+ \LowTimes \bpConverse{\fonctionprimal\np{\primal} } }
        \intertext{by definition~\eqref{eq:circ-polar_transform} of the polar
        transform~\(\Polarity{\fonctionprimal}\) and by
        isotony~\eqref{eq:lower_upper_Times_isotony}
        of the upper multiplication~$\UppTimes$}
      &=
        \fonctionprimal\np{\primal} \UppTimes
        \Bp{ {\nscal{\primal}{\dual}}_+ 
        \UppTimes \bpConverse{\fonctionprimal\np{\primal} } }
        \tag{because \( {\nscal{\primal}{\dual}}_+ \in \RR_{+} \)}
      \\
      &=
        \fonctionprimal\np{\primal}
        \UppTimes \bpConverse{\fonctionprimal\np{\primal} } \UppTimes
        {\nscal{\primal}{\dual}}_+ 
        \tag{by associativity and commutativity of \(\UppTimes\)}
      \\
      &\geq
        1 \UppTimes {\nscal{\primal}{\dual}}_+
        \tag{because \( \fonctionprimal\np{\primal}
        \UppTimes \bpConverse{\fonctionprimal\np{\primal} } \in \na{1,+\infty} \)}
      \\
      &\geq
        \nscal{\primal}{\dual}
        \eqfinp
        \tag{as \( {\nscal{\primal}{\dual}}_+ = \sup\bp{\nscal{\primal}{\dual},0}
        \geq \nscal{\primal}{\dual} \)}
    \end{align*}
  \item
    We first prove that \( \Polarity{\fonctionprimal}
    = \SupportFunction{\Primal_{\fonctionprimal}} \)
    in~\eqref{eq:properties_circ-polarity_nonnegative_functions_support_function}. 
  
    By definition~\eqref{eq:circ-polar_transform} of the polar
    transform~$\Polarity{\fonctionprimal}$, we have that 
    \begin{align*}
      \Polarity{\fonctionprimal}
      &=
        \sup_{\primal \in \PRIMAL } \Bp{ {\nscal{\primal}{\cdot}}_+ 
        \LowTimes \bpConverse{\fonctionprimal\np{\primal} } }
        \intertext{where \( \nscal{\primal}{\cdot} \) denotes the continuous linear form
        \( \DUAL\ni\dual \mapsto \nscal{\primal}{\dual}\),}
      &=
        \sup\Bp{
        \sup_{\primal \in \MidLevelCurve{\fonctionprimal}{+\infty} } {\nscal{\primal}{\cdot}}_+ 
        \LowTimes 0, 
        \sup_{\primal \in \na{0<\fonctionprimal <+\infty} } {\nscal{\primal}{\cdot}}_+ 
        \times \bpConverse{\fonctionprimal\np{\primal} },
        \sup_{\primal \in \MidLevelCurve{\fonctionprimal}{0} } {\nscal{\primal}{\cdot}}_+ 
        \LowTimes \np{+\infty}
        }
        \intertext{where we have
        used~\eqref{eq:and_the_inverse_operation_extended_as},
        \eqref{eq:lower_multiplication_real_numbers},}
      &=
        \sup\Bp{
        \sup_{\primal \in \MidLevelCurve{\fonctionprimal}{+\infty} } 0,
        \sup_{\primal \in \na{0<\fonctionprimal <+\infty} } {\bscal{\frac{\primal}{\fonctionprimal\np{\primal}}}{\cdot}}_+, 
        \sup_{\primal \in \MidLevelCurve{\fonctionprimal}{0} } {\nscal{\primal}{\cdot}}_+ 
        \LowTimes \np{+\infty}  
        }
    \end{align*}
    where we have used~\eqref{eq:lower_multiplication_0}.
    In the above expression with three terms, the
    middle term is
    \begin{align*}
      \sup_{\primal \in \na{0<\fonctionprimal <+\infty} }
      {\bscal{\frac{\primal}{\fonctionprimal\np{\primal}}}{\cdot}}_+       
      &=
        \sup_{\primal \in \na{0<\fonctionprimal <+\infty} } \sup\bp{0,
        \bscal{\frac{\primal}{\fonctionprimal\np{\primal}}}{\cdot} } 
      \\
      &=
        \sup\bp{\sup_{\primal \in \na{0<\fonctionprimal <+\infty} }0,
        \sup_{\primal \in \na{0<\fonctionprimal <+\infty} }
        \bscal{\frac{\primal}{\fonctionprimal\np{\primal}}}{\cdot} }
        \eqfinp 
    \end{align*}
    In the last term, we have that (by~\eqref{eq:lower_multiplication_infty},
    \eqref{eq:lower_multiplication_0})
    \begin{equation*}
      {\nscal{\primal}{\dual}}_+ \LowTimes \np{+\infty}=
      \begin{cases}
        0 & \text{ if } \nscal{\primal}{\dual} \leq 0 
        \\
        +\infty & \text{ if } \nscal{\primal}{\dual} > 0 
      \end{cases}
      \qquad = \Indicator{\na{\nscal{\primal}{\cdot} \leq 0}}\np{\dual}
      = \SupportFunction{\RR_{+}\primal}\np{\dual}
      \eqfinv
    \end{equation*}
    so that the last term can be rewritten as
    \begin{equation*}
      \sup_{\primal \in \MidLevelCurve{\fonctionprimal}{0} } {\nscal{\primal}{\cdot}}_+ 
      \LowTimes \np{+\infty}
      = \sup_{\primal \in \MidLevelCurve{\fonctionprimal}{0}}
      \SupportFunction{\RR_{+}\primal}
      = \SupportFunction{\bigcup_{\primal \in \MidLevelCurve{\fonctionprimal}{0}}\RR_{+}\primal}
      = \SupportFunction{\RR_{+}\bigcup_{\primal \in \MidLevelCurve{\fonctionprimal}{0}}\primal}
      = \SupportFunction{\RR_{+}\MidLevelCurve{\fonctionprimal}{0}}
      \eqfinp
    \end{equation*}
    Thus, finally, we have obtained that
    \begin{align*}
      \Polarity{\fonctionprimal}
      &=
        \sup\Bp{
        \sup_{\primal \in \MidLevelCurve{\fonctionprimal}{+\infty} } 0, \,
        \sup\bp{\sup_{\primal \in \na{0<\fonctionprimal <+\infty} }0,
        \sup_{\primal \in \na{0<\fonctionprimal <+\infty} }
        \bscal{\frac{\primal}{\fonctionprimal\np{\primal}}}{\cdot}},
        \SupportFunction{\RR_{+}\MidLevelCurve{\fonctionprimal}{0}}
        }
      \\
      &=
        \sup\Bp{
        \sup_{\primal \in \MidLevelCurve{\fonctionprimal}{+\infty} } 0,
        \sup_{\primal \in \na{0<\fonctionprimal <+\infty} }0,
        \sup_{\primal \in \na{0<\fonctionprimal <+\infty} }
        \bscal{\frac{\primal}{\fonctionprimal\np{\primal}}}{\cdot},
        \SupportFunction{\RR_{+}\MidLevelCurve{\fonctionprimal}{0}}
        }
      \\
      &=
        \sup\Bp{
        \sup_{\primal \in \na{0<\fonctionprimal} } 0,
        \sup_{\primal \in \na{0<\fonctionprimal <+\infty} }
        \bscal{\frac{\primal}{\fonctionprimal\np{\primal}}}{\cdot},
        \SupportFunction{\RR_{+}\MidLevelCurve{\fonctionprimal}{0}}
        }
        \tag{as \( \na{0<\fonctionprimal} = \MidLevelCurve{\fonctionprimal}{+\infty}
        \cup \na{0<\fonctionprimal <+\infty} \)}
      \\
      &=
      \begin{cases}
        \sup\Bp{
        0,
        \sup_{\primal \in \na{0<\fonctionprimal <+\infty} }
        \bscal{\frac{\primal}{\fonctionprimal\np{\primal}}}{\cdot},
        \SupportFunction{\RR_{+}\MidLevelCurve{\fonctionprimal}{0}} }
        &  \text{ if } \na{0<\fonctionprimal}\neq\emptyset 
        \eqfinv
        \\
        = \SupportFunction{
        \underbrace{ \na{0} \cup 
        \bigcup_{\primal \in \na{0<\fonctionprimal <+\infty} }
        \na{\frac{\primal}{\fonctionprimal\np{\primal}}}
        \cup  \RR_{+}\MidLevelCurve{\fonctionprimal}{0} }_{\Primal_{\fonctionprimal}} }
        &
        \\
        & \quad
        \\
        \sup\Bp{
        -\infty,-\infty,    \SupportFunction{\RR_{+}\MidLevelCurve{\fonctionprimal}{0}}}
        &  \text{ if } \na{0<\fonctionprimal}=\emptyset
          \text{ (as \( \na{0<\fonctionprimal <+\infty} \subset \na{0<\fonctionprimal}=\emptyset \))}
        \\
        = \SupportFunction{ \underbrace{ \RR_{+}\MidLevelCurve{\fonctionprimal}{0}}_{\Primal_{\fonctionprimal}} }
        \eqfinv
        &
      \end{cases}
      \\
      &=
        \SupportFunction{\Primal_{\fonctionprimal}}
        \text{  where \( \Primal_{\fonctionprimal} \) is given
        by~\eqref{eq:Primal_fonctionprimal}.}
    \end{align*}
    Thus, we have shown that \( \Polarity{\fonctionprimal}
    = \SupportFunction{\Primal_{\fonctionprimal}} \), which is the left hand side equality 
    in~\eqref{eq:properties_circ-polarity_nonnegative_functions_support_function}. 
    The right hand side equality 
    in~\eqref{eq:properties_circ-polarity_nonnegative_functions_support_function}
    is a consequence of~\eqref{eq:closedconvexhull_Primal_fonctionprimal},
    that we are going to prove now.

    We have
    \begin{align*}
      \PolarSet{ \np{\closedconvexhull{\Primal_{\fonctionprimal}} } }
      &=
        \PolarSet{ \Primal_{\fonctionprimal}}
        \tag{by definition~\eqref{eq:(negative)polar_set} of a (negative) or (one-sided) polar set}
      \\
      &=
      \PolarSet{ \Bp{
      \na{0} \cup 
      \RR_{+}\MidLevelCurve{\fonctionprimal}{0} \cup
      \bigcup_{\primal \in \na{0<\fonctionprimal <+\infty} }
      \na{\frac{\primal}{\fonctionprimal\np{\primal}}} } }
      \tag{by definition~\eqref{eq:Primal_fonctionprimal} of the set~$\Primal_{\fonctionprimal}$}
      \\
      &=
        \PolarSet{ \na{0} } \cap 
        \PolarSet{\np{ \RR_{+}\MidLevelCurve{\fonctionprimal}{0} } } \cap
        \bigcap_{\primal \in \na{0<\fonctionprimal <+\infty} }
        \PolarSet{ \na{\frac{\primal}{\fonctionprimal\np{\primal}}} }
        \tag{by~\eqref{eq:polar_of_union}}
      \\
      &=
        \DUAL \cap 
        \PolarCone{\MidLevelCurve{\fonctionprimal}{0} } \cap
        \bigcap_{\primal \in \na{0<\fonctionprimal <+\infty} }
        \PolarSet{ \na{\frac{\primal}{\fonctionprimal\np{\primal}}} }
        \tag{by the polar cone definition~\eqref{eq:(negative)polar_cone}}
      \\
      &=
        \bigcap_{\primal \in \MidLevelCurve{\fonctionprimal}{0}}
        \defset{ \dual \in \DUAL }{ \nscal{\primal}{\dual} \leq 0}
        \cap
        \bigcap_{\primal \in \na{0<\fonctionprimal <+\infty} }
        \defset{ \dual \in \DUAL }{ \nscal{\frac{\primal}{\fonctionprimal\np{\primal}}}{\dual} \leq 1}
      \\ 
      &=
        \bigcap_{\primal \in \MidLevelCurve{\fonctionprimal}{0}}
        \defset{ \dual \in \DUAL }{ \nscal{\primal}{\dual} \leq \fonctionprimal\np{\primal}}
        \cap
        \bigcap_{\primal \in \na{0<\fonctionprimal <+\infty} }
        \defset{ \dual \in \DUAL }{ \nscal{\primal}{\dual} \leq \fonctionprimal\np{\primal}}
        \intertext{} 
      &=
        \bigcap_{\primal \in \na{\fonctionprimal <+\infty} }
        \defset{ \dual \in \DUAL }{ \nscal{\primal}{\dual} \leq \fonctionprimal\np{\primal}}
      \\ 
      &=
        \bigcap_{\primal \in \PRIMAL}
        \defset{ \dual \in \DUAL }{ \nscal{\primal}{\dual} \leq
        \fonctionprimal\np{\primal}}
      \\
      &=
        \defset{ \dual \in \DUAL }{ \nscal{\primal}{\dual} -
        \fonctionprimal\np{\primal} \leq 0 
        \eqsepv \forall \primal \in \PRIMAL} 
      \\
      &=
        \defset{ \dual \in \DUAL }{ \sup_{\primal \in \PRIMAL} \bp{ 
        \nscal{\primal}{\dual} - \fonctionprimal\np{\primal} } \leq 0 }
      \\
      &=
        \defset{ \dual \in \DUAL }{ \LFM{\fonctionprimal}\np{\dual} \leq 0 }
      \\
      &=
        \MidLevelSet{\LFM{\fonctionprimal}}{0}
        \eqfinp
    \end{align*}
    As the set \( \closedconvexhull{\Primal_{\fonctionprimal}} \) is closed convex and
    contains~$0,$ it is a \BipolarSet\ (by Item~\ref{it:bipolar_set_def_1} of
    Definition~\ref{de:bipolar_set_dual_pair})
    and we deduce, using the bipolar Theorem expressed
    in~\eqref{eq:biPolarSet}, that
    \begin{equation*}
      \closedconvexhull{\Primal_{\fonctionprimal}} 
      = \biPolarSet{ \np{\closedconvexhull{\Primal_{\fonctionprimal}} } }
      =\PolarSet{ \MidLevelSet{\LFM{\fonctionprimal}}{0} }
      \eqfinv
    \end{equation*}
    which exactly is~\eqref{eq:closedconvexhull_Primal_fonctionprimal}.
    As we have shown that \( \Polarity{\fonctionprimal}
    = \SupportFunction{\Primal_{\fonctionprimal}} \), the right hand side equality 
    in~\eqref{eq:properties_circ-polarity_nonnegative_functions_support_function}
    follows from \( \SupportFunction{\Primal_{\fonctionprimal}} =
    \SupportFunction{\closedconvexhull{\Primal_{\fonctionprimal}}} =
    \SupportFunction{\PolarSet{ \MidLevelSet{\LFM{\fonctionprimal}}{0}}} \).

    We also have that \( {\MidLevelSet{\LFM{\fonctionprimal}}{0}} \) is a
    \BipolarSet: indeed, it is closed convex as a level set of a closed convex
    function and \( 0\in {\MidLevelSet{\LFM{\fonctionprimal}}{0}} \)
    by~\eqref{eq:pr:properties_polarity_nonnegative_functions_preliminary_result},
    and we conclude with Item~\ref{it:bipolar_set_def_1} of
    Definition~\ref{de:bipolar_set_dual_pair}.
  
  \item By the just proven
    Item~\ref{it:properties_circ-polarity_nonnegative_functions_support_function},
    the polar transform~$\Polarity{\fonctionprimal}$ is the support
    function~\(\SupportFunction{\PolarSet{
        \MidLevelSet{\LFM{\fonctionprimal}}{0} }}\) of a nonempty set, hence it
    is convex, lsc, (strictly positively) $1$-homogeneous and takes the
    value~$0$ at the origin (\( \Polarity{\fonctionprimal}\np{0}=0 \)), as
    recalled in~\S\ref{Polar_of_a_set,_bipolar_set}.
    The effective domain of the support
    function~\(\SupportFunction{\PolarSet{
        \MidLevelSet{\LFM{\fonctionprimal}}{0}}} \) is
    $\RR_{++}\biPolarSet{ \MidLevelSet{\LFM{\fonctionprimal}}{0}}$, which is
    equal to $\RR_{++}{ \MidLevelSet{\LFM{\fonctionprimal}}{0}}$, using the fact
    that \( {\MidLevelSet{\LFM{\fonctionprimal}}{0}} \) is a \BipolarSet, as
    proved above in
    Item~\ref{it:properties_circ-polarity_nonnegative_functions_support_function}.

    
  \item Finally, we
    prove~\eqref{eq:duality_coupling_polar_transform_+}--\eqref{eq:duality_coupling_polar_transform}.
    For any \( \dual \in \DUAL \), we have that
    \begin{subequations}
      \begin{align*}
        \Polarity{\fonctionprimal}\np{\dual}
        &=
          \SFM{\fonctionprimal}{\Duality\np{\coupling}}\np{\dual}
          \tag{by definition~\eqref{eq:circ-polar_transform} of the polar
          transform~$\Polarity{\fonctionprimal}$
          and by~\eqref{eq:times-duality_definition}} 
        \\      
        &=
          \inf\defset{\alpha \in \barRR_{+} }{
          {\nscal{\primal}{\dual}}_+ \LowTimes \Converse{\alpha}
          \leq 
          \fonctionprimal\np{\primal} 
          \eqsepv \forall \primal \in \PRIMAL }
          \tag{by~\cite[Equation~(4.17)]{Martinez-Legaz-Singer:1994}}
        \\
        &=
          \inf\defset{\lambda \in \OpenIntervalOpen{0}{+\infty} }{
          \frac{1}{\lambda} {\nscal{\primal}{\dual}}_+ 
          \leq     \fonctionprimal\np{\primal} 
          \eqsepv \forall \primal \in \PRIMAL }
          \eqfinv 
          \tag{by~\cite[Equation~(4.17)]{Martinez-Legaz-Singer:1994}}
        \\
        &=
          \inf\defset{\lambda \in \OpenIntervalOpen{0}{+\infty} }{
          {\nscal{\primal}{\dual}}_+
          \leq \lambda\fonctionprimal\np{\primal} 
          \eqsepv \forall \primal \in \PRIMAL }
          \eqfinv 
        \\
        &=
          \inf\defset{\lambda \in \OpenIntervalOpen{0}{+\infty} }{
          \nscal{\primal}{\dual} \leq 
          \lambda\fonctionprimal\np{\primal} 
          \eqsepv \forall \primal \in \PRIMAL }
          \eqfinv 
      \end{align*}
    \end{subequations}
    because \(
    {\nscal{\primal}{\dual}}_+=\max\na{\nscal{\primal}{\dual},0}\)
    and \( \lambda\fonctionprimal\np{\primal} \geq 0 \),
    as \( \lambda \in \OpenIntervalOpen{0}{+\infty} \) and
    \( \fonctionprimal\np{\primal} \in \barRR_{+} \).

  \item
    We prove~\eqref{eq:bipolarity_nonnegative_functions_support_function} as follows:
    \begin{align*}
      \Polaritybi{\fonctionprimal}
      &=
        \PolarityReverse{\np{\SupportFunction{\PolarSet{\MidLevelSet{\LFM{\fonctionprimal}}{0}}}}}
        \tag{by definition~\eqref{eq:circ-polar_bitransform} of the bipolar
        transform and by~\eqref{eq:properties_circ-polarity_nonnegative_functions_support_function}}
      \\
      &=
        \SupportFunction{\biPolarSet{\MidLevelSet{\LFM{\fonctionprimal}}{0}}}
        \intertext{by~\eqref{eq:PolaritySupportFunctionDual} proven below
        (there is no circularity in the reasoning,
        as~\eqref{eq:PolaritySupportFunctionDual} is proven by only
        using~\eqref{eq:properties_circ-polarity_nonnegative_functions_support_function}
        established before)}
      &=
        \SupportFunction{\MidLevelSet{\LFM{\fonctionprimal}}{0}}
        \eqfinp
        \tag{since \(\MidLevelSet{\LFM{\fonctionprimal}}{0} \) is a bipolar set
        as seen in Item~\ref{it:properties_circ-polarity_nonnegative_functions_support_function}}
    \end{align*}
    The rest of the assertions in
    Item~\ref{it:properties_circ-bipolarity_nonnegative_functions_support_function}
    are proven in the same way than for
    Item~\ref{it:properties_circ-polarity_nonnegative_functions_convex_lsc}.

  \item
    Using Equation~\eqref{eq:pr:properties_polarity_nonnegative_functions_preliminary_result},
    we obtain that $0 \in \MidLevelSet{\LFM{\fonctionprimal}}{0}$. Thus, 
    $\MidLevelSet{\LFM{\fonctionprimal}}{0}\not=\emptyset$ and, by assumption, we also have  $0\in\dom\fonctionprimal$.
    Thus, using Proposition~\ref{prop:best_pos_hom_cvx_subset} postponed in Appendix~\ref{app:best_cvx_general},
    we obtain that the greatest lsc convex
    (strictly positively) $1$-homogeneous lower approximation
    of $\fonctionprimal$ is given by $\SupportFunction{\MidLevelSet{\LFM{\fonctionprimal}}{0}}$.
    As this function is also proper (as the support function of a nonempty set), we conclude that it is also the
    greatest lsc proper convex (strictly positively) $1$-homogeneous lower approximation
    of $\fonctionprimal$.
    Now, using Equation~\eqref{eq:bipolarity_nonnegative_functions_support_function}, we have that 
    \( \Polaritybi{\fonctionprimal} = \SupportFunction{\MidLevelSet{\LFM{\fonctionprimal}}{0}} \) and
    the conclusion follows for $\Polaritybi{\fonctionprimal}$.
  \end{enumerate}

  This ends the proof. 
\end{proof}

\subsubsection{Examples of polar transforms as support functions}
\label{Examples_of_polar_transforms_as_support_functions}

Using
Item~\ref{it:properties_circ-polarity_nonnegative_functions_support_function}
(Equation~\eqref{eq:properties_circ-polarity_nonnegative_functions_support_function})
and
Item~\ref{it:properties_circ-bipolarity_nonnegative_functions_support_function}
(Equation~\eqref{eq:bipolarity_nonnegative_functions_support_function}) in
Proposition~\ref{pr:properties_circ-polarity_nonnegative_functions}, we obtain
expressions of the polar transforms of nonnegative support functions, of
Minkowski functionals, of indicator functions and of generalized indicator
functions as support functions.  Equations~\eqref{eq:biPolaritySupportFunction}
can be deduced from \cite[Corollary~15.1.2, p.~129]{Rockafellar:1970}.

\begin{proposition}
  \label{pr:polar_transforms_of_nonnegative_functions_as_support_functions}
  \quad
  \begin{enumerate}
  \item
    \label{it:PolaritySupportFunction}
    Polar transform of a nonnegative support function as a support function.\\
    \begin{subequations}
      For any \BipolarSet s \( \Primal \subset \PRIMAL \) and $\Dual \subset \DUAL$, we have that 
      \begin{align}
        \Polarity{\SupportFunction{\Dual}}
        &=
          \SupportFunction{\PolarSet{\Dual}}
          \eqfinv
          \label{eq:PolaritySupportFunctionDual}
        \\
        \PolarityReverse{\SupportFunction{\Primal}}
        &=
          \SupportFunction{\PolarSet{\Primal}}
          \eqfinv
          \label{eq:PolaritySupportFunctionPrimal}
        \\
        \Polaritybi{\SupportFunction{\Dual}}
        &=
          \SupportFunction{\Dual}
          \eqfinp
          \label{eq:biPolaritySupportFunctionDual}
      \end{align}
      \label{eq:biPolaritySupportFunction}
    \end{subequations}
    
  \item
    Polar transform of a Minkowski functional as a support function.\\
    \begin{subequations}
      For any subsets \( \Primal \subset \PRIMAL \) and $\Dual \subset \DUAL$, we have that 
      \begin{align}
        \Polarity{\MinkowskiFunctional{\Primal}}
        &=
          \SupportFunction{\biPolarSet{\Primal}}
          \eqfinv
          \label{eq:PolarityMinkowskiFunctionalPrimal_as_a_support_function}
        \\
        \PolarityReverse{\MinkowskiFunctional{\Dual}}
        &=
          \SupportFunction{\biPolarSet{\Dual}}
          \eqfinv
          \label{eq:PolarityMinkowskiFunctionalDual_as_a_support_function}
        \\
        \Polaritybi{\MinkowskiFunctional{\Primal}}
        &=
          \SupportFunction{\PolarSet{\Primal}}
          \eqfinp
          \label{eq:biPolarityMinkowskiFunctionalPrimal_as_a_support_function}
      \end{align}
    \end{subequations}

  \item 
    Polar transform of an indicator function as a support function.\\
    For any subsets \( \Primal \subset \PRIMAL \) and $\Dual \subset \DUAL$, we have that 
    \begin{subequations}
      \begin{align}
        \Polarity{\Indicator{\Primal}}
        &=
          \SupportFunction{\biPolarCone{\Primal}}=
          \Indicator{\PolarCone{\Primal}}
          \eqfinv 
          \label{eq:Polar_transform_of_a_indicator_function_as_a_support_function}
        \\
        \PolarityReverse{\Indicator{\Dual}}
        &=
          \SupportFunction{\biPolarCone{\Dual}}=
          \Indicator{\PolarCone{\Dual}}
          \eqfinv 
          \label{eq:ReversePolar_transform_of_a_indicator_function_as_a_support_function}
        \\      
        \Polaritybi{\Indicator{\Primal}}
        &=
          \SupportFunction{\PolarCone{\Primal}}=
          \Indicator{\biPolarCone{\Primal}}
          \eqfinv 
          \label{eq:biPolar_transform_of_a_indicator_function_as_a_support_function}
      \end{align}
    \end{subequations}
    
  \item 
    Polar transform of a generalized indicator function as a support function.\\
    For any subsets \( \Primal \subset \PRIMAL \) and $\Dual \subset \DUAL$, we have that 
    \begin{subequations}
      \begin{align}
        \Polarity{\chi_{\Primal}}
        &=
          \SupportFunction{\biPolarSet{\Primal}}
          \eqfinv
          \label{eq:Polar_transform_of_a_generalized_indicator_function_as_a_support_function}
        \\
        \PolarityReverse{\chi_{\Dual}}
        &=
          \SupportFunction{\biPolarSet{\Dual}}
          \label{eq:ReversePolar_transform_of_a_generalized_indicator_function_as_a_support_function}
          \eqfinv
        \\
        \Polaritybi{\chi_{\Primal}}
        &=
          \SupportFunction{\PolarSet{\Primal}} 
          \eqfinp
          \label{eq:biPolar_transform_of_a_generalized_indicator_function_as_a_support_function}
      \end{align}
    \end{subequations}

  \end{enumerate}
\end{proposition}

\begin{proof}
  \quad
  \begin{enumerate}

  \item 
    As both \( \Primal \) and $\Dual $ are \BipolarSet s, they both contain~$0$
    (see Item~\ref{it:bipolar_set_def_1} of Definition~\ref{de:bipolar_set_dual_pair}), and thus both
    \( \SupportFunction{\Primal} \geq 0 \) 
    and \( \SupportFunction{\Dual} \geq 0 \).
    As the set~$\Dual $ is nonempty closed convex (see Item~\ref{it:bipolar_set_def_1} of Definition~\ref{de:bipolar_set_dual_pair}),
    the function~\( \Indicator{\Dual} \) is proper closed convex.
    As the Fenchel conjugacy \( \fonctiondual \mapsto \LFM{\fonctiondual} \)
induces a one-to-one correspondence
between the closed convex functions on~$\DUAL$ and themselves
(see \cite[Theorem~5]{Rockafellar:1974} recalled in~\S\ref{Dual_pair,_paired_vector_spaces}),
we get that \( \Indicator{\Dual}=\LFMrbi{\Indicator{\Dual}} =\LFM{\SupportFunction{\Dual}} \)
as the equality \( \LFMr{\Indicator{\Dual}} =\SupportFunction{\Dual} \)
follows from the very definition of the support fonction~\( \SupportFunction{\Dual} \).
Thus, we get that \( \MidLevelSet{\LFM{\SupportFunction{\Dual}}}{0}=
\MidLevelSet{\Indicator{\Dual}}{0}= \Dual \) and
    then,
    by~\eqref{eq:properties_circ-polarity_nonnegative_functions_support_function},
    we obtain~\eqref{eq:PolaritySupportFunctionDual}.

    Because the {reverse polar transform}~\eqref{eq:reverse_circ-polar_transform}
    acts like the polar transform~\eqref{eq:circ-polar_transform} on nonnegative
    functions, we obtain~\eqref{eq:PolaritySupportFunctionPrimal}
    in the same fashion.

    Finally, by definition~\eqref{eq:circ-polar_bitransform} of the bipolar
    transform, we apply~\eqref{eq:PolaritySupportFunctionDual}
    and then~\eqref{eq:PolaritySupportFunctionPrimal} with $\Primal=\PolarSet{\Dual}$,
    which is a \BipolarSet, and get
    \( \Polaritybi{\SupportFunction{\Dual}}
    =\PolarityReverse{\np{\Polarity{\SupportFunction{\Dual}}}}
    =\PolarityReverse{\np{\SupportFunction{\PolarSet{\Dual}}}}
    =\SupportFunction{\biPolarSet{\Dual}}
    =\SupportFunction{\Dual}
    \), since \( \biPolarSet{\Dual}= \Dual \) as $\Dual $ is a \BipolarSet.
    Thus, we have obtained~\eqref{eq:biPolaritySupportFunctionDual}.

  \item 
    The {Minkowski functional} in~\eqref{eq:MinkowskiFunctional}
    associated with the subset~\( \Primal \subset \PRIMAL \) can be written as
    \( \MinkowskiFunctional{\Primal}
    = \inf_{\lambda\in\RR_{++}} \bp{\lambda+\Indicator{\lambda\Primal}} \),
    from which we obtain the Fenchel conjugate
    \begin{align*}
      \LFM{\MinkowskiFunctional{\Primal}}
      &=
        \LFM{\bp{\inf_{\lambda\in\RR_{++}} \np{\lambda+\Indicator{\lambda\Primal}}}}
      \\
      &=
        \sup_{\lambda\in\RR_{++}} \bp{-\lambda+\LFM{\Indicator{\lambda\Primal}}}
        \tag{by property of conjugacies} 
      \\
      &=
        \sup_{\lambda\in\RR_{++}} \bp{\lambda \np{\SupportFunction{\Primal}-1}}
        \tag{as \(
        \LFM{\Indicator{\lambda\Primal}}=\SupportFunction{\lambda\Primal}
        =\lambda\SupportFunction{\Primal}\)}
      \\
      &=
        \begin{cases}
          0 & \textrm{ if } \SupportFunction{\Primal}\leq 1
          \\
          +\infty & \textrm{ if } \SupportFunction{\Primal}> 1
        \end{cases}
      \\
      &=
        \Indicator{\MidLevelSet{\SupportFunction{\Primal}}{1}}
      \\
      &=
        \Indicator{\PolarSet{\Primal}}
        \tag{by definition~\eqref{eq:(negative)polar_set} of $\PolarSet{\Primal}$}
        \eqfinp
    \end{align*}
    Thus, we get that 
    \begin{equation}
      \LFM{\MinkowskiFunctional{\Primal}}=
      \Indicator{\PolarSet{\Primal}}
      \eqsepv
      \text{ and }
      \MidLevelSet{\LFM{\MinkowskiFunctional{\Primal}}}{0}
      = \PolarSet{\Primal}
      \eqfinv
      \label{eq:Polar_transform_of_Minkowski_functionals}
    \end{equation}
    and then, 
    using~\eqref{eq:properties_circ-polarity_nonnegative_functions_support_function},
    we obtain~\eqref{eq:PolarityMinkowskiFunctionalPrimal_as_a_support_function} by
    \(    \Polarity{\MinkowskiFunctional{\Primal}}
    = \SupportFunction{\PolarSet{\MidLevelSet{\LFM{\MinkowskiFunctional{\Primal}}}{0}}}
    = \SupportFunction{\biPolarSet{\Primal}} \). 

    Because the {reverse polar transform}~\eqref{eq:reverse_circ-polar_transform}
    acts like the polar transform~\eqref{eq:circ-polar_transform} on nonnegative functions,
    we obtain~\eqref{eq:PolarityMinkowskiFunctionalDual_as_a_support_function} in the same fashion.

    Finally, by definition~\eqref{eq:circ-polar_bitransform} of the bipolar
    transform, we apply~\eqref{eq:PolarityMinkowskiFunctionalPrimal_as_a_support_function}
    and then~\eqref{eq:PolaritySupportFunctionPrimal} to the support function of
    the bipolar set $\biPolarSet{\Primal}$ and get
    \[ \Polaritybi{\MinkowskiFunctional{\Primal}}
      =\PolarityReverse{\np{\Polarity{\MinkowskiFunctional{\Primal}}}}
      =\PolarityReverse{\np{\SupportFunction{\biPolarSet{\Primal}}}}
      =\SupportFunction{\triPolarSet{\Primal}}
      =\SupportFunction{\PolarSet{\Primal}}
      \eqfinv
    \]
    since \( \triPolarSet{\Primal}= \PolarSet{\Primal} \).
    Thus, we have obtained~\eqref{eq:biPolarityMinkowskiFunctionalPrimal_as_a_support_function}.

  \item
    We have that
    \( \LFM{\Indicator{\Primal}} = \SupportFunction{\Primal} \), and hence 
    \( \MidLevelSet{\LFM{\Indicator{\Primal}}}{0}
    =  \MidLevelSet{\SupportFunction{\Primal}}{0} 
    = \PolarCone{\Primal} \)
    by~\eqref{eq:(negative)polar_cone}. 
    By~\eqref{eq:properties_circ-polarity_nonnegative_functions_support_function},
    we get that
    \( \Polarity{\Indicator{\Primal}} = 
    \SupportFunction{\PolarSet{\MidLevelSet{\LFM{\Indicator{\Primal}}}{0}}}=
    \SupportFunction{\PolarSet{\np{\PolarCone{\Primal}}}} \),
    where \( \PolarSet{\np{\PolarCone{\Primal}}}=
    \PolarCone{\np{\PolarCone{\Primal}}} \) because
    \( \PolarCone{\Primal} \) is a cone.
    By~\eqref{eq:(negative)bipolar_cone}, we get that
    \( \Polarity{\Indicator{\Primal}} =
    \SupportFunction{\biPolarCone{\Primal}} \).
    Finally, as \( \biPolarCone{\Primal} \) is cone, we have that
    \( \SupportFunction{\biPolarCone{\Primal}}
    = \Indicator{\triPolarCone{\Primal}}
    = \Indicator{\PolarCone{\Primal}} \) since
    \( \triPolarCone{\Primal} = \PolarCone{\Primal} \).
    We have proven~\eqref{eq:Polar_transform_of_a_indicator_function_as_a_support_function}.

    Because the {reverse polar
      transform}~\eqref{eq:reverse_circ-polar_transform} acts like the polar
    transform~\eqref{eq:circ-polar_transform} on nonnegative functions, we
    obtain~\eqref{eq:ReversePolar_transform_of_a_indicator_function_as_a_support_function}
    in the same fashion.

    Finally,
    from~\eqref{eq:Polar_transform_of_a_indicator_function_as_a_support_function}
    and~\eqref{eq:ReversePolar_transform_of_a_indicator_function_as_a_support_function},
    we deduce
    \( \Polaritybi{\Indicator{\Primal}} =
    \PolarityReverse{\np{\Polarity{\Indicator{\Primal}}}} =
    \PolarityReverse{\Indicator{\PolarCone{\Primal}}} =
    \Indicator{\biPolarCone{\Primal}} \), where
    \( \Indicator{\biPolarCone{\Primal}} =
    \SupportFunction{\PolarCone{\Primal}}\) because \( \PolarCone{\Primal} \) is
    a cone.
        
  \item 
    As $\chi_{\Primal}=\Indicator{\Primal}+1$, we have that
    \(
    \LFM{\chi_{\Primal}}=\LFM{\np{\Indicator{\Primal}+1}}=\LFM{\Indicator{\Primal}}-1=
    \SupportFunction{\Primal}-1 \), and hence that
    
    \( \MidLevelSet{\LFM{\chi_{\Primal}}}{0}
    =  \MidLevelSet{\SupportFunction{\Primal}}{1}
    = \PolarSet{\Primal} \)
    by~\eqref{eq:(negative)polar_set}.
    By~\eqref{eq:properties_circ-polarity_nonnegative_functions_support_function},
    we
    deduce~\eqref{eq:Polar_transform_of_a_generalized_indicator_function_as_a_support_function}
    by \(     \Polarity{\chi_{\Primal}}
    =\SupportFunction{\PolarSet{\MidLevelSet{\LFM{\chi_{\Primal}}}{0}}}
    =\SupportFunction{\biPolarSet{\Primal}} \).

    Because the {reverse polar
      transform}~\eqref{eq:reverse_circ-polar_transform} acts like the polar
    transform~\eqref{eq:circ-polar_transform} on nonnegative functions, we
    obtain~\eqref{eq:ReversePolar_transform_of_a_generalized_indicator_function_as_a_support_function}
    in the same fashion.

    Finally,
    from~\eqref{eq:Polar_transform_of_a_generalized_indicator_function_as_a_support_function}
    and~\eqref{eq:PolaritySupportFunctionPrimal}, we deduce
    \( \Polaritybi{\chi_{\Primal}} =\PolarityReverse{\np{\Polarity{\chi_{\Primal}}}}
    =\PolarityReverse{\SupportFunction{\biPolarSet{\Primal}}} =
    \SupportFunction{\triPolarSet{\Primal}} =
    \SupportFunction{\PolarSet{\Primal}} \) since
    \( \triPolarSet{\Primal} = \PolarSet{\Primal} \).
  \end{enumerate}
\end{proof}


When \( \Dual \) is a unit ball, \( \SupportFunction{\Dual} \) is a
norm~$\norm{\cdot}$
and \( \Polarity{\SupportFunction{\Dual}}=\SupportFunction{\PolarSet{\Dual}}=\norm{\cdot}_{\star} \) is the
so-called \emph{dual norm}.

\subsection{Polar operation on functions}
\label{Polar_operation_on_functions}

In~\S\ref{Definition_of_the_polar_of_a_function},
we propose an extension of the polar transform from nonnegative to any functions.
Then, in~\S\ref{Polar_transform_as_a_Minkowski_functional}, we express the polar
transform of any function as a Minkowski functional, and we provide several
results about the polar of functions, some well-known (and scattered in the
literature) and some new.  Finally,
in~\S\ref{Examples_of_polar_transforms_as_Minkowski_functionals}, we provide
examples of polar transforms expressed as Minkowski functionals.

\subsubsection{Definition of the polar of a function}
\label{Definition_of_the_polar_of_a_function}

The equality~\eqref{eq:duality_coupling_polar_transform}
is taken as the definition of the polar transform of a gauge
in~\cite[p.~128]{Rockafellar:1970}.  In fact, we can use the
formula~\eqref{eq:duality_coupling_polar_transform} to extend
Definition~\ref{de:circ-polar_transform} to all functions, and not necessarily
nonnegative ones.
  
\begin{definition}
  \label{de:circ-polar_transform_any_function}
  For any function~$\fonctionprimal \colon \PRIMAL \to \barRR$,
  we define the \emph{polar transform}
  \( \Polarity{\fonctionprimal} \colon \DUAL \to \barRR_{+} \)
  by 
  \begin{subequations}
    \begin{equation}
      \Polarity{\fonctionprimal}\np{\dual}
      =
      \inf\defset{\lambda \in \OpenIntervalOpen{0}{+\infty} }{ \nscal{\primal}{\dual} \leq 
        \lambda\fonctionprimal\np{\primal} 
        \eqsepv \forall \primal \in \PRIMAL }
      \eqsepv \forall \dual\in\DUAL
      \eqfinp
      \label{eq:circ-polar_transform_any_function}
    \end{equation}
    For any function~$\fonctiondual \colon \DUAL \to \barRR$,
    the \emph{reverse polar transform}
    \( \PolarityReverse{\fonctiondual} \colon \PRIMAL \to \barRR_{+} \)
    of the function~$\fonctiondual$ is defined by
    \begin{equation}
      \PolarityReverse{\fonctiondual}\np{\primal}
      =
      \inf\defset{\lambda \in \OpenIntervalOpen{0}{+\infty} }{ \nscal{\primal}{\dual} \leq 
        \lambda\fonctiondual\np{\dual} 
        \eqsepv \forall \dual \in \DUAL }
      \eqsepv \forall \primal \in \PRIMAL
      \eqfinp
      \label{eq:reverse_circ-polar_transform_any_function}
    \end{equation}
    For any function~$\fonctionprimal \colon \PRIMAL \to \barRR$,
    we define the \emph{bipolar transform}
    \( \Polaritybi{\fonctionprimal} \colon \PRIMAL \to \barRR_{+} \)
    of the function~$\fonctionprimal$ by 
    \begin{equation}
      \Polaritybi{\fonctionprimal}
      =
      \PolarityReverse{\np{\Polarity{\fonctionprimal}}} 
      \eqfinp
      \label{eq:circ-polar_bitransform_any_function}
    \end{equation}
  \end{subequations}
  By the formula~\eqref{eq:duality_coupling_polar_transform},
  which coincides with~\eqref{eq:circ-polar_transform_any_function},
  the three definitions are consistent with those in Definition~\ref{de:circ-polar_transform}
  when $\fonctionprimal \colon \PRIMAL \to \barRR_{+}$.
\end{definition}


\subsubsection{Polar transform as a Minkowski functional}
\label{Polar_transform_as_a_Minkowski_functional}

We display systematic relationships of polar functions with Minkowski
functionals.  To our knowledge, the results in
Proposition~\ref{pr:polar_bipolar_Minkowski_functional} are new, if only because
they hold for any function, in contrast to
\cite[Theorem~15.1]{Rockafellar:1970},
\cite[Proposition~2.1]{Friedlander-Macedo-KeiPong:2014},
\cite[Theorem~4.1]{Aravkin-Burke-Drusvyatskiy-Friedlander-MacPhee:2018})
established for functions that are convex, or vanishing at zero, or (strictly
positively) $1$-homogeneous, or nonnegative.

\begin{proposition}
  \label{pr:polar_bipolar_Minkowski_functional}
  For any function~$\fonctionprimal \colon \PRIMAL\to\barRR$, 
  we have the following properties.
  \begin{enumerate}
  \item
    \label{it:polar_function_as_a_Minkowski_functional}
    The function~\( \Polarity{\fonctionprimal} \colon \DUAL\to\barRR_{+} \)
    is the Minkowski functional~\( \MinkowskiFunctional{\MidLevelSet{\LFM{\fonctionprimal}}{0}} \)
    of the closed convex subset 
    \( \MidLevelSet{\LFM{\fonctionprimal}}{0} \):
    \begin{equation}
      \Polarity{\fonctionprimal}=
      \MinkowskiFunctional{\MidLevelSet{\LFM{\fonctionprimal}}{0}}
      \eqfinp
      \label{eq:polar=Minkowski_functional}
    \end{equation}
    As a consequence, the polar
    transform~\( \Polarity{\fonctionprimal} \colon \DUAL\to\barRR_{+} \) is convex
    (strictly positively) $1$-homogeneous\footnote{Note that here, by
      contrast with
      Item~\ref{it:properties_circ-polarity_nonnegative_functions_convex_lsc} in
      Proposition~\ref{pr:properties_circ-polarity_nonnegative_functions}, the
      function $\Polarity{\fonctionprimal}$ may not be lsc. As an example,
      consider the function $\fonctionprimal=\SupportFunction{\na{1}}$
      on $\PRIMAL=\RR$. Using
      Equation~\eqref{eq:PolaritySupportFunctionDual_as_a_Minkowski_functional},
      we obtain that
      $\Polarity{\fonctionprimal}= \Polarity{\SupportFunction{\na{1}}} =
      \MinkowskiFunctional{\closedconvexhull\na{1}}
      =\MinkowskiFunctional{\na{1}}$.
      Now, \(\MinkowskiFunctional{\na{1}}(\primal) \) is equal to~$+\infty$ for \( \primal \leq 0\) and
      to~$\primal$ 
      for \( \primal > 0\). As a consequence, the function
      \(\MinkowskiFunctional{\na{1}} \) is not lsc at~$0$.}, with effective domain
    \( \domain\Polarity{\fonctionprimal} = \RR_{++}
    \MidLevelSet{\LFM{\fonctionprimal}}{0} \).  
  \item
     If \( \fonctionprimal\np{0}=0 \), the (Rockafellar-Moreau) subdifferential
     satisfies
     \begin{subequations}
    \begin{equation}
      \subdifferential{}{\fonctionprimal}\np{0} =
      \MidLevelCurve{\LFM{\fonctionprimal}}{0}=
      \MidLevelSet{\LFM{\fonctionprimal}}{0}
      \eqfinp
    \end{equation}
    As a consequence, we have that
    \begin{equation}
      \Polarity{\fonctionprimal}
      =\MinkowskiFunctional{\MidLevelSet{\LFM{\fonctionprimal}}{0}}
      =\MinkowskiFunctional{\MidLevelCurve{\LFM{\fonctionprimal}}{0}}
      =\MinkowskiFunctional{\subdifferential{}{\fonctionprimal}\np{0}}
      \eqfinp
      \label{eq:polar=Minkowski_functional_when_f(0)=0}
    \end{equation}       
     \end{subequations}
   \item
     \label{it:bipolar_function_as_a_Minkowski_functional}
     The function~\( \Polaritybi{\fonctionprimal} \colon \DUAL\to\barRR_{+} \)
     is the Minkowski functional~\(
     \MinkowskiFunctional{\PolarSet{\MidLevelSet{\LFM{\fonctionprimal}}{0}}} \) 
    of the \BipolarSet~\( \PolarSet{\MidLevelSet{\LFM{\fonctionprimal}}{0}} \):
    \begin{equation}
      \Polaritybi{\fonctionprimal} =
      \MinkowskiFunctional{\PolarSet{\MidLevelSet{\LFM{\fonctionprimal}}{0}}}
      \eqfinp
      \label{eq:bipolarasminkowski}
    \end{equation}
    As a consequence,
    the bipolar transform~\( \Polaritybi{\fonctionprimal} \colon \PRIMAL \to
    \barRR_{+} \)
    is convex lsc (strictly positively) $1$-homogeneous
    and vanishes at the origin
    (\( \biPolarity{\fonctionprimal}\np{0}=0 \)) ---  that is,
    the function~\( \Polaritybi{\fonctionprimal} \) is a lsc gauge.
  \end{enumerate}
\end{proposition}

\begin{proof}
  We consider a function~$\fonctionprimal \colon \PRIMAL \to \barRR$.
  \begin{enumerate}

  \item
    Let \( \dual\in\DUAL \). We have that 
    \begin{subequations}
      \begin{align*}
        \Polarity{\fonctionprimal}\np{\dual}
        &=
          \inf\defset{\lambda \in \OpenIntervalOpen{0}{+\infty} }{ \nscal{\primal}{\dual} \leq 
          \lambda\fonctionprimal\np{\primal} 
          \eqsepv \forall \primal \in \PRIMAL }
          \eqfinv 
          \intertext{by expression~\eqref{eq:circ-polar_transform_any_function} of the
          $\circ$-polar transform~$\Polarity{\fonctionprimal}$} 
          &=
          \inf\defset{\lambda \in \OpenIntervalOpen{0}{+\infty} }{ \nscal{\primal}{\dual}
            + 
            \lambda\bp{ -\fonctionprimal\np{\primal} } \leq 0
          \eqsepv \forall \primal \in \PRIMAL }
          \eqfinv 
        \\
        &=
          \inf\defset{\lambda \in \OpenIntervalOpen{0}{+\infty} }{ \sup_{\primal \in \PRIMAL}
          \Bp{ \nscal{\primal}{\frac{\dual}{\lambda}}
             -\fonctionprimal\np{\primal}  } \leq 0 }
          \eqfinv 
        \\
        &=
          \inf\defset{\lambda \in \OpenIntervalOpen{0}{+\infty} }{ \LFM{\fonctionprimal}\np{\frac{\dual}{\lambda}}
          \leq 0 }
          \eqfinv 
          \tag{by definition~\eqref{eq:Fenchel_conjugate} of the Fenchel conjugate~$ \LFM{\fonctionprimal}$} 
        \\
        &=
          \MinkowskiFunctional{\MidLevelSet{\LFM{\fonctionprimal}}{0}}\np{\dual}
          \eqfinp 
          \tag{by definition~\eqref{eq:MinkowskiFunctional} of the Minkowski functional}
      \end{align*}
    \end{subequations}
    Thus, we have proven~\eqref{eq:polar=Minkowski_functional}.
    As \( \MidLevelSet{\LFM{\fonctionprimal}}{0} \) is a closed convex subset,
    the function~\( \Polarity{\fonctionprimal} \colon \RR^{\spacedim} \to\barRR_{+} \)
    is a nonnegative (strictly positively) $1$-homogeneous convex function,
    by Item~\ref{it:positively_1-homogeneous_function_is_a_MinkowskiFunctional} in
    Proposition~\ref{pr:MinkowskiFunctional} (nonnegative (strictly positively) $1$-homogeneous),
    and by~\eqref{eq:MinkowskiFunctional_is_convex} (convex).
    The effective domain \( \domain\Polarity{\fonctionprimal}
    = \RR_{++} \MidLevelSet{\LFM{\fonctionprimal}}{0} \)
    by~\eqref{eq:MinkowskiFunctional_effective_domain}.



  \item 
    If \( \fonctionprimal\np{0}=0 \), then \( 0 \in  \dom\fonctionprimal \)
    and the  (Rockafellar-Moreau) subdifferential
    \( \subdifferential{}{\fonctionprimal}\np{0} \) in~\eqref{eq:Rockafellar-Moreau-subdifferential_a}
    can be expressed either as
    \begin{equation}
      \subdifferential{}{\fonctionprimal}\np{0}   
      =    \defset{ \dual\in\DUAL }{ %
        \LFM{\fonctionprimal}\np{\dual} 
        = \nscal{0}{\dual} -\fonctionprimal\np{0} }
      =    \defset{ \dual\in\DUAL }{ \LFM{\fonctionprimal}\np{\dual} 
        = 0 }
      = \MidLevelCurve{\LFM{\fonctionprimal}}{0}
      \eqfinv 
    \end{equation}
    or as (using the property that \( \LFM{\fonctionprimal}\np{\dual} 
    \geq \nscal{0}{\dual} -\fonctionprimal\np{0}=0 \) by definition~\eqref{eq:Fenchel_conjugate})
    \begin{equation}
      \subdifferential{}{\fonctionprimal}\np{0}   
      =    \defset{ \dual\in\DUAL }{ %
        \LFM{\fonctionprimal}\np{\dual} 
        \leq \nscal{0}{\dual} -\fonctionprimal\np{0} }
      =    \defset{ \dual\in\DUAL }{ \LFM{\fonctionprimal}\np{\dual} 
        \leq 0 }
      = \MidLevelSet{\LFM{\fonctionprimal}}{0}
      \eqfinp 
    \end{equation}

%


  \item
    This is a simple application of
    Item~\ref{it:polar_function_as_a_Minkowski_functional}.
    Indeed, we have that
    \begin{align*}
      \Polaritybi{\fonctionprimal}
      &=
        \PolarityReverse{\np{\Polarity{\fonctionprimal}}}
        \tag{by definition~\eqref{eq:circ-polar_bitransform_any_function} of the {bipolar transform}}
      \\
      &=
        \PolarityReverse{\np{\MinkowskiFunctional{\MidLevelSet{\LFM{\fonctionprimal}}{0}}}}
        \tag{by~\eqref{eq:polar=Minkowski_functional} in Item~\ref{it:polar_function_as_a_Minkowski_functional}.}
      \\
      &=
        \MinkowskiFunctional{\PolarSet{\MidLevelSet{\LFM{\fonctionprimal}}{0}}}
        \eqfinv
    \end{align*}
    by the expression~\eqref{eq:PolarityMinkowskiFunctionalDual_as_a_Minkowski_functional}
    of
    \(\PolarityReverse{\MinkowskiFunctional{\Dual}}=\MinkowskiFunctional{\PolarSet{\Dual}}\)
    (there is no circularity in the reasoning,
    as~\eqref{eq:PolarityMinkowskiFunctionalDual_as_a_Minkowski_functional} is proven by only
    using~\eqref{eq:polar=Minkowski_functional} 
    established before). 

    Thus, we have proven~\eqref{eq:bipolarasminkowski}.
    We also have that  \( 0\in \PolarSet{\MidLevelSet{\LFM{\fonctionprimal}}{0}} \), by
    definition~\eqref{eq:(negative)polar_set} of~\(
    \PolarSet{\MidLevelSet{\LFM{\fonctionprimal}}{0}} \),
    hence that
    \( \Polaritybi{\fonctionprimal}\np{0}=\MinkowskiFunctional{\PolarSet{\MidLevelSet{\LFM{\fonctionprimal}}{0}}}\np{0}
    = \inf\bset{\lambda> 0}{0\in\lambda\PolarSet{\MidLevelSet{\LFM{\fonctionprimal}}{0}}}
    = \inf\RR_{++}  =0 \).
  \end{enumerate}
  This ends the proof. 
\end{proof}

\subsubsection{Examples of polar transforms as Minkowski functionals}
\label{Examples_of_polar_transforms_as_Minkowski_functionals}

Using Item~\ref{it:polar_function_as_a_Minkowski_functional}
(Equation~\eqref{eq:polar=Minkowski_functional}) and
Item~\ref{it:bipolar_function_as_a_Minkowski_functional}
(Equation~\eqref{eq:bipolarasminkowski}) in
Proposition~\ref{pr:polar_bipolar_Minkowski_functional}, we obtain expressions
of the polar transforms of Minkowski functionals, of support functions, of
indicator functions and of generalized indicator functions as Minkowski
functionals.
Equation~\eqref{eq:PolarityMinkowskiFunctionalPrimal_as_a_Minkowski_functional}
can be found in \cite[Theorem~15.1, p.~128]{Rockafellar:1970}.

\begin{proposition}
  \label{pr:examples_of_polar_transforms_of_nonnegative_functions_as_Minkowski_functionals}
  \quad
  \begin{enumerate}
  \item
    Polar transform of a Minkowski functional as a Minkowski functional.\\
    \begin{subequations}
      For any subsets \( \Primal \subset \PRIMAL \) and $\Dual \subset \DUAL$, we have that 
      \begin{align}
        \Polarity{\MinkowskiFunctional{\Primal}}
        &=
          \MinkowskiFunctional{\PolarSet{\Primal}}
          \eqfinv
          \label{eq:PolarityMinkowskiFunctionalPrimal_as_a_Minkowski_functional}
        \\
        \PolarityReverse{\MinkowskiFunctional{\Dual}}
        &=
          \MinkowskiFunctional{\PolarSet{\Dual}}
          \eqfinv
          \label{eq:PolarityMinkowskiFunctionalDual_as_a_Minkowski_functional}
        \\
        \Polaritybi{\MinkowskiFunctional{\Primal}}
        &=
          \MinkowskiFunctional{\biPolarSet{\Primal}}
          \eqfinp
          \label{eq:biPolarityMinkowskiFunctionalPrimal_as_a_Minkowski_functional}
      \end{align}
    \end{subequations}
    
  \item
    Polar transform of a support function\footnote{%
      To the difference of Item~\ref{it:PolaritySupportFunction}
      in Proposition~\ref{pr:polar_transforms_of_nonnegative_functions_as_support_functions},
      the support functions that we consider here are not supposed to be nonnegative.}
    as a Minkowski functional.\\
    \begin{subequations}
      For any subsets \( \Primal \subset \PRIMAL \) and $\Dual \subset \DUAL$, we have that 
      \begin{align}
        \Polarity{\SupportFunction{\Dual}}
        &=
          \MinkowskiFunctional{\closedconvexhull\Dual}
          \eqfinv
          \label{eq:PolaritySupportFunctionDual_as_a_Minkowski_functional}
        \\
        \PolarityReverse{\SupportFunction{\Primal}}
        &=
          \MinkowskiFunctional{\closedconvexhull\Primal}
          \eqfinv
          \label{eq:PolaritySupportFunctionPrimal_as_a_Minkowski_functional}
        \\
        \Polaritybi{\SupportFunction{\Dual}}
        &=
          \MinkowskiFunctional{\PolarSet{\Dual}}
          \eqfinp
          \label{eq:biPolaritySupportFunctionDual_as_a_Minkowski_functional}
      \end{align}
    \end{subequations}

  \item 
    Polar transform of an indicator function as a Minkowski functional.\\
    For any subset \( \Primal \subset \PRIMAL \), we have that
    \begin{subequations}
      \begin{align}
        \Polarity{\Indicator{\Primal}}
        &=
          \MinkowskiFunctional{\PolarCone{\Primal}}
          = \Indicator{\PolarCone{\Primal}}
          \eqfinv 
          \label{eq:Polar_transform_of_a_indicator_function_as_a_Minkowski_functional}
        \\      
        \PolarityReverse{\Indicator{\Dual}}
        &=
          \MinkowskiFunctional{\PolarCone{\Dual}}
          = \Indicator{\PolarCone{\Dual}}
          \eqfinv 
          \label{eq:ReversePolar_transform_of_a_indicator_function_as_a_Minkowski_functional}
        \\      
        \Polaritybi{\Indicator{\Primal}}
        &=
          \MinkowskiFunctional{\biPolarCone{\Primal}}
          = \Indicator{\biPolarCone{\Primal}}
          \eqfinv 
          \label{eq:biPolar_transform_of_a_indicator_function_as_a_Minkowski_functional}
      \end{align}
    \end{subequations}

  \item 
    Polar transform of a generalized indicator function as a Minkowski functional.\\
    For any subset \( \Primal \subset \PRIMAL \), we have that
    \begin{subequations}
      \begin{align}
        \Polarity{\chi_{\Primal}}
        &=
          \MinkowskiFunctional{\PolarSet{\Primal}}
          \eqfinv
          \label{eq:Polar_transform_of_a_generalized_indicator_function_as_a_Minkowski_functional}
        \\      
        \PolarityReverse{\chi_{\Dual}}
        &=
          \MinkowskiFunctional{\PolarSet{\Dual}}
          \eqfinv
          \label{eq:ReversePolar_transform_of_a_generalized_indicator_function_as_a_Minkowski_functional}
        \\      
        \Polaritybi{\chi_{\Primal}}
        &=
          \MinkowskiFunctional{\biPolarSet{\Primal}}
          \eqfinp
          \label{eq:biPolar_transform_of_a_generalized_indicator_function_as_a_Minkowski_functional}
      \end{align}
    \end{subequations}
  \end{enumerate}
\end{proposition}

\begin{proof}
  \quad
  \begin{enumerate}
  \item 
    By~\eqref{eq:Polar_transform_of_Minkowski_functionals}, we know that
    \( \MidLevelSet{\LFM{\MinkowskiFunctional{\Primal}}}{0}
    = \PolarSet{\Primal} \).
    Then, 
    using~\eqref{eq:polar=Minkowski_functional}, 
    we obtain~\eqref{eq:PolarityMinkowskiFunctionalPrimal_as_a_Minkowski_functional}
    by \( \Polarity{\MinkowskiFunctional{\Primal}}
    = \MinkowskiFunctional{\MidLevelSet{\LFM{\MinkowskiFunctional{\Primal}}}{0}}
    = \MinkowskiFunctional{\PolarSet{\Primal}} \).

    Because the reverse polar
    transform~\eqref{eq:reverse_circ-polar_transform_any_function} acts like the
    polar transform~\eqref{eq:circ-polar_transform_any_function}, we
    obtain~\eqref{eq:PolarityMinkowskiFunctionalDual_as_a_Minkowski_functional} in
    the same fashion.

    Finally, by definition~\eqref{eq:circ-polar_bitransform_any_function} of the bipolar
    transform, we
    apply~\eqref{eq:PolarityMinkowskiFunctionalPrimal_as_a_Minkowski_functional} and
    then~\eqref{eq:PolarityMinkowskiFunctionalDual_as_a_Minkowski_functional} and
    get
    \( \Polaritybi{\MinkowskiFunctional{\Primal}}
    =\PolarityReverse{\np{\Polarity{\MinkowskiFunctional{\Primal}}}}
    =\PolarityReverse{\np{\MinkowskiFunctional{\PolarSet{\Primal}}}}
    =\MinkowskiFunctional{\biPolarSet{\Primal}} \).  Thus, we have
    obtained~\eqref{eq:biPolarityMinkowskiFunctionalPrimal_as_a_Minkowski_functional}.

  \item 
    We have that \( \LFM{\SupportFunction{\Dual}}=\Indicator{\closedconvexhull\Dual} \), from which
    we get that
    \( \MidLevelSet{\LFM{\SupportFunction{\Dual}}}{0}=\closedconvexhull\Dual \) and then, 
    by~\eqref{eq:polar=Minkowski_functional}, 
    we obtain~\eqref{eq:PolaritySupportFunctionDual_as_a_Minkowski_functional}
    by \(  \Polarity{\SupportFunction{\Dual}}
    = \MinkowskiFunctional{\MidLevelSet{\LFM{\SupportFunction{\Dual}}}{0}}
          = \MinkowskiFunctional{\closedconvexhull\Dual} \).

    Because the {reverse polar transform}~\eqref{eq:reverse_circ-polar_transform_any_function}
    acts like the $o$-polar transform~\eqref{eq:circ-polar_transform_any_function},
    we obtain~\eqref{eq:PolaritySupportFunctionPrimal_as_a_Minkowski_functional}
    in the same fashion.

    Finally, by definition~\eqref{eq:circ-polar_bitransform_any_function} of the bipolar
    transform, we apply~\eqref{eq:PolaritySupportFunctionDual_as_a_Minkowski_functional}
    and then~\eqref{eq:PolarityMinkowskiFunctionalDual_as_a_Minkowski_functional}
    and get
    \( \Polaritybi{\SupportFunction{\Dual}}
    =\PolarityReverse{\np{\Polarity{\SupportFunction{\Dual}}}}
    =\PolarityReverse{\MinkowskiFunctional{\closedconvexhull\Dual}}
    =\MinkowskiFunctional{\PolarSet{\np{\closedconvexhull\Dual}}}
    =\MinkowskiFunctional{\PolarSet{\Dual}}
    \), since \( \PolarSet{\np{\closedconvexhull\Dual}}= \PolarSet{\Dual} \).
    Thus, we have obtained~\eqref{eq:biPolaritySupportFunctionDual_as_a_Minkowski_functional}

  \item
    We have that
    \( \LFM{\Indicator{\Primal}} = \SupportFunction{\Primal} \), and hence 
    \( \MidLevelSet{\LFM{\Indicator{\Primal}}}{0}
    =  \MidLevelSet{\SupportFunction{\Primal}}{0} 
    = \PolarCone{\Primal} \)
    by~\eqref{eq:(negative)polar_cone}.
    By~\eqref{eq:polar=Minkowski_functional}, 
    we deduce that \(  \Polarity{\Indicator{\Primal}}
    =\MinkowskiFunctional{\MidLevelSet{\LFM{\Indicator{\Primal}}}{0}}
    =\MinkowskiFunctional{\PolarCone{\Primal}} \).
    Now, as \( \PolarCone{\Primal} \) is a cone, we easily see that 
    \(   \MinkowskiFunctional{\PolarCone{\Primal}}
    = \Indicator{\PolarCone{\Primal}} \)
    by definition~\eqref{eq:MinkowskiFunctional} of the Minkowski functional.
    Thus, we have
    proven~\eqref{eq:Polar_transform_of_a_indicator_function_as_a_Minkowski_functional}.

    Because the reverse polar
    transform~\eqref{eq:reverse_circ-polar_transform_any_function} acts like the
    polar transform~\eqref{eq:circ-polar_transform_any_function}, we
    obtain~\eqref{eq:ReversePolar_transform_of_a_indicator_function_as_a_Minkowski_functional}
    in the same fashion.

    Finally,
    using~\eqref{eq:Polar_transform_of_a_indicator_function_as_a_Minkowski_functional}
    and~\eqref{eq:ReversePolar_transform_of_a_indicator_function_as_a_Minkowski_functional},
    we
    get~\eqref{eq:biPolar_transform_of_a_indicator_function_as_a_Minkowski_functional}.
    
  \item
    As $\chi_{\Primal}=\Indicator{\Primal}+1$, we have that
    \(
    \LFM{\chi_{\Primal}}=\LFM{\np{\Indicator{\Primal}+1}}=\LFM{\Indicator{\Primal}}-1=
    \SupportFunction{\Primal}-1 \), and hence that

    \( \MidLevelSet{\LFM{\chi_{\Primal}}}{0}
    =  \MidLevelSet{\SupportFunction{\Primal}}{1}
    = \PolarSet{\Primal} \)
    by~\eqref{eq:(negative)polar_set}.
    By~\eqref{eq:polar=Minkowski_functional}, 
    we
    deduce~\eqref{eq:Polar_transform_of_a_generalized_indicator_function_as_a_Minkowski_functional}.

    Because the reverse polar transform~\eqref{eq:reverse_circ-polar_transform_any_function} acts like the
    polar transform~\eqref{eq:circ-polar_transform_any_function}, we
    obtain~\eqref{eq:ReversePolar_transform_of_a_generalized_indicator_function_as_a_Minkowski_functional}
    in the same fashion.
    
    Finally,
    using~\eqref{eq:Polar_transform_of_a_generalized_indicator_function_as_a_Minkowski_functional}
    and~\eqref{eq:PolarityMinkowskiFunctionalDual_as_a_Minkowski_functional}, we
    get~\eqref{eq:biPolar_transform_of_a_generalized_indicator_function_as_a_Minkowski_functional}.
  \end{enumerate}

  This ends the proof.
\end{proof}

Tables~\ref{tab:transforms} and~\ref{tab:bitransforms}
are consequences of
Propositions~\ref{pr:properties_circ-polarity_nonnegative_functions},
\ref{pr:polar_transforms_of_nonnegative_functions_as_support_functions},
\ref{pr:polar_bipolar_Minkowski_functional}
and
\ref{pr:examples_of_polar_transforms_of_nonnegative_functions_as_Minkowski_functionals}.

\begin{table}[hbtp]
  \centering
  \begin{tabular}{||c||c|c|c||}  
    \hline\hline 
    function
    & Fenchel conjugate 
    & $0$-level set of the 
    & $\circ$-polar transform 
    \\
    & & Fenchel conjugate & 
    \\
    \hline\hline
    $\fonctionprimal$
    & \( \LFM{\fonctionprimal} \)
    & \( \MidLevelSet{\LFM{\fonctionprimal}}{0} \)
    &  \( \Polarity{\fonctionprimal} =
      \MinkowskiFunctional{\MidLevelSet{\LFM{\fonctionprimal}}{0}} \)
    \\ &&& by~\eqref{eq:polar=Minkowski_functional}
    \\      
    \hline 
    $\SupportFunction{\Dual}$
    & $\Indicator{\closedconvexhull{\Dual}} $
    & $\closedconvexhull{\Dual}$
    & $\Polarity{\SupportFunction{\Dual}}
      =\MinkowskiFunctional{\closedconvexhull{\Dual}}$
    \\ &&& by~\eqref{eq:PolaritySupportFunctionDual_as_a_Minkowski_functional}
    \\
    \hline\hline
    $\fonctionprimal \geq 0 $
    & \( \LFM{\fonctionprimal} \)
    & \( \MidLevelSet{\LFM{\fonctionprimal}}{0} \ni 0 \)
    &  \( \Polarity{\fonctionprimal} =
      \MinkowskiFunctional{\MidLevelSet{\LFM{\fonctionprimal}}{0}} \)
      \( = \SupportFunction{\PolarSet{\MidLevelSet{\LFM{\fonctionprimal}}{0}}} \)
    \\ &&& by~\eqref{eq:polar=Minkowski_functional}
           and~\eqref{eq:properties_circ-polarity_nonnegative_functions_support_function}
    \\      
    \hline 
    $\SupportFunction{\Dual} \geq 0$
    & $\Indicator{\closedconvexhull{\Dual}} $
      & $\closedconvexhull{\Dual}$
    & $\Polarity{\SupportFunction{\Dual}}
      =\MinkowskiFunctional{\closedconvexhull{\Dual}}
      = \SupportFunction{\PolarSet{\Dual}} $
    \\ $0\in\closedconvexhull{\Dual}$
    &&&
        by~\eqref{eq:PolaritySupportFunctionDual_as_a_Minkowski_functional}
        and~\eqref{eq:PolaritySupportFunctionDual}
    \\
    \hline
    $\Indicator{\Primal}$
    & $\SupportFunction{\Primal}$
    & $\PolarCone{\Primal}$
    & $\Polarity{\Indicator{\Primal}}=
      \MinkowskiFunctional{\PolarCone{\Primal}} = 
      \SupportFunction{\biPolarCone{\Primal}}=
      \Indicator{\PolarCone{\Primal}}$
    \\ &&& by~\eqref{eq:Polar_transform_of_a_indicator_function_as_a_Minkowski_functional}
           and~\eqref{eq:Polar_transform_of_a_indicator_function_as_a_support_function}
    \\      
    \hline
    $\chi_{\Primal}$(=$\Indicator{\Primal}+1$)
    & $\SupportFunction{\Primal}-1$
    & $\PolarSet{\Primal}$
    & $ \Polarity{\chi_{\Primal}}=
      \MinkowskiFunctional{\PolarSet{\Primal}}
      = \SupportFunction{\biPolarSet{\Primal}} $
    \\ &&& by~\eqref{eq:Polar_transform_of_a_generalized_indicator_function_as_a_Minkowski_functional}
           and~\eqref{eq:Polar_transform_of_a_generalized_indicator_function_as_a_support_function}
    \\
    \hline 
    $\MinkowskiFunctional{\Primal}$
    & $\Indicator{\PolarSet{\Primal}} $
    & $\PolarSet{\Primal}$
    & $\Polarity{\MinkowskiFunctional{\Primal}}=
      \MinkowskiFunctional{\PolarSet{\Primal}}
      = \SupportFunction{\biPolarSet{\Primal}} $
    \\ &
         by~\eqref{eq:Polar_transform_of_Minkowski_functionals}
    &&
       by~\eqref{eq:PolarityMinkowskiFunctionalPrimal_as_a_Minkowski_functional}
       and by~\eqref{eq:PolarityMinkowskiFunctionalPrimal_as_a_support_function}
    \\
    \hline 
    $\MinkowskiFunctional{\PolarSet{\Dual}}$
    & $\Indicator{\biPolarSet{\Dual}} $
      & $\biPolarSet{\Dual}$
    & $\Polarity{\MinkowskiFunctional{\PolarSet{\Dual}}}=
      \MinkowskiFunctional{\biPolarSet{\Dual}}
      = \SupportFunction{\PolarSet{\Dual}} $
    \\ &
         by~\eqref{eq:Polar_transform_of_Minkowski_functionals}
    &&
       by~\eqref{eq:PolarityMinkowskiFunctionalPrimal_as_a_Minkowski_functional},
       \eqref{eq:PolarityMinkowskiFunctionalPrimal_as_a_support_function}
       and \( \triPolarSet{\Dual} = \PolarSet{\Dual} \)
    \\
   \hline\hline
  \end{tabular}
  \caption{Fenchel conjugates and polar transforms,
    for any function~$\fonctionprimal \colon \PRIMAL \to \barRR$,     
    any subset~$\Primal \subset \PRIMAL$ (in the primal space)
    and any subset~$\Dual \subset \DUAL$ (in the dual space)}
  \label{tab:transforms}
\end{table}

\begin{table}[hbtp]
  \centering
  \begin{tabular}{||c||c|c||}  
    \hline\hline 
    function
    & polar of $0$-level set 
    & $\circ$-polar bitransform 
    \\
    & of the Fenchel conjugate & 
    \\
    \hline\hline     
    $\fonctionprimal$
    & $\PolarSet{\MidLevelSet{\LFM{\fonctionprimal}}{0}}$ 
    &  \( \Polaritybi{\fonctionprimal} =
      \MinkowskiFunctional{\PolarSet{\MidLevelSet{\LFM{\fonctionprimal}}{0}}} \)
    \\ && by~\eqref{eq:bipolarasminkowski}
    \\      
    \hline 
    $\SupportFunction{\Dual}$
    & $\PolarSet{\Dual}$
    & $\Polaritybi{\SupportFunction{\Dual}}=
      \MinkowskiFunctional{\PolarSet{\Dual}}
      =\SupportFunction{\biPolarSet{\Dual}} $
    \\ && by~\eqref{eq:biPolaritySupportFunctionDual_as_a_Minkowski_functional}
          and~\eqref{eq:MinkowskiFunctional=SupportFunction_b} 
    \\
    \hline     \hline
    $\fonctionprimal \geq 0 $
    & $\PolarSet{\MidLevelSet{\LFM{\fonctionprimal}}{0}} \ni 0$ 
    &  \( \Polaritybi{\fonctionprimal} =
      \MinkowskiFunctional{\PolarSet{\MidLevelSet{\LFM{\fonctionprimal}}{0}}} \)
      \( = \SupportFunction{\MidLevelSet{\LFM{\fonctionprimal}}{0}} \)
    \\ && by~\eqref{eq:bipolarasminkowski}
          and~\eqref{eq:bipolarity_nonnegative_functions_support_function}
    \\
    \hline
    $\Indicator{\Primal}$
    & $\biPolarCone{\Primal}$
    & $ \Polaritybi{\Indicator{\Primal}}=
      \MinkowskiFunctional{\biPolarCone{\Primal}}
      = \SupportFunction{\PolarCone{\Primal}}
      =\Indicator{\biPolarCone{\Primal}} $
    \\ && by~\eqref{eq:biPolar_transform_of_a_indicator_function_as_a_Minkowski_functional}
          and~\eqref{eq:biPolar_transform_of_a_indicator_function_as_a_support_function}
    \\      
    \hline
    $\chi_{\Primal}$(=$\Indicator{\Primal}+1$) 
    & $\biPolarSet{\Primal}$
    & $\Polaritybi{\chi_{\Primal}}=
      \MinkowskiFunctional{\biPolarSet{\Primal}}$
      $=\SupportFunction{\PolarSet{\Primal}}$
    \\
    &&
       by~\eqref{eq:biPolar_transform_of_a_generalized_indicator_function_as_a_Minkowski_functional}
              and  by~\eqref{eq:biPolar_transform_of_a_generalized_indicator_function_as_a_support_function}
    \\
    \hline 
    $\MinkowskiFunctional{\Primal}$
    & $\biPolarSet{\Primal}$
    & $\Polaritybi{\MinkowskiFunctional{\Primal}}=
      \MinkowskiFunctional{\biPolarSet{\Primal}} $
      $=\SupportFunction{\PolarSet{\Primal}}$
    \\
    &&
       by~\eqref{eq:biPolarityMinkowskiFunctionalPrimal_as_a_Minkowski_functional}
       and by~\eqref{eq:biPolarityMinkowskiFunctionalPrimal_as_a_support_function}
    \\
    \hline 
    $\MinkowskiFunctional{\PolarSet{\Dual}}$
    & $\PolarSet{\Dual}$
    & $\Polaritybi{\MinkowskiFunctional{\PolarSet{\Dual}}}=
      \MinkowskiFunctional{\PolarSet{\Dual}} $
      $=\SupportFunction{\biPolarSet{\Dual}}$
    \\
    &&
       by~\eqref{eq:biPolarityMinkowskiFunctionalPrimal_as_a_Minkowski_functional},
       \( \triPolarSet{\Dual} = \PolarSet{\Dual} \)
       and~\eqref{eq:biPolarityMinkowskiFunctionalPrimal_as_a_support_function}
    \\
    \hline\hline
  \end{tabular}
  \caption{Fenchel biconjugates and bipolar transforms,
    for any function~$\fonctionprimal \colon \PRIMAL \to \barRR$,     
    any subset~$\Primal \subset \PRIMAL$ (in the primal space)
    and any subset~$\Dual \subset \DUAL$ (in the dual space)}
  \label{tab:bitransforms}
\end{table}

Tables~\ref{tab:transforms} and~\ref{tab:bitransforms} lead to the following results.
 \begin{proposition}
    For any subset~$\Primal \subset \PRIMAL$ (in the primal space)
 and any subset~$\Dual \subset \DUAL$ (in the dual space), we have that 
 \begin{subequations}
 \begin{equation}
   \MinkowskiFunctional{\PolarSet{\Primal}}
   = \SupportFunction{\biPolarSet{\Primal}}
   \quad\text{ and }\quad
    \MinkowskiFunctional{\biPolarSet{\Dual}}
    = \SupportFunction{\PolarSet{\Dual}}
    \eqfinv
      \label{eq:MinkowskiFunctional=SupportFunction_a}
  \end{equation}
 \begin{equation}
   \MinkowskiFunctional{\PolarSet{\Dual}}
   = \SupportFunction{\biPolarSet{\Dual}}
   \quad\text{ and }\quad
    \MinkowskiFunctional{\biPolarSet{\Primal}}
    = \SupportFunction{\PolarSet{\Primal}}
    \eqfinv
    \label{eq:MinkowskiFunctional=SupportFunction_b}
    \end{equation}
  which, in the case of a {polar pair} (see Definition~\ref{de:bipolar_set_dual_pair}),
  gives
  \begin{equation}
    \Primal, \Dual \mtext{ polar pair} \implies
   \MinkowskiFunctional{\Primal}
   = \SupportFunction{\Dual}
   \quad\text{ and }\quad
    \MinkowskiFunctional{\Dual}
    = \SupportFunction{\Primal}
    \eqfinp
    \label{eq:MinkowskiFunctional=SupportFunction_c}
  \end{equation}
  \label{eq:MinkowskiFunctional=SupportFunction}
\end{subequations}
For any nonnegative (strictly positively) $1$-homogeneous function
\( \fonctionprimal \colon \PRIMAL \to {\barRR}_{+} \),
we have that
\begin{subequations}
  \begin{align}
    \MidLevelSet{\LFM{\fonctionprimal}}{0}
    &=
      \PolarSet{\MidStrictLevelSet{\fonctionprimal}{1}}
      \eqfinv
      \label{eq:PolarSet_MidStrictLevelSet_positively_1-homogeneous_function}
            \intertext{and, when~\( \fonctionprimal\np{0}=0 \), we have that}
     \MidLevelSet{\LFM{\fonctionprimal}}{0}
    &=
      \PolarSet{\MidStrictLevelSet{\fonctionprimal}{1}}
     =
      \PolarSet{\MidLevelSet{\fonctionprimal}{1}}
      \eqfinp
            \label{eq:PolarSet_MidLevelSet_positively_1-homogeneous_function}
  \end{align}
\end{subequations}
\end{proposition}

\begin{proof}
  The last two lines
  of Table~\ref{tab:transforms}
  give~\eqref{eq:MinkowskiFunctional=SupportFunction},
  which are well-known results
  (see \cite[Theorem~14.5, p.~125]{Rockafellar:1970},
  \cite[Corollaries~15.1.1-2, p.~129]{Rockafellar:1970}). 

  On the one hand, by line~4 of Table~\ref{tab:bitransforms}, we have that 
  \( \Polaritybi{\fonctionprimal} =
  \SupportFunction{\MidLevelSet{\LFM{\fonctionprimal}}{0}} \).
  On the other hand, by~\eqref{eq:positively_1-homogeneous_function_is_a_MinkowskiFunctional_strict},
  we have that \( \fonctionprimal =
  \MinkowskiFunctional{\MidStrictLevelSet{\fonctionprimal}{1}} \).
  Then, by line~7 of Table~\ref{tab:bitransforms}
  with \( \Primal=\MidStrictLevelSet{\fonctionprimal}{1}\), we get that 
  \( \Polaritybi{\fonctionprimal} =
  \SupportFunction{\PolarSet{\MidStrictLevelSet{\fonctionprimal}{1}}} \).
  As both sets \( \MidLevelSet{\LFM{\fonctionprimal}}{0} \) and
  \( \PolarSet{\MidStrictLevelSet{\fonctionprimal}{1}} \) are closed convex, we
  obtain Equation~\eqref{eq:PolarSet_MidStrictLevelSet_positively_1-homogeneous_function}
  from the equality \( \Polaritybi{\fonctionprimal} =
  \SupportFunction{\MidLevelSet{\LFM{\fonctionprimal}}{0}} =
  \SupportFunction{\PolarSet{\MidStrictLevelSet{\fonctionprimal}{1}}} \).
  
  When \( \fonctionprimal\np{0}=0 \), by
  Equation~\eqref{eq:positively_1-homogeneous_function_is_a_MinkowskiFunctional_both}
  in Proposition~\ref{pr:MinkowskiFunctional}, we have that
  \( \fonctionprimal = \MinkowskiFunctional{\MidLevelSet{\fonctionprimal}{1}} \).
  Equation~\eqref{eq:PolarSet_MidLevelSet_positively_1-homogeneous_function}
  follows in the same way as above.  
\end{proof}

\section{Bipolar functions}
\label{Generalized_circ-convex_functions}

In this Sect.~\ref{Generalized_circ-convex_functions}, 
we consider $\PRIMAL$ and $\DUAL$ two (real) vector spaces that are {paired}
(see \S\ref{Dual_pair,_paired_vector_spaces}).

We present in a systematic fashion different expressions for the set of
bipolar functions defined as follows.
%
\begin{definition}
  We say that a function $\fonctionprimal \colon \PRIMAL \to \barRR$
  is a \emph{bipolar function} if
  \begin{equation}
    \Polaritybi{\fonctionprimal} = \fonctionprimal
    \eqfinp
    \label{eq:bipolar_function}
  \end{equation}
We denote by $\BIPOLARFUNC$ the set of bipolar functions from~$\PRIMAL$ to the extended reals. 
\label{de:bipolar_function}
\end{definition}

The following equivalences --- between
Item~\ref{it:Generalized_circ-convex_functions}
and all but the last Item~\ref{it:Corollary4.3_Martinez-Legaz-Singer:1994} ---
have been long established in the class of gauges (see \cite[Theorem~15.1, Corollary~15.1,
Corollary~15.2]{Rockafellar:1970}).
The equivalence between 
Item~\ref{it:Generalized_circ-convex_functions}
and Item~\ref{it:Generalized_circ-convex_functions_lsc_gauge}
has been established 
in~\cite[Theorem~5.2]{Martinez-Legaz-Singer:1994}
in the class of nonnegative functions.
The equivalence between 
Item~\ref{it:Generalized_circ-convex_functions}
and Item~\ref{it:Corollary4.3_Martinez-Legaz-Singer:1994}
has been established in \cite[Corollary~4.3]{Martinez-Legaz-Singer:1994},
for a large class of couplings which includes the coupling
\( \PRIMAL\times\DUAL \ni \np{\primal,\dual} \mapsto \coupling\np{\primal,\dual}={\nscal{\primal}{\dual}}_+ \).
Once again, the novelty of
Proposition~\ref{pr:Generalized_circ-convex_functions} is to 
provide equivalence in the class of all functions
(and not necessarily convex ones,
or lsc, or (strictly positively) $1$-homogeneous, or gauges, or even nonnegative),
and to display the equivalences in a unified fashion.
Item~\ref{it:bipolar_function=Minkowski=SupportFunction} is possibly new.

\begin{proposition}
  For any function $\fonctionprimal \colon \PRIMAL \to \barRR$,
  the following statements are equivalent.
  \begin{enumerate}
    
  \item
    The function~\( \fonctionprimal \) is a bipolar function, that is, 
    \( \Polaritybi{\fonctionprimal} = \fonctionprimal \).
    \label{it:Generalized_circ-convex_functions}
    
  \item
    \label{it:bipolar_function=Minkowski=SupportFunction}
    The function~\( \fonctionprimal \) is the Minkowski functional
    of the \BipolarSet~\( \PolarSet{\MidLevelSet{\LFM{\fonctionprimal}}{0}} \)
  and also the support function of the \BipolarSet~\( \MidLevelSet{\LFM{\fonctionprimal}}{0} \):
\begin{equation}
  {\fonctionprimal} =
      \MinkowskiFunctional{\PolarSet{\MidLevelSet{\LFM{\fonctionprimal}}{0}}}
      = \SupportFunction{\MidLevelSet{\LFM{\fonctionprimal}}{0}}
      \eqfinp
      \label{eq:bipolar_function=Minkowski=SupportFunction}
    \end{equation}
    
 \item
   \label{it:Generalized_circ-convex_functions_two_subsets}
    There exist a polar pair \( \Primal\subset\PRIMAL \),
    \( \Dual\subset\DUAL \) 
(that is, \( \Dual=\PolarSet{\Primal} \), \(  \Primal=\PolarSet{\Dual} \))    
    such that \( \fonctionprimal = \MinkowskiFunctional{\Primal}
    =\SupportFunction{\Dual} \).
    
  \item
    \label{it:Generalized_circ-convex_functions_MinkowskiFunctional}
  There exists a bipolar (primal) set~\( \Primal\subset\PRIMAL \) 
    such that the function~\( \fonctionprimal \)
    is the Minkowski functional~\( \MinkowskiFunctional{\Primal} \).
  
  \item
    There exists a bipolar (dual) set~\( \Dual\subset\DUAL \) 
    such that the function~\( \fonctionprimal \)
    is the support function~\( \fonctionprimal =\SupportFunction{\Dual} \).
    \label{it:Generalized_circ-convex_functions_support_function}

    \item
    \label{it:Generalized_circ-convex_functions_lsc_gauge}
   The function~\( \fonctionprimal \)
    is a nonnegative (strictly positively) $1$-homogeneous convex lsc function
    satisfying \( \fonctionprimal\np{0}=0 \), that is,
    the function~\( \fonctionprimal \) is a lsc gauge.

  \item
    \label{it:Corollary4.3_Martinez-Legaz-Singer:1994}
    The function~\( \fonctionprimal \)
    is the pointwise supremum of a family of functions of the form
    \( \lambda {\nscal{\cdot}{\dual}}_+ \), with \( \dual\in\DUAL \)
    and \( \lambda \in \RR_{++}
    \).
  \end{enumerate}
  \label{pr:Generalized_circ-convex_functions} 
\end{proposition}

\begin{proof}
  
  \noindent$\bullet$
We prove that Item~\ref{it:Generalized_circ-convex_functions} implies
  Item~\ref{it:bipolar_function=Minkowski=SupportFunction}.

  Suppose that \( \Polaritybi{\fonctionprimal} = \fonctionprimal \). 
  By~\eqref{eq:circ-polar_bitransform_any_function}
  and~\eqref{eq:circ-polar_transform_any_function},
  we get that \( \Polaritybi{\fonctionprimal} \geq 0 \),
  hence \( \fonctionprimal \geq 0 \).
  By line~4 in Table~\ref{tab:bitransforms}
  (or by~\eqref{eq:bipolarasminkowski}
  and~\eqref{eq:bipolarity_nonnegative_functions_support_function}),
  we get~\eqref{eq:bipolar_function=Minkowski=SupportFunction}.
 By Item~\ref{it:properties_circ-polarity_nonnegative_functions_support_function}
 in Proposition~\ref{pr:properties_circ-polarity_nonnegative_functions}, we know that
 \( \MidLevelSet{\LFM{\fonctionprimal}}{0} \) is a \BipolarSet,
 hence so is~\( \PolarSet{\MidLevelSet{\LFM{\fonctionprimal}}{0}} \).
  \medskip

  \noindent$\bullet$
  It is obvious that Item~\ref{it:bipolar_function=Minkowski=SupportFunction} implies
  Item~\ref{it:Generalized_circ-convex_functions_two_subsets}, which implies
  Item~\ref{it:Generalized_circ-convex_functions_MinkowskiFunctional}
  and
  Item~\ref{it:Generalized_circ-convex_functions_support_function}
  (Item~\ref{it:Generalized_circ-convex_functions_MinkowskiFunctional}
  and
  Item~\ref{it:Generalized_circ-convex_functions_support_function}
  are equivalent, using~\eqref{eq:MinkowskiFunctional=SupportFunction_c}).
    \medskip

    \noindent$\bullet$
    Item~\ref{it:Generalized_circ-convex_functions_support_function} implies
    Item~\ref{it:Generalized_circ-convex_functions_lsc_gauge},
    because the support function of a subset containing~$0$ is
    a nonnegative (strictly positively)
  $1$-homogeneous convex lsc function satisfying \( \fonctionprimal\np{0}=0 \).

  \medskip
  
  \noindent$\bullet$
  We prove that Item~\ref{it:Generalized_circ-convex_functions_lsc_gauge}
  implies Item~\ref{it:Generalized_circ-convex_functions_MinkowskiFunctional}.

  Suppose that
  the function~\( \fonctionprimal \)
  is a nonnegative (strictly positively) $1$-homogeneous convex lsc function
  satisfying \( \fonctionprimal\np{0}=0 \).
  By~\eqref{eq:positively_1-homogeneous_function_is_a_MinkowskiFunctional_both}, we know that
  \( \fonctionprimal =
  \MinkowskiFunctional{\MidLevelSet{\fonctionprimal}{1}} \),
  where \( \MidLevelSet{\fonctionprimal}{1} \)
  is a closed convex subset containing zero, hence is a \BipolarSet,
  by Item~\ref{it:bipolar_set_def_1} of Definition~\ref{de:bipolar_set_dual_pair}.  
  \medskip
  

  \noindent$\bullet$
Item~\ref{it:Generalized_circ-convex_functions_MinkowskiFunctional}
 implies
 Item~\ref{it:Generalized_circ-convex_functions}, using~\eqref{eq:biPolarityMinkowskiFunctionalPrimal_as_a_Minkowski_functional}.

 \medskip
  


  \noindent$\bullet$ Finally, the equivalence between
  Item~\ref{it:Generalized_circ-convex_functions} and
  Item~\ref{it:Corollary4.3_Martinez-Legaz-Singer:1994} is an application of
  \cite[Corollary~4.3]{Martinez-Legaz-Singer:1994} to the case of the coupling
  \( \coupling\np{\primal,\dual}={\nscal{\primal}{\dual}}_+ \).
  \medskip

  This ends the proof. 
\end{proof}

The one-to-one correspondence between \BipolarSet s and lsc gauges
(or, equivalently, \BipolarFunction s) is outlined 
at the beginning of \cite[Section~15]{Rockafellar:1970}.
We show that this correspondence is an isomorphism between lattices.
To the best of our knowledge, this result is new
(and different from \cite[p.~292]{Aliprantis-Border:2006}, which points out an
isomorphism between the lattice of weak~$*$ compact convex subsets of~$\DUAL$
and the lattice of continuous gauges on~$\PRIMAL$).

\begin{theorem}
  \label{th:isomorphism_between_lattices}
  The set of \BipolarFunction s, ordered by~$\leq$, is a lattice
  \( \np{\BIPOLARFUNC,\wedge,\vee} \).
  Consider two \BipolarFunction s $\fonctionprimal \colon \PRIMAL \to \barRR_{+}$
  and $\fonctiondual \colon \PRIMAL \to \barRR_{+}$.
  The \emph{greatest lower bound $\fonctionprimal\wedge\fonctiondual$} is given by
  \begin{subequations}
    \begin{align}
      \fonctionprimal\wedge\fonctiondual
      &=
        \SupportFunction{
        \MidLevelSet{\LFM{\fonctionprimal}}{0}
        \cap \MidLevelSet{\LFM{\fonctiondual}}{0} }
        =
        \MinkowskiFunctional{ \closedconvexhull\np{
        \PolarSet{\MidLevelSet{\LFM{\fonctionprimal}}{0}}
        \cup \PolarSet{\MidLevelSet{\LFM{\fonctiondual}}{0}} } }
        \eqfinv 
        \label{eq:BipolarFunction_wedge_SupportFunction}
        %
        %
        \intertext{whereas the \emph{least upper bound
        $\fonctionprimal\vee\fonctiondual$} is given by the supremum}
        \fonctionprimal\vee\fonctiondual
      &=
        \sup\na{\fonctionprimal,\fonctiondual}
        \eqfinp
        \label{eq:BipolarFunction_vee_MinkowskiFunctional}
    \end{align}
  \end{subequations}
  The two mappings
  \begin{align}
    \varphi: \BIPOLARSET
    &\to \BIPOLARFUNC
    &
      \theta: \BIPOLARFUNC
    &\to \BIPOLARSET
      \label{eq:isomorphism_between_lattices}
    \\
    \Primal\; &\mapsto \MinkowskiFunctional{\Primal}=
                \SupportFunction{\PolarSet{\Primal}} 
    & \fonctionprimal\quad
    &\mapsto \PolarSet{\MidLevelSet{\LFM{\fonctionprimal}}{0}}
      \nonumber
  \end{align}
  define an isomorphism between
  the lattice \( \np{\BIPOLARSET,\wedge,\vee} \) of \BipolarSet s
  (see Definition~\ref{de:bipolar_set_dual_pair} and
  Proposition~\ref{pr:BipolarSet_lattice})
  and 
  the lattice \( \np{\BIPOLARFUNC,\wedge,\vee} \) of \BipolarFunction s. 
  \label{th:BipolarFunction_lattice}  
\end{theorem}

\begin{proof}

\noindent$\bullet$ First, we prove that the set of \BipolarFunction s, ordered by~$\leq$, is a lattice
\( \np{\BIPOLARFUNC,\wedge,\vee} \).
\medskip

\noindent$\triangleright$
We prove~\eqref{eq:BipolarFunction_wedge_SupportFunction}
for two \BipolarFunction s $\fonctionprimal \colon \PRIMAL \to \barRR_{+}$
and $\fonctiondual \colon \PRIMAL \to \barRR_{+}$.
The greatest lower bound $\fonctionprimal\wedge\fonctiondual$ 
is the greatest \BipolarFunction\ below~$\inf\na{\fonctionprimal,\fonctiondual}$,
hence is \( \Polaritybi{\bp{\inf\na{\fonctionprimal,\fonctiondual}}} \)
(using that, for any $\fonctionprimalbis \colon \PRIMAL \to \barRR_{+}$,
we have that \( \Polaritybi{\fonctionprimalbis} \leq \fonctionprimalbis \) by 
the Inequality~\eqref{eq:bipolar_is_smaller}, where \( \Polaritybi{\fonctionprimalbis} \)
is a bipolar function).
We have that
\begin{align}
  \fonctionprimal\wedge\fonctiondual
  &=
    \Polaritybi{\bp{\inf\na{\fonctionprimal,\fonctiondual}}}
    \label{eq:fonctionprimal_wedge_fonctiondual=Polaritybi}
  \\
  &=  
    \Polaritybi{\bp{\inf\na{
    \SupportFunction{ \MidLevelSet{\LFM{\fonctionprimal}}{0} },
    \SupportFunction{ \MidLevelSet{\LFM{\fonctiondual}}{0} }
    }}}
    \tag{as 
    \( {\fonctionprimal} = \SupportFunction{\MidLevelSet{\LFM{\fonctionprimal}}{0}} \)
    and
    \( {\fonctiondual} = \SupportFunction{\MidLevelSet{\LFM{\fonctiondual}}{0}} \)
    by~\eqref{eq:bipolar_function=Minkowski=SupportFunction}}
    \nonumber \\
  &=
    \PolarityReverse{\np{ \sup\na{
    \Polarity{\np{\SupportFunction{ \MidLevelSet{\LFM{\fonctionprimal}}{0}}}},
    \Polarity{\np{\SupportFunction{ \MidLevelSet{\LFM{\fonctiondual}}{0}}}}
    }} }
    \tag{by~\eqref{eq:the_circ-polarity_is_a_times-duality_duality}}
    \nonumber \\
  &=
    \PolarityReverse{\np{ \sup\na{
    \SupportFunction{\PolarSet{\MidLevelSet{\LFM{\fonctionprimal}}{0}}},
    \SupportFunction{\PolarSet{\MidLevelSet{\LFM{\fonctiondual}}{0}}}
    }}}
    \tag{by~\eqref{eq:PolaritySupportFunctionDual}}
    \nonumber \\
  &=
     \PolarityReverse{\np{ 
     \SupportFunction{\PolarSet{\MidLevelSet{\LFM{\fonctionprimal}}{0}}
     \cup \PolarSet{\MidLevelSet{\LFM{\fonctiondual}}{0}}} }}
    \tag{as is well known for support functions}
     \nonumber \\
  &=
     \SupportFunction{\PolarSet{\np{
     \PolarSet{\MidLevelSet{\LFM{\fonctionprimal}}{0}}
     \cup \PolarSet{\MidLevelSet{\LFM{\fonctiondual}}{0}}} }}
      \tag{by~\eqref{eq:PolaritySupportFunctionPrimal}}
     \nonumber \\
  &=  \SupportFunction{
    \MidLevelSet{\LFM{\fonctionprimal}}{0}
    \cap \MidLevelSet{\LFM{\fonctiondual}}{0} }
    \tag{by~\eqref{eq:polar_of_union}
    as $\MidLevelSet{\LFM{\fonctionprimal}}{0}$ and
    $\MidLevelSet{\LFM{\fonctiondual}}{0}$ are \BipolarSet s}
    \eqfinp
\end{align}
Thus, we have proved the left equality in~\eqref{eq:BipolarFunction_wedge_SupportFunction}.
The right equality
is a consequence of~\eqref{eq:MinkowskiFunctional=SupportFunction_c} and~\eqref{eq:polar_of_intersection}.
\medskip

\medskip

\noindent$\triangleright$ We prove~\eqref{eq:BipolarFunction_vee_MinkowskiFunctional}.
Consider two \BipolarFunction s $\fonctionprimal \colon \PRIMAL \to \barRR_{+}$
and $\fonctiondual \colon \PRIMAL \to \barRR_{+}$.
We are going to show that the supremum
\( \fonctionprimalbis=\sup\na{\fonctionprimal,\fonctiondual} \) is a
\BipolarFunction.  Indeed, on the one hand, we have that
\( \Polaritybi{\fonctionprimalbis} \geq
\sup\na{\Polaritybi{\fonctionprimal},\Polaritybi{\fonctiondual}} =
\sup\na{\fonctionprimal,\fonctiondual}=\fonctionprimalbis \), where we have
first used that the bipolar operation is isotone, and second that both
$\fonctionprimal$ and $\fonctiondual$ are \BipolarFunction s (hence
\( \Polaritybi{\fonctionprimal}={\fonctionprimal}\) and
\( \Polaritybi{\fonctiondual}={\fonctiondual}\) by
Definition~\ref{de:bipolar_function}).  Now, on the other hand, we have that
\( \Polaritybi{\fonctionprimalbis} \leq \fonctionprimalbis \) by the
Inequality~\eqref{eq:bipolar_is_smaller}.  We conclude that
\( \Polaritybi{\fonctionprimalbis} = \fonctionprimalbis \), that is,
\( \fonctionprimalbis=\sup\na{\fonctionprimal,\fonctiondual} \) is a
\BipolarFunction.  Thus, the least upper bound
$\fonctionprimal\vee\fonctiondual= \sup\na{\fonctionprimal,\fonctiondual}$, which
is~\eqref{eq:BipolarFunction_vee_MinkowskiFunctional}.

\bigskip

\noindent$\bullet$ Second, we show that the two mappings~\eqref{eq:isomorphism_between_lattices}
define a one-to-one correspondence between \BipolarSet s and \BipolarFunction s.

We show that the mapping~$\varphi$ takes values in~$\BIPOLARFUNC$.
Indeed, when $\Primal$ is a bipolar set, 
we obtain by line~6 column~3 of Table~\ref{tab:bitransforms}
that $\Polaritybi{\varphi(\Primal)}=\Polaritybi{\MinkowskiFunctional{\Primal}}=
\MinkowskiFunctional{\biPolarSet{\Primal}}=\MinkowskiFunctional{\Primal}=
\varphi(\Primal)$, and thus $\varphi(\Primal)\in \BIPOLARFUNC$.
It is immediate to check that the mapping~$\theta$ takes values in~$\BIPOLARSET$.

Now, we have that $\theta \compo \varphi = \identity_{\BIPOLARSET}$ as,
for $\Primal \in \BIPOLARSET$, line~6 column~2 of Table~\ref{tab:bitransforms}
gives
\begin{subequations}
  \begin{equation}
    (\theta \compo \varphi)(\Primal)
    =
    \theta \np{\MinkowskiFunctional{\Primal}}
    = \PolarSet{\MidLevelSet{\LFM{\MinkowskiFunctional{\Primal}}}{0}}
    = \biPolarSet{\Primal} = \Primal
    \eqfinp
  \end{equation}
  We have that $\varphi \compo \theta = \identity_{\BIPOLARFUNC}$ as,
  for $\fonctionprimal \in \BIPOLARFUNC$, 
  line~1 column~3 of Table~\ref{tab:bitransforms} gives
  \begin{equation}
    (\varphi \compo \theta)(\fonctionprimal)
    =
    \varphi \np{\PolarSet{\MidLevelSet{\LFM{\fonctionprimal}}{0}}}
    = \MinkowskiFunctional{\PolarSet{\MidLevelSet{\LFM{\fonctionprimal}}{0}}}
    = \Polaritybi{\fonctionprimal} = \fonctionprimal
    \eqfinp
  \end{equation}  
\end{subequations}
\bigskip

\noindent$\bullet$ Third, we show that the two mappings~\eqref{eq:isomorphism_between_lattices}
define an isomorphism between lattices.
For this purpose, we consider two \BipolarSet s
$\Primal$ and $\Primal'$ in $\BIPOLARSET$.
On the one hand, we have that 
\begin{align*}
  \varphi( \Primal \wedge \Primal')
  &=
    \varphi( \Primal \cap \Primal')
    \tag{by~\eqref{eq:lattice_set_wedge}}
    \\
  &=
    \SupportFunction{\PolarSet{\np{\Primal \cap \Primal'}}}
     \tag{by~\eqref{eq:isomorphism_between_lattices}}
  %
  \\
  &=
    \SupportFunction{\closedconvexhull{\np{\PolarSet{\Primal} \cup \PolarSet{\Primal'}}}}
    \tag{by~\eqref{eq:polar_of_intersection}}
  \\
  &=
    \SupportFunction{\PolarSet{\Primal} \cup \PolarSet{\Primal'}}
    \tag{as is well known for support functions}
  \\
  &  =
    \sup
    \na{ \SupportFunction{\PolarSet{\Primal}}, \SupportFunction{\PolarSet{\Primal'}}}
    \tag{as is well known for support functions}
  \\
  &=
    \SupportFunction{\PolarSet{\Primal}}\vee\SupportFunction{\PolarSet{\Primal'}}
    \tag{by~\eqref{eq:BipolarFunction_vee_MinkowskiFunctional}}
  \\
  &=
    \varphi( \Primal) \vee  \varphi(\Primal')
    \eqfinp
    \tag{by~\eqref{eq:isomorphism_between_lattices}}
\end{align*}
On the other hand, we have that 
\begin{align*}
  \varphi\np{\Primal \vee \Primal'}
  &=
    \varphi\bp{ \closedconvexhull\np{ \Primal \cup \Primal'}}
    \tag{by~\eqref{eq:lattice_set_vee}}
   \\
  &=
    \SupportFunction{\PolarSet{\bp{\closedconvexhull\np{ \Primal \cup \Primal'}}}}
    \tag{by~\eqref{eq:isomorphism_between_lattices}}
  \\
  &=
    \SupportFunction{\PolarSet{\Primal} \cap \PolarSet{\Primal'}}
    \tag{using $\PolarSet{\bp{\closedconvexhull\np{ \Primal \cup \Primal'}}}=
    \PolarSet{\np{\Primal \cup \Primal'}}$ and~\eqref{eq:polar_of_union}}
  \\
  &=
    \SupportFunction{
    \MidLevelSet{\LFM{\SupportFunction{\PolarSet{\Primal}}}}{0} 
    \cap \MidLevelSet{\LFM{\SupportFunction{\PolarSet{\Primal'}}}}{0} }
    \tag{as $\MidLevelSet{\LFM{\SupportFunction{\PolarSet{Z}}}}{0}
    =\MidLevelSet{\Indicator{\PolarSet{Z}}}{0}
    =\PolarSet{Z}$}
  \\
  &=
    \SupportFunction{\PolarSet{\Primal}}
     \wedge  \SupportFunction{\PolarSet{\Dual}}
       \tag{by~\eqref{eq:BipolarFunction_wedge_SupportFunction}}
   \\
  &=
    \varphi( \Primal) \wedge \varphi(\Primal')
    \eqfinp
    \tag{by~\eqref{eq:isomorphism_between_lattices}}
\end{align*}
This concludes the proof.
\end{proof}

Recall that the \emph{infimal convolution} (or \emph{inf~convolution})
$\fonctionprimal \infcv \fonctiondual$
of two functions
$\fonctionprimal \colon \PRIMAL \to \barRR$
and $\fonctiondual \colon \PRIMAL \to \barRR$
is the function defined by
\begin{equation}
  \epigraph_{s}\np{\fonctionprimal \infcv \fonctiondual} = 
  \epigraph_{s}\fonctionprimal +   \epigraph_{s}\fonctiondual
  \eqfinp 
  \label{eq:inf_convolution}
\end{equation}

\begin{proposition}
Consider two \BipolarFunction s $\fonctionprimal \colon \PRIMAL \to \barRR_{+}$
and $\fonctiondual \colon \PRIMAL \to \barRR_{+}$.
  The function $\fonctionprimal \wedge \fonctiondual$
 is related to the infimal convolution
  $\fonctionprimal \infcv \fonctiondual$, as we have
  \begin{equation}
    \fonctionprimal \wedge \fonctiondual
    \le \fonctionprimal \infcv \fonctiondual
    \le \inf\na{\fonctionprimal,\fonctiondual}
    \eqfinp
    \label{eq:infconvol_and_wedge}
  \end{equation}
  We deduce from Equation~\eqref{eq:infconvol_and_wedge}
  that, when $\fonctionprimal \infcv \fonctiondual$ is lsc,
  it coincides with $\fonctionprimal \wedge \fonctiondual$.  
\end{proposition}

\begin{proof}
  We have that 
  \begin{align*}
    \LFMbi{\np{\fonctionprimal \infcv \fonctiondual}}
    &=
      \LFMr{
      \np{\LFM{\fonctionprimal} + \LFM{\fonctiondual}}}
      \tag{by~\cite[Theorem 2.3.1 (ix)]{Zalinescu:2002}}
    \\
    &=
      \LFMr{
      \np{\LFM{\SupportFunction{\MidLevelSet{\LFM{\fonctionprimal}}{0}}} +
      \LFM{\SupportFunction{\MidLevelSet{\LFM{\fonctiondual}}{0}}}}}
      \tag{as 
      \( {\fonctionprimal} = \SupportFunction{\MidLevelSet{\LFM{\fonctionprimal}}{0}} \)
      and
      \( {\fonctiondual} = \SupportFunction{\MidLevelSet{\LFM{\fonctiondual}}{0}} \)
      by~\eqref{eq:bipolar_function=Minkowski=SupportFunction}}
    \\
    &=
      \LFMr{
      \np{\Indicator{\MidLevelSet{\LFM{\fonctionprimal}}{0}} +\Indicator{\MidLevelSet{\LFM{\fonctiondual}}{0}}}}
      \intertext{as both $\MidLevelSet{\LFM{\fonctionprimal}}{0}$ and $\MidLevelSet{\LFM{\fonctiondual}}{0}$ are closed convex sets,
      using Item~\ref{it:Generalized_circ-convex_functions_lsc_gauge} in Proposition~\ref{pr:Generalized_circ-convex_functions}}
    &=\LFMr{\np{
      \Indicator{\MidLevelSet{\LFM{\fonctionprimal}}{0} \cap \MidLevelSet{\LFM{\fonctiondual}}{0}}}}
    \\
    &=
      \SupportFunction{\MidLevelSet{\LFM{\fonctionprimal}}{0} \cap \MidLevelSet{\LFM{\fonctiondual}}{0}}
    \\
    &= \fonctionprimal \wedge \fonctiondual
      \tag{by~\eqref{eq:BipolarFunction_wedge_SupportFunction}}
       \\
     &= \Polaritybi{\bp{\inf\na{\fonctionprimal,\fonctiondual}}}
       \tag{by~\eqref{eq:fonctionprimal_wedge_fonctiondual=Polaritybi}}
      \eqfinp
  \end{align*}
  We therefore get that
  \begin{align}
    \fonctionprimal \wedge \fonctiondual =
    \Polaritybi{\bp{\inf\na{\fonctionprimal,\fonctiondual}}}
    =
    \LFMbi{\np{\fonctionprimal \infcv \fonctiondual}}
    \le \fonctionprimal \infcv \fonctiondual
    \le \inf\na{\fonctionprimal,\fonctiondual}
    \eqfinv
  \end{align}
  as $\fonctionprimal \infcv \fonctiondual \le \inf\na{\fonctionprimal,\fonctiondual}$.
  Indeed, we have that
  \begin{align*}
    \epigraph_{s}\np{ \inf\na{\fonctionprimal,\fonctiondual} }
    &=
      \epigraph_{s}\fonctionprimal \cup  \epigraph_{s}\fonctiondual
      \subset 
      \epigraph_{s}\fonctionprimal + \epigraph_{s}\fonctiondual
      \eqfinv
  \end{align*}
  because $\np{0,0} \in \epigraph_{s}\fonctionprimal \cap \epigraph_{s}\fonctiondual$,
  as \BipolarFunction s vanish at the origin.

  The conclusion follows.
\end{proof}

\section{Polar subdifferentials and alignement}
\label{Subdifferential_of_the_circ-polarity}

In this Sect.~\ref{Subdifferential_of_the_circ-polarity},
we consider $\PRIMAL$ and $\DUAL$ two (real) vector spaces that are {paired}
(see \S\ref{Dual_pair,_paired_vector_spaces}).
In~\S\ref{Polar_subdifferentials_of_a_nonnegative_function},
we present three possible definitions for the polar subdifferential of a
nonnegative function.
The first two definitions are inspired by definitions of subdifferentials of dualities
\cite{Akian-Gaubert-Kolokoltsov:2002,Martinez-Legaz-Singer:1995}.
We propose a third definition which, to our knowledge, is new and will be explored in more detail
in~\S\ref{Middle_polar_subdifferential_and_alignement} in relation to the notion
of alignement.

\subsection{Polar subdifferentials of a nonnegative function}
\label{Polar_subdifferentials_of_a_nonnegative_function}

In duality in convex analysis, one uses the (Rockafellar-Moreau)
subdifferential~\eqref{eq:Rockafellar-Moreau-subdifferential_a}
which is defined over the effective domain of a proper function
(in general convex lsc, but this is not compulsory).
By restricting to proper functions, one avoids the value~$-\infty$,
which is the bottom of the ordered set~\( \np{\barRR,\leq} \).
By contrast, with polarity one deals with nonnegative functions than can take the value~$0$,
which is the bottom of the ordered set \( \np{\barRR_{+},\leq} \).
This explains the three formulas~\eqref{eq:polarity_subdifferentials},
for the polar subdifferential of a nonnegative function,
which do not necessarily give the same result, especially when
\( \fonctionprimal\np{\primal}=0 \) or $+\infty$. 

\begin{definition}
  For any function~$\fonctionprimal \colon \PRIMAL\to\barRR_{+}$,
  we define
  \begin{subequations}
    \begin{enumerate}
\item
the \emph{lower polar subdifferential} of~$\fonctionprimal$
(inspired by~\cite[Equation~(10a)]{Akian-Gaubert-Kolokoltsov:2002}) by
  \begin{equation}
  \Lowsubdifferential{\circ}{\fonctionprimal}\np{\primal} = 
\bset{ \dual\in\DUAL }%
{    \Polarity{\fonctionprimal}\np{\dual} = 
    {\nscal{\primal}{\dual}}_+ 
    \LowTimes \npConverse{\fonctionprimal\np{\primal} } }
  \eqsepv
  \forall\primal\in\PRIMAL
  \eqfinv 
  \label{eq:polarity_subdifferential_Akian-Gaubert-Kolokoltsov:2002}
\end{equation}
\item
the \emph{upper polar subdifferential} of~$\fonctionprimal$
(inspired by~\cite[Equation~(1.7) in
Definition~1.2]{Martinez-Legaz-Singer:1995}) by
  \begin{equation}
  \Uppsubdifferential{\circ}{\fonctionprimal}\np{\primal} = 
\bset{ \dual\in\DUAL }%
{    \fonctionprimal\np{\primal} = 
    {\nscal{\primal}{\dual}}_+ 
    \LowTimes \npConverse{ \Polarity{\fonctionprimal}\np{\dual} } }
  \eqsepv
  \forall\primal\in\PRIMAL
  \eqfinv
  \label{eq:polarity_subdifferential_Martinez-Legaz-Singer:1995}
\end{equation}
\item
  the \emph{middle polar subdifferential} of~$\fonctionprimal$
  (inspired by the equality case in the Fenchel-Young, Cauchy-Schwarz and polar inequalities) by 
  \begin{equation}
  \Midsubdifferential{\circ}{\fonctionprimal}\np{\primal} = 
\bset{ \dual\in\DUAL }%
{  {\nscal{\primal}{\dual}}_+ =
  \fonctionprimal\np{\primal} \UppTimes
  \Polarity{\fonctionprimal}\np{\dual}}
  \eqsepv
  \forall\primal\in\PRIMAL
  \eqfinp 
  \label{eq:polarity_subdifferential_Fenchel-Young}
  \end{equation}
  \label{eq:polarity_subdifferentials}
    \end{enumerate}
  \end{subequations}
   \label{de:polarity_subdifferentials}
\end{definition}
The first two 
expressions~\eqref{eq:polarity_subdifferential_Akian-Gaubert-Kolokoltsov:2002}
and~\eqref{eq:polarity_subdifferential_Martinez-Legaz-Singer:1995}
are inspired by definitions of subdifferentials of dualities
in \cite{Akian-Gaubert-Kolokoltsov:2002} and
\cite{Martinez-Legaz-Singer:1995}.
We propose the third
expression~\eqref{eq:polarity_subdifferential_Fenchel-Young} which, to the best
of our knowledge, is new and will be explored in more detail
in~\S\ref{Middle_polar_subdifferential_and_alignement}.

Each of the three definitions,
for the polar subdifferential of a nonnegative function,
in Definition~\ref{de:polarity_subdifferentials}
have their own advantages as shown in
Proposition~\ref{pr:polarity_subdifferentials}. 

\begin{proposition}
  \label{pr:polarity_subdifferentials}
  For any function~$\fonctionprimal \colon \PRIMAL\to\barRR_{+}$,
  we have the following results. 
  \begin{enumerate}
\item
  Regarding
  definition~\eqref{eq:polarity_subdifferential_Akian-Gaubert-Kolokoltsov:2002}
  of the lower polar
  subdifferential~$\Lowsubdifferential{\circ}{\fonctionprimal}$ of~$\fonctionprimal$,
  we have the alternate expressions, for any~\( \primal\in\PRIMAL \),
  \begin{subequations}
    \begin{align}
      \Lowsubdifferential{\circ}{\fonctionprimal}\np{\primal}
      &= 
\bset{ \dual \in \DUAL }%
{    \Polarity{\fonctionprimal}\np{\dual} \leq 
    {\nscal{\primal}{\dual}}_+ 
    \LowTimes \npConverse{\fonctionprimal\np{\primal} } }
  \eqfinv 
      \\
      &= 
\bset{ \dual\in\DUAL }%
{ {\nscal{\primal'}{\dual}}_+ 
        \LowTimes \npConverse{\fonctionprimal\np{\primal'} }
        \leq 
    {\nscal{\primal}{\dual}}_+ 
    \LowTimes \npConverse{\fonctionprimal\np{\primal}} \eqsepv
  \forall\primal'\in\RR^{\spacedim} }
  \eqfinv 
      \\
      &= 
\bset{ \dual\in\DUAL }%
        { \primal \in \argmax_{\primal' \in \RR^{\spacedim}} \bp{ {\nscal{\primal'}{\dual}}_+ 
    \LowTimes \npConverse{\fonctionprimal\np{\primal'} } } }
 \eqfinp
    \end{align}
    \label{eq:polarity_subdifferential_Akian-Gaubert-Kolokoltsov:2002_alternate}
\end{subequations}
%
\item
Regarding
  definition~\eqref{eq:polarity_subdifferential_Martinez-Legaz-Singer:1995}
  of the upper polar subdifferential~$\Uppsubdifferential{\circ}{\fonctionprimal}$ of~$\fonctionprimal$,
  we have the property that
  \begin{equation}
\Uppsubdifferential{\circ}{\fonctionprimal}\np{\primal}
\neq \emptyset \implies
\Polaritybi{\fonctionprimal}\np{\primal}
= \fonctionprimal\np{\primal}
\eqfinp
\label{eq:polarity_subdifferential_Martinez-Legaz-Singer:1995_property}
  \end{equation}
\item
Regarding
  definition~\eqref{eq:polarity_subdifferential_Fenchel-Young}
  of the middle polar subdifferential~$\Midsubdifferential{\circ}{\fonctionprimal}$ of~$\fonctionprimal$,
  we have the alternate expression, for any~\( \primal\in\RR^{\spacedim} \),
  \begin{equation}
      \Midsubdifferential{\circ}{\fonctionprimal}\np{\primal}
      =
  \bset{ \dual\in\DUAL }%
{  \nscal{\primal}{\dual} = 
  \fonctionprimal\np{\primal} \UppTimes
  \Polarity{\fonctionprimal}\np{\dual} }      
        \eqfinp
\label{eq:polarity_subdifferential_Fenchel-Young_alternate}
  \end{equation}
  %
  \end{enumerate}
\end{proposition}

\begin{proof}
  \begin{enumerate}
  \item
    Equations~\eqref{eq:polarity_subdifferential_Akian-Gaubert-Kolokoltsov:2002_alternate}
    come from the definition~\eqref{eq:circ-polar_transform}
    of~\( \Polarity{\fonctionprimal}\np{\dual} \).
    
  \item
Let \( \dual\in\Uppsubdifferential{\circ}{\fonctionprimal}\np{\primal} \).
    The Implication~\eqref{eq:polarity_subdifferential_Martinez-Legaz-Singer:1995_property}
    is a consequence of
    \begin{align*}
      \Polaritybi{\fonctionprimal}\np{\primal}
      &=
              \sup_{\dual' \in \DUAL } \Bp{ {\nscal{\primal}{\dual'}}_+ 
        \LowTimes \bpConverse{\Polarity{\fonctionprimal}\np{\dual'} } }
        \tag{by definition~\eqref{eq:circ-polar_bitransform} of~\( \PolarityReverse{\np{\Polarity{\fonctionprimal}}} \)
        and by definition~\eqref{eq:reverse_circ-polar_transform}
        of~\( \PolarityReverse{\fonctiondual} \)}
      \\
      &\geq
        {\nscal{\primal}{\dual}}_+ 
        \LowTimes \bpConverse{\Polarity{\fonctionprimal}\np{\dual} }
      \\
      &=
        \fonctionprimal\np{\primal}
        \tag{by~\eqref{eq:polarity_subdifferential_Martinez-Legaz-Singer:1995}}
        \eqfinp 
    \end{align*}
    As \( \Polaritybi{\fonctionprimal}\np{\primal} \leq
    \fonctionprimal\np{\primal} \) by the Inequality~\eqref{eq:bipolar_is_smaller},
    we conclude that \( \Polaritybi{\fonctionprimal}\np{\primal} =
    \fonctionprimal\np{\primal} \).
    
  \item
Equation~\eqref{eq:polarity_subdifferential_Fenchel-Young_alternate}
    follows from the fact that \( \fonctionprimal\np{\primal} \UppTimes
  \Polarity{\fonctionprimal}\np{\dual} \geq 0 \).
  \end{enumerate}
\end{proof}

\subsection{Middle polar subdifferential and alignement}
\label{Middle_polar_subdifferential_and_alignement}

When the function~$\fonctionprimal \colon \PRIMAL\to\barRR_{+}$ is a norm~\(
\TripleNorm{\cdot} \) on~$\PRIMAL$, then
$\Polarity{\fonctionprimal}$ is the so-called
\emph{dual norm}~\( \TripleNormDual{\cdot} \) on~$\DUAL$.
Couples \( \np{\primal,\dual} \in \PRIMAL\times\DUAL\) of vectors
satisfying \( \nscal{\primal}{\dual} =
\TripleNorm{\primal}\TripleNormDual{\dual} \)
  are said to be
\( \TripleNorm{\cdot} \)-dual in \cite[page~2]{Marques_de_Sa-Sodupe:1993},
to form a \emph{dual vector pair} in~\cite[Equation~(1.11)]{Gries:1967},
to be \emph{dual vectors} in~\cite[p.~283]{Gries-Stoer:1967}, 
or to satisfy \emph{polar alignment} in~\cite[Definition~2.4]{Fan-Jeong-Sun-Friedlander:2020}.

We propose the following definition that encompasses the above definitions,
and goes beyond.
\begin{definition}
Let \( \Primal\subset\PRIMAL \), \( \Dual\subset\DUAL \) 
be a polar pair 
(that is, \( \Dual=\PolarSet{\Primal} \), \(  \Primal=\PolarSet{\Dual} \)).

  We say that the couple \( \np{\primal,\dual} \in \PRIMAL\times\DUAL\) is
  \emph{aligned \wrt~$\np{\Primal,\Dual}$} 
    (or \wrt~$\np{\SupportFunction{\Dual},\SupportFunction{\Primal}}$)
    if\footnote{%
      Equation~\eqref{eq;alignement} could be replaced by
      \( \SupportFunction{\Dual}\np{\primal}, \SupportFunction{\Primal}\np{\dual} \in \RR_{++}\)
      and \( \nscal{\primal}{\dual}  = \SupportFunction{\Dual}\np{\primal} 
      \SupportFunction{\Primal}\np{\dual} \).
      Indeed, the upper multiplication~$\UppTimes$ can be replaced by the usual multiplication~$\times$
      since \(         0 < \nscal{\primal}{\dual}  =
 \SupportFunction{\Dual}\np{\primal} \UppTimes
 \SupportFunction{\Primal}\np{\dual} \) 
$ \iff$
 \( 0 < \nscal{\primal}{\dual}  =
 \SupportFunction{\Dual}\np{\primal} 
 \SupportFunction{\Primal}\np{\dual} \), 
because \( \SupportFunction{\Dual}\np{\primal} \neq +\infty \)
and \( \SupportFunction{\Primal}\np{\dual} \neq +\infty \)
(else the right hand side would be $+\infty$). 
}
      
  \begin{equation}
    0 < \nscal{\primal}{\dual}
    \mtext{ and }
    \nscal{\primal}{\dual} 
 =
 \SupportFunction{\Dual}\np{\primal} \UppTimes
 \SupportFunction{\Primal}\np{\dual}
 \eqfinp
 \label{eq;alignement}
\end{equation}
In this definition, \( \SupportFunction{\Dual}\) can be replaced by
\( \MinkowskiFunctional{\Primal} \) and
\( \SupportFunction{\Primal}\) by \( \MinkowskiFunctional{\Dual} \),
because of~\eqref{eq:MinkowskiFunctional=SupportFunction_c}.
\end{definition}

The relationship between alignement
and the middle polar subdifferential~\eqref{eq:polarity_subdifferential_Fenchel-Young} is
as follows.

\begin{proposition}
Let \( \Primal\subset\PRIMAL \), \( \Dual\subset\DUAL \) 
be a polar pair, and \( \np{\primal,\dual} \in \PRIMAL\times\DUAL\).
Then, we have that 
\begin{equation}
    \dual \in \Midsubdifferential{\circ}{\SupportFunction{\Dual}}\np{\primal}
    \iff
    \primal \in \Midsubdifferential{\circ}{\SupportFunction{\Primal}}\np{\dual}
      \iff
     \begin{cases}
       \text{either}   & \primal \bot \dual 
       \eqfinv
          \\
          \text{or}    &  \np{\primal,\dual} \text{ is aligned \wrt}~\np{\Primal,\Dual}
          \eqfinp
                 \end{cases}
\end{equation}
\end{proposition}

\begin{proof}
The proof follows from the very 
   definition~\eqref{eq:polarity_subdifferential_Fenchel-Young}
   of the middle polar subdifferentials
   \( \Midsubdifferential{\circ}{\SupportFunction{\Dual}}\) and
   \( \Midsubdifferential{\circ}{\SupportFunction{\Primal}}\).
\end{proof}

As in \cite[\S~3.3]{Fan-Jeong-Sun-Friedlander:2020}, we relate alignement to
well-known geometric objets in convex analysis. 

\begin{definition}
  For any nonempty closed convex subset $\Convex \subset \PRIMAL$,
the \emph{exposed face} of~$\Convex$ by the dual vector~$\dual \in \DUAL$ is
  \begin{equation}
    \ExposedFace(\Convex,\dual) 
    =\argmax_{\primal\in\Convex} \nscal{\primal}{\dual}
    \eqfinv
  \end{equation}
and the \emph{normal  cone}~$\NormalCone(\Convex,\primal)$ 
   at any primal vector~\( \primal\in\Convex \) is defined by the conjugacy
   relation 
\begin{equation}
 \primal \in \Convex \mtext{ and } \dual \in \NormalCone(\Convex,\primal) 
\iff 
\primal \in \ExposedFace(\Convex,\dual) 
\eqfinv
\end{equation}
that is, equivalently, by
\begin{equation}
  \NormalCone(\Convex,\primal) = \defset{\dual \in \DUAL}%
  { \nscal{\primalbis-\primal}{\dual} \leq 0 \eqsepv
    \forall \primalbis\in\Convex} \eqsepv
  \forall \primal\in\Convex
  \eqfinp 
\end{equation}
\end{definition}

The following Proposition is in the vein of~\cite[Proposition~3.3]{Fan-Jeong-Sun-Friedlander:2020},
but more detailed.
The proof is left to the reader.

\begin{proposition}
Let \( \Primal\subset\PRIMAL \), \( \Dual\subset\DUAL \) 
be a polar pair, and \( \np{\primal,\dual} \in \PRIMAL\times\DUAL\).
Then, we have that 
\begin{subequations}
  \begin{align}
    &  \np{\primal,\dual} \text{ is aligned \wrt}~\np{\Primal,\Dual}
      \eqfinv
                        \\
    \iff & 
           \SupportFunction{\Dual}\np{\primal}\in\RR_{++} \eqsepv \SupportFunction{\Primal}\np{\dual}\in\RR_{++} \text{ and }
           \nscal{\primal}{\dual} =  \SupportFunction{\Dual}\np{\primal}
                  \SupportFunction{\Primal}\np{\dual}
           \eqfinv
    \\
    \iff & 
       \SupportFunction{\Dual}\np{\primal}\in\RR_{++} \eqsepv \SupportFunction{\Primal}\np{\dual}\in\RR_{++} \text{ and }
           \frac{\dual}{\SupportFunction{\Primal}\np{\dual}} \in
           \ExposedFace\np{\Dual,\frac{\primal}{\SupportFunction{\Dual}\np{\primal}}}
           \eqfinv
    \\
    \iff & 
       \SupportFunction{\Dual}\np{\primal}\in\RR_{++} \eqsepv \SupportFunction{\Primal}\np{\dual}\in\RR_{++} \text{ and }
           \frac{\dual}{\SupportFunction{\Primal}\np{\dual}} \in
           \ExposedFace\np{\Dual,\primal}
           \eqfinv
    \\
     \iff & 
       \SupportFunction{\Dual}\np{\primal}\in\RR_{++} \eqsepv \SupportFunction{\Primal}\np{\dual}\in\RR_{++} \text{ and }
      \primal 
      \in \NormalCone\np{\Dual,\frac{\dual}{\SupportFunction{\Primal}\np{\dual}}}
            \eqfinv
    \\
     \iff & 
       \SupportFunction{\Dual}\np{\primal}\in\RR_{++} \eqsepv \SupportFunction{\Primal}\np{\dual}\in\RR_{++} \text{ and }
      \frac{\primal}{\SupportFunction{\Dual}\np{\primal}}
      \in \NormalCone\np{\Dual,\frac{\dual}{\SupportFunction{\Primal}\np{\dual}}}
            \eqfinv
    \\
     \iff & 
       \SupportFunction{\Dual}\np{\primal}\in\RR_{++} \eqsepv \SupportFunction{\Primal}\np{\dual}\in\RR_{++} \text{ and }
      \frac{\primal}{\SupportFunction{\Dual}\np{\primal}}
      \in \ExposedFace\np{\Primal,\frac{\dual}{\SupportFunction{\Primal}\np{\dual}}}
            \eqfinv
    \\
     \iff & 
       \SupportFunction{\Dual}\np{\primal}\in\RR_{++} \eqsepv \SupportFunction{\Primal}\np{\dual}\in\RR_{++} \text{ and }
      \frac{\primal}{\SupportFunction{\Dual}\np{\primal}}
      \in \ExposedFace\np{\Primal,\dual}
            \eqfinv
    \\
     \iff & 
       \SupportFunction{\Dual}\np{\primal}\in\RR_{++} \eqsepv \SupportFunction{\Primal}\np{\dual}\in\RR_{++} \text{ and }
\dual   \in \NormalCone\np{\Primal,
            \frac{\primal}{\SupportFunction{\Dual}\np{\primal}}} 
            \eqfinv
    \\
     \iff & 
       \SupportFunction{\Dual}\np{\primal}\in\RR_{++} \eqsepv \SupportFunction{\Primal}\np{\dual}\in\RR_{++} \text{ and }
\frac{\dual}{\SupportFunction{\Primal}\np{\dual}}     
      \in \NormalCone\np{\Primal,
            \frac{\primal}{\SupportFunction{\Dual}\np{\primal}}} 
            \eqfinp
  \end{align}
\end{subequations}
In this Proposition, \( \SupportFunction{\Dual}\) can be replaced by
\( \MinkowskiFunctional{\Primal} \) and
\( \SupportFunction{\Primal}\) by \( \MinkowskiFunctional{\Dual} \),
because of~\eqref{eq:MinkowskiFunctional=SupportFunction_c}.

\end{proposition}

\appendix

\section{Best lsc convex lower approximations of a function}
\label{app:best_cvx_general}

We consider $\PRIMAL$ and $\DUAL$ two (real) vector spaces that are {paired}
(see \S\ref{Dual_pair,_paired_vector_spaces}).

Consider $\GSubset \subset {\barRR}^{\PRIMAL}$, a subset of the functions defined on ${\PRIMAL}$ and
taking values in the extended reals.
Then, for any function~$\fonctionprimal \colon \PRIMAL \to \barRR$,
we define the subset $\GLower{\GSubset}{\fonctionprimal}{\Uncertain} \subset
\GSubset$ by
\begin{equation}
  \GLower{\GSubset}{\fonctionprimal}{\Uncertain}
  =
  \bset{ \fonctionprimalbis \in \GSubset}%
  { \fonctionprimalbis\np{\primal}\leq \fonctionprimal\np{\primal}
    \eqsepv \forall \primal \in \PRIMAL}.
  \label{eq:BestConvexLowerApproximations}
\end{equation}
Then, the best lower $\GSubset$-approximation of $\fonctionprimal$, denoted by
$\GreatestIn{\GSubset}{\fonctionprimal}{\Uncertain}$,
is defined as the greatest element
of $\GLower{\GSubset}{\fonctionprimal}{\Uncertain}$
(which always exists as $+\infty$ is always a possible value).

Now, we are going to consider two cases for the subset~$\GSubset \subset {\barRR}^{\PRIMAL}$:
lsc convex extended functions in~\S\ref{subsec:ccef} and
lsc convex (strictly positively) $1$-homogeneous extended functions
in~\S\ref{The_case_of_closed_convex_positively_1-homogeneous_extended_function}. 

\subsection{The case of lsc convex extended functions}
\label{subsec:ccef}

Let $\CCSubset$ denote the set of lsc convex extended functions on $\PRIMAL$.
When $\GSubset=\CCSubset$, the greatest lsc convex lower approximation
of the function~$\fonctionprimal$, that is $\GreatestIn{\CCSubset}{\fonctionprimal}{\Uncertain}$,
is given by~\cite[Theorem 3.1]{Volle-Legaz-Perez:2015} that we reproduce in Proposition~\ref{prop:best_cvx_subset}.
Following the terminology
of~\cite{Penot:2000}, 
the \emph{valley function} $\valley_{A}$ of a subset $A\subset \PRIMAL$ is defined
by $\valley_{A}(\primal)= -\infty$ if $\primal\in A$ and
$\valley_{A}(\primal)= +\infty$ if $\primal\in \PRIMAL\setminus A$.

\begin{proposition}(\cite[Theorem 3.1]{Volle-Legaz-Perez:2015})
  \label{prop:best_cvx_subset}
  For any function \( \fonctionprimal \colon \PRIMAL \to \barRR \),
  the greatest lsc convex lower approximation
  of the function~$\fonctionprimal$ is given by 
  \begin{align}
    \GreatestIn{\CCSubset}{\fonctionprimal}{\Uncertain}
    =
    \LFMbi{\fonctionprimal}
    \UppPlus
    \Indicator{\closedconvexhull\np{\dom\fonctionprimal}}
    &=
      \begin{cases}
        \LFMbi{\fonctionprimal}
        &
          \text{if $\fonctionprimal$ is proper}
          \eqfinv
        \\
        \valley_{\closedconvexhull\np{\dom\fonctionprimal}}
        & \text{if $\fonctionprimal$ is not proper}
          \eqfinp
      \end{cases}
      \label{eq:bestlscconvex}
  \end{align}
\end{proposition}
The function $\GreatestIn{\GSubset}{\fonctionprimal}{\PRIMAL}$
is also classically denoted by $\closedconvexhull\np{\fonctionprimal}$.

\subsection{The case of lsc convex (strictly positively) $1$-homogeneous extended functions}
\label{The_case_of_closed_convex_positively_1-homogeneous_extended_function}

Let $\PHCCSubset$ denote the set of lsc convex (strictly positively)
$1$-homogeneous extended functions on $\PRIMAL$.  For $\GSubset=\PHCCSubset$, we
give in Equation~\eqref{eq:bestlscconvex1hom} the expression of
$\GreatestIn{\PHCCSubset}{\fonctionprimal}{\PRIMAL}$, the greatest lsc convex
(strictly positively) $1$-homogeneous lower approximation of the function
$\fonctionprimal$ which is, to our knowledge, a new result.
Recall that $\closedconvexcone \Primal$ is the \emph{closed conical hull} of
$\Primal\subset \PRIMAL$, that is, the smallest closed cone in~$\PRIMAL$
containing~$\Primal$.


\begin{proposition}
  \label{prop:best_pos_hom_cvx_subset}
  For any  function \( \fonctionprimal \colon \PRIMAL \to \barRR \),
  such that  $0\in\dom\fonctionprimal$, the greatest lsc convex
  (strictly positively) $1$-homogeneous lower approximation
  of $\fonctionprimal$ is given by 
  \begin{equation}
    \GreatestIn{\PHCCSubset}{\fonctionprimal}{\PRIMAL}
    =
    \begin{cases}
      \SupportFunction{\MidLevelSet{\LFM{\fonctionprimal}}{0}} 
      & \text{if}\quad {\MidLevelSet{\LFM{\fonctionprimal}}{0}}\not=\emptyset\eqsepv
      \\
      \valley_{\closedconvexcone\np{\closedconvexhull\np{\dom\fonctionprimal}}}
      & \text{if}\quad{\MidLevelSet{\LFM{\fonctionprimal}}{0}}=\emptyset
        \eqfinp
    \end{cases}
    \label{eq:bestlscconvex1hom}
  \end{equation}
 \end{proposition}

\begin{proof}
  Let $\fonctionprimal: \PRIMAL \to \barRR$ be given.
  We will successively consider the two disjoint cases: 
  ${\MidLevelSet{\LFM{\fonctionprimal}}{0}}=\emptyset$ and   ${\MidLevelSet{\LFM{\fonctionprimal}}{0}}\not=\emptyset$.
  
  
  \noindent $\bullet$
  We consider the case where ${\MidLevelSet{\LFM{\fonctionprimal}}{0}}=\emptyset$. 
    We are going to extend~\cite[Lemma 2.3]{Volle-Legaz-Perez:2015} in order to obtain that
    \begin{equation}
      \GreatestIn{\PHCCSubset}{\fonctionprimal}{\PRIMAL}
      =
      \valley_{\closedconvexcone{(\closedconvexhull(\dom \fonctionprimal))}}
      \eqfinp
    \end{equation}
    We denote by $g=\GreatestIn{\PHCCSubset}{\fonctionprimal}{\PRIMAL}$.
    We observe that the function~$g$ is not proper. Indeed, otherwise $g$ would admit a continuous affine minorant
    and, being also (strictly positively) $1$-homogeneous, it is easily deduced that it would admit a continuous linear
    minorant;
    thus, as $g \le \fonctionprimal$, the function $\fonctionprimal$ would also admit a
    continuous linear minorant, which would imply that ${\MidLevelSet{\LFM{\fonctionprimal}}{0}}\not=
    \emptyset$ (contradiction).
Consequently, the function~$g$ is not proper and,
    by \cite[Lemma 2.2]{Volle-Legaz-Perez:2015}, we get that $g = \valley_{\dom{g}}$.
    
    As the function~$g= \valley_{\dom{g}}$ is in $\GLower{\PHCCSubset}{\fonctionprimal}{\PRIMAL}$, it
    is lsc. Hence, we obtain that $\valley_{\dom{g}}$ is lsc and, using the
    definition of a valley function, that the subset~$\dom{g}$ is closed.
    Combined with the fact that the function~$g$ is convex and (strictly positively) $1$-homogeneous as an 
    element of $\GLower{\PHCCSubset}{\fonctionprimal}{\PRIMAL}$, this gives that 
    $\dom{g}$ is a closed convex cone.

    As $\valley_{\dom{g}}=g \le \fonctionprimal$,
    we get that $\dom \fonctionprimal \subset \dom{g}$, from which we obtain that
    $\closedconvexcone\bp{\closedconvexhull\np{\dom \fonctionprimal}}
    \subset \closedconvexcone\bp{\closedconvexhull\np{\dom g}}= \dom g$. 
    Finally, we observe that 
    $\valley_{\closedconvexcone\bp{\closedconvexhull\np{\dom \fonctionprimal}}}$ is in
    $\GLower{\PHCCSubset}{\fonctionprimal}{\PRIMAL}$, and is therefore smaller
    than~$g$ by definition of $g=
    \GreatestIn{\PHCCSubset}{\fonctionprimal}{\PRIMAL}$.
    This gives that
    $\valley_{\closedconvexcone\bp{\closedconvexhull\np{\dom \fonctionprimal}}}\le \valley_{\dom{g}}$ and thus
    $\dom{g}\subset {\closedconvexcone\bp{\closedconvexhull\np{\dom \fonctionprimal}}}$. 
    We conclude that $\dom{g}={\closedconvexcone\bp{\closedconvexhull\np{\dom \fonctionprimal}}}$, from which we derive
    the equalities $g = \valley_{\dom{g}}= \valley_{\closedconvexcone\bp{\closedconvexhull\np{\dom \fonctionprimal}}}$.
    
  \medskip
  
  \noindent$\bullet$ We consider the case where ${\MidLevelSet{\LFM{\fonctionprimal}}{0}} \not=\emptyset$.

  
  First, the function $\sigma_{\na{\LFM{\fonctionprimal} \le 0}}$ is (strictly positively)
  $1$-homogeneous convex lsc as a support function and smaller than
  $\fonctionprimal$ (by~\eqref{eq:bipolar_is_smaller}).
  Moreover, using using the fact that ${\MidLevelSet{\LFM{\fonctionprimal}}{0}} \not=\emptyset$ we obtain
  that $\sigma_{\na{\LFM{\fonctionprimal} \le 0}}$ is proper.
  Therefore, we have by definition
  of $\GreatestIn{\PHCCSubset}{\fonctionprimal}{\PRIMAL}$ that
  \begin{equation}
    \sigma_{\na{\LFM{\fonctionprimal} \le 0}} \le \GreatestIn{\PHCCSubset}{\fonctionprimal}{\PRIMAL}
    \eqfinp
    \label{sigma_less_g}
  \end{equation}
  Second, as the function~$\GreatestIn{\PHCCSubset}{\fonctionprimal}{\PRIMAL}$
  is (strictly positively) $1$-homogeneous, by definition, and proper, as 
  it is minorized by a proper function and thus cannot take the value $-\infty$,
    we have
  $-\infty < \GreatestIn{\PHCCSubset}{\fonctionprimal}{\PRIMAL}(0) \le
  \fonctionprimal(0) < + \infty$ since we have assumed $0\in\dom\fonctionprimal$.  We obtain, using
  Lemma~\ref{le:bifenchel-homogeneous}, that the greatest proper lsc convex
  function majorized by~$\GreatestIn{\PHCCSubset}{\fonctionprimal}{\PRIMAL}$
  is~$\SupportFunction{{\na{\LFM{\GreatestIn{\PHCCSubset}{\fonctionprimal}{\PRIMAL}}
        \le 0}}}$.  Now, as the
  function~$\GreatestIn{\PHCCSubset}{\fonctionprimal}{\PRIMAL}$ is also lsc
  convex by definition, it must coincide with its best proper lsc convex lower
  approximation, that is
  $\GreatestIn{\PHCCSubset}{\fonctionprimal}{\PRIMAL}=\SupportFunction{{\na{\LFM{\GreatestIn{\PHCCSubset}{\fonctionprimal}{\PRIMAL}}
        \le 0}}}$.
  
  Third, by definition of $\GreatestIn{\PHCCSubset}{\fonctionprimal}{\PRIMAL}$, we have $\GreatestIn{\PHCCSubset}{\fonctionprimal}{\PRIMAL} \le f$, and thus ${\na{\LFM{\GreatestIn{\PHCCSubset}{\fonctionprimal}{\PRIMAL}} \le 0}} \subset {\na{\LFM{\fonctionprimal} \le 0}}$.
  We obtain that
  $\GreatestIn{\PHCCSubset}{\fonctionprimal}{\PRIMAL} =\SupportFunction{\na{\LFM{\GreatestIn{\PHCCSubset}{\fonctionprimal}{\PRIMAL}} \le 0}} \le \SupportFunction{\na{\LFM{\fonctionprimal} \le 0}}$
  which combined with~\eqref{sigma_less_g} gives that
  $\GreatestIn{\PHCCSubset}{\fonctionprimal}{\PRIMAL}=\SupportFunction{\na{\LFM{\fonctionprimal} \le 0}}$.
  \medskip
  
  This concludes the proof.
\end{proof}

We give here an instrumental lemma used in the proof of the previous Proposition~\ref{prop:best_pos_hom_cvx_subset}.
\begin{lemma}
  \label{le:bifenchel-homogeneous}
  Assume that $\fonctionprimalbis: \PRIMAL \to \barRR$ is a (strictly positively) $1$-homogeneous proper function.
  Then, the greatest lsc proper convex lower approximation
  of the function~$\fonctionprimalbis$ is given by $\sigma_{\na{\LFM{\fonctionprimalbis} \le 0}}$.
\end{lemma}

\begin{proof}
  We start by proving some properties of $\LFM{\fonctionprimalbis}$.
  
  \noindent $(i)$ If $\LFM{\fonctionprimalbis}(\dual) >0$ then $\LFM{\fonctionprimalbis}(\dual) = +\infty$.
  Indeed, if $\LFM{\fonctionprimalbis}(\dual) >0$ there exists $\alpha > 0$ and $\primal \in \PRIMAL$ such that
  $\nscal{\primal}{\dual} - \fonctionprimalbis(\primal) > \alpha$. Thus,  for all $\lambda > 0$, we obtain that
  \begin{align}
    \LFM{\fonctionprimalbis}(\dual)
    \ge \nscal{\lambda \primal}{\dual} - \fonctionprimalbis(\lambda \primal)
    = \lambda \bp{\nscal{\primal}{\dual} - \fonctionprimalbis( \primal)}
    \ge \lambda \alpha
    \eqfinv
  \end{align}
  and thus that $\LFM{\fonctionprimalbis}(\dual) = +\infty$ by letting $\lambda$
  goto $+\infty$.

  \noindent $(ii)$ If $\LFM{\fonctionprimalbis}(\dual) \le 0$, then $\LFM{\fonctionprimalbis}(\dual)=0$.
  Indeed, for all $\lambda >0$ and $\primal \in \PRIMAL$, we have that 
  \begin{equation*}
    \lambda \bp{\nscal{\primal}{\dual} - \fonctionprimalbis( \primal)}
    = \nscal{\lambda \primal}{\dual} - \fonctionprimalbis( \lambda \primal) \le \LFM{\fonctionprimalbis}(\dual) \le 0
    \eqfinv 
  \end{equation*}
  and the result follows by chosing $x$ such that $\fonctionprimalbis(x) \in \RR$,
  which exists as $\fonctionprimalbis$ is proper, and letting $\lambda$ go to $0$.

  Now, we turn to the proof of the Lemma.
  As $\fonctionprimalbis$ is assumed to be proper, the
  greatest lsc proper convex lower approximation of the function~$\fonctionprimalbis$ is given by
  the Fenchel biconjugate~$\LFMbi{ \fonctionprimalbis}$ in~\eqref{eq:Fenchel_biconjugate} of the function
  $\fonctionprimalbis$, which is 
  \begin{align*}
    \LFMbi{\fonctionprimalbis}
    &= 
      \sup_{\dual \in \DUAL} \bp{ \nscal{\cdot}{\dual} 
      -\LFM{\fonctionprimalbis}\np{\dual} }
    \\
    &= \sup \Bp{
      \sup_{\dual \in \na{\LFM{\fonctionprimalbis} >0 }}
      \bp{\nscal{\cdot}{\dual} -\LFM{\fonctionprimalbis}\np{\dual}}
      ,
      \sup_{\dual \in \na{\LFM{\fonctionprimalbis} \le 0 }}
      \bp{\nscal{\cdot}{\dual} -\LFM{\fonctionprimalbis}\np{\dual}}}
    \\
    &= \sup \Bp{
      \sup_{\dual \in \na{\LFM{\fonctionprimalbis} > 0}} \bp{\nscal{\cdot}{\dual} - (+ \infty)}
      ,
      \sup_{\dual \in \na{\LFM{\fonctionprimalbis} \le 0 }}
      \bp{\nscal{\cdot}{\dual}-0} }
      \tag{using $(i)$ and $(ii)$}
    \\
    &= \sup \bp{ -\infty ,  \sigma_{\na{\LFM{\fonctionprimalbis} \le 0}}}
    = \sigma_{\na{\LFM{\fonctionprimalbis} \le 0}}
      \eqfinp
  \end{align*}
  This ends the proof.
\end{proof}

\section{Background on *-dualities}
\label{Background_on_*-dualities}

We provide here the necessary background on *-dualities, as defined ans
studied in \cite{Martinez-Legaz-Singer:1994}. 

\subsection{Canonical enlargement of a complete totally ordered group}

A complete totally ordered (commutative\footnote{%
  See \cite[bottom of page~296]{Martinez-Legaz-Singer:1994} for why we restrict
  to commutative groups.}) group is a triplet
\( \np{\Group,\leq,\GroupOperation} \), where $\np{\Group,\leq}$ is a totally ordered
set (either $\group\leq \groupbis$ or $\groupbis\leq \group$)
$\np{\Group,\GroupOperation}$ is a (commutative) group --- such that all
translations are isotone
(\( \group\leq \groupbis \implies \group\GroupOperation\groupter\leq
\groupbis\GroupOperation\groupter\)), and any nonempty (order) bounded subset
admits a supremum and an infimum.

As defined and studied in \cite[Sect.~1]{Martinez-Legaz-Singer:1994}, we
describe the \emph{canonical enlargement}
\( \np{\barGroup,\leq,\UppGroupOperation,\LowGroupOperation} \) of a complete
totally ordered group \( \np{\Group,\leq,\GroupOperation} \) by
\begin{equation}
  \barGroup = \Group \cup \na{-\infty}  \cup \na{+\infty}
  \eqfinv
\end{equation}
with order extended by
\begin{equation}
  -\infty \leq \group \leq +\infty
  \eqsepv \forall \group \in \barGroup 
  \eqfinv
\end{equation}
and with \emph{upper composition}~$\UppGroupOperation$
and \emph{lower composition}~$\LowGroupOperation$ given by
\cite[Equations~(1.4)-(1.8)]{Martinez-Legaz-Singer:1994}
\begin{subequations}
  \begin{align}
    \group\UppGroupOperation\groupbis
    &=
      \group\LowGroupOperation\groupbis
      =       \group\GroupOperation\groupbis
      \eqsepv \forall \group,\groupbis \in \Group 
      \eqfinv
    \\
    \np{+\infty} \UppGroupOperation \group
    &= \group \UppGroupOperation \np{+\infty} = +\infty 
      \eqsepv \forall \group \in \barGroup 
      \eqfinv
    \\
    \np{-\infty} \UppGroupOperation \group
    &= \group \UppGroupOperation \np{-\infty} = -\infty 
      \eqsepv \forall \group \in \Group \cup \na{-\infty} 
      \eqfinv
    \\
    \np{+\infty} \LowGroupOperation \group
    &= \group \LowGroupOperation \np{+\infty} = +\infty 
      \eqsepv \forall \group \in \Group \cup \na{+\infty} 
      \eqfinv
    \\
    \np{-\infty} \LowGroupOperation \group
    &= \group \LowGroupOperation \np{-\infty} = -\infty 
      \eqsepv \forall \group \in \barGroup 
      \eqfinp 
  \end{align}
\end{subequations}
As~\( \GroupOperation \), the operations~$\UppGroupOperation$
and $\LowGroupOperation$ are associative and commutative on~$\barGroup$.
The unit element~$\UnitElement$ satisfies
\cite[Equation~(1.9)]{Martinez-Legaz-Singer:1994}
\begin{equation}
  \group\UppGroupOperation\UnitElement
  =\group\LowGroupOperation\UnitElement
  = \group
  \eqsepv \forall \group \in \barGroup 
  \eqfinv
\end{equation}
and we extend the group inverse operation~$\npConverse{\cdot}$ by
\cite[Equation~(1.11)]{Martinez-Legaz-Singer:1994}
\begin{equation}
  \npConverse{+\infty}=-\infty  \eqsepv
  \npConverse{-\infty}=+\infty  \eqfinp
\end{equation}
With the conventions        \cite[Equation~(1.13)]{Martinez-Legaz-Singer:1994}
\begin{equation}
  \inf \emptyset=+\infty  \eqsepv
  \sup \emptyset=-\infty  \text{ \qquad as elements of~}\barGroup \eqsepv
\end{equation}
we have the following properties,
where \( \group,\groupbis\) are any elements of~\( \barGroup \)
and \( \sequence{\group_i}{i\in I} \), \( \sequence{\groupbis_j}{j\in J} \),
are any collections, indexed by any sets~$I$, $J$, 
with values in~\( \barGroup \)
\begin{subequations}
  \begin{align}
    \npConverse{ \inf_{i\in I} \group_i } 
    &=
      \sup_{i\in I} \Converse{\group_i }
      \eqsepv
      \npConverse{ \sup_{i\in I} \group_i }
      = \inf_{i\in I} \Converse{\group_i }
      \eqfinv 
      \label{eq:lower_upper_GroupOperation_converse}
    \\
    \inf_{i\in I} \np{\group\UppGroupOperation\group_i } 
    &=
      \group\UppGroupOperation \inf_{i\in I} \group_i
      \eqsepv
      \sup_{i\in I} \np{\group\LowGroupOperation\group_i } 
      =
      \group\LowGroupOperation \sup_{i\in I} \group_i
      \eqfinv
           \label{eq:lower_upper_GroupOperation_inf_sup_one_element}
    \\
    \inf_{i\in I, j\in J} \np{\group_i\UppGroupOperation\groupbis_j} 
    &=
      \inf_{i\in I}\group_i\UppGroupOperation\inf_{j\in J} \groupbis_j
      \eqsepv
      \sup_{i\in I, j\in J} \np{\group_i\UppGroupOperation\groupbis_j}
      \leq
      \sup_{i\in I}\group_i\UppGroupOperation\sup_{j\in J} \groupbis_j     
      \eqfinv
      \label{eq:lower_upper_GroupOperation_inf_sup}
    \\      
    \sup_{i\in I, j\in J} \np{\group_i\LowGroupOperation\groupbis_j} 
    &=
      \sup_{i\in I}\group_i\LowGroupOperation\inf_{j\in J} \groupbis_j
      \eqsepv
      \inf_{i\in I, j\in J} \np{\group_i\LowGroupOperation\groupbis_j}
      \geq
      \inf_{i\in I}\group_i\LowGroupOperation\inf_{j\in J} \groupbis_j     
      \eqfinv
    \\
    \groupbis < +\infty
    &\implies
      \inf_{i\in I} \np{\group_i\LowGroupOperation\groupbis}
      = \np{\inf_{i\in I}\group_i} \LowGroupOperation\groupbis
      \eqfinv
    \\
    -\infty < \groupbis
    & \implies
      \sup_{i\in I}\np{\group_i\UppGroupOperation\groupbis}
      = \np{\sup_{i\in I}\group_i}\UppGroupOperation\groupbis
      \eqfinv
  \end{align}
\end{subequations}
and the following properties,
where \( \group,\groupbis,\groupter \) are any elements of~\( \barGroup \), 
\begin{subequations}
  \begin{align}
    \groupbis \leq \groupter
    &\implies
      \group\UppGroupOperation\groupbis \leq \group\UppGroupOperation\groupter
      \eqsepv
      \group\LowGroupOperation\groupbis \leq \group\LowGroupOperation\groupter        
      \eqfinv
      \label{eq:GroupOperation_isotony}
    \\
    \npConverse{ \group\UppGroupOperation\groupbis }
    &=
      \Converse{\group}\LowGroupOperation\Converse{\groupbis}
      \eqsepv
      \npConverse{ \group\LowGroupOperation\groupbis }
      =
      \Converse{\group}\UppGroupOperation\Converse{\groupbis}
      \eqfinv
            \label{eq:converse_GroupOperation}
    \\
    \group\LowGroupOperation \groupbis
    &\leq
      \group\UppGroupOperation \groupbis
      \eqfinv 
      \label{eq:lower_leq_upper_GroupOperation}
    \\
    \Converse{\group}\UppGroupOperation\Converse{\groupbis}
    &\geq
      \npConverse{\group\UppGroupOperation\groupbis}
      \eqsepv
      \Converse{\group}\LowGroupOperation\Converse{\groupbis}
      \leq
      \npConverse{\group\LowGroupOperation\groupbis} 
      \eqfinv 
    \\
    \group\UppGroupOperation\Converse{\group}
    &\geq
      \UnitElement
      \eqsepv
      \group\LowGroupOperation\Converse{\group}
      \leq
      \UnitElement
      \eqfinv 
    \\
    \group \LowGroupOperation \Converse{\groupbis} \leq \UnitElement
    &\iff
      \group \leq \groupbis
      \iff
      \UnitElement \leq \groupbis \UppGroupOperation \Converse{\group} 
      \eqfinv
    \\
    \group \LowGroupOperation \Converse{\groupbis}
    \leq \groupter
    &\iff
      \group \leq \groupbis \UppGroupOperation \groupter
      \iff
      \group \LowGroupOperation \Converse{\groupter} \leq \groupbis
      \eqfinv 
      \label{eq:lower_upper_GroupOperation_comparisons_b}
    \\  
    \groupter \leq \groupbis \UppGroupOperation \Converse{\group}
    &\iff
      \group \LowGroupOperation \groupter \leq \groupbis
      \iff
      \group \leq \groupbis \UppGroupOperation \Converse{\groupter} 
      \eqfinv
    \\
    ( \group \UppGroupOperation \groupbis ) \LowGroupOperation \groupter 
    &\leq
      \group \UppGroupOperation ( \groupbis \LowGroupOperation \groupter ) 
\eqfinp
  \end{align}
\end{subequations}

\subsection{*-duality}
\label{*-duality}

Let be given two sets $\PRIMAL$ (``primal''), $\DUAL$ (``dual''), together with
a \emph{coupling} function \( \coupling \colon \PRIMAL \times \DUAL \to \barGroup \).  We
define a mapping
\( \Duality\np{\coupling} \colon \barGroup^\PRIMAL \to \barGroup^\DUAL \) as follows:
for any function \( \fonctionprimal \colon \PRIMAL \to \barGroup \), we define the
function \( \Duality\np{\coupling}\fonctionprimal \colon \DUAL \to \barGroup\), denoted
\( \SFM{\fonctionprimal}{\Duality\np{\coupling}} \), by
\cite[Equation~(2.9)]{Martinez-Legaz-Singer:1994}
\begin{subequations}
  \begin{equation}
    \SFM{\fonctionprimal}{\Duality\np{\coupling}}\np{\dual} = 
    \sup_{\primal \in \PRIMAL} \Bp{ \coupling\np{\primal,\dual} 
      \LowGroupOperation \bpConverse{\fonctionprimal\np{\primal} } }
    \eqsepv \forall \dual \in \DUAL
    \eqfinp
    \label{eq:*-duality_definition}
  \end{equation}
  By \cite[Corollary~2.1]{Martinez-Legaz-Singer:1994},
  and using \cite[Lemma~1.5]{Martinez-Legaz-Singer:1994}
  (which establishes
  equation~\eqref{eq:lower_upper_GroupOperation_comparisons_b}),   
  we have that\footnote{%
    In the last term of \cite[Equation~(2.15)]{Martinez-Legaz-Singer:1994},
    the $\overline{A}$ should be a~$A$ (personal communication of
    Juan-Enrique Mart\'\i nez-Legaz).}
  \begin{equation}
    \SFM{\fonctionprimal}{\Duality\np{\coupling}}\np{\dual} = 
    \inf\defset{\group\in\Group}{\coupling\np{\primal,\dual} \leq
      \fonctionprimal\np{\primal}\UppGroupOperation \group
      \eqsepv \forall \primal \in \PRIMAL }
    \eqfinv
        \label{eq:*-duality_as_an_inf}
  \end{equation}
\end{subequations}
and, by \cite[Theory~2.1]{Martinez-Legaz-Singer:1994},
\( \Duality\np{\coupling} \) satisfies the following properties,
that define a *-duality
\cite[Definition~2.3]{Martinez-Legaz-Singer:1994}:
\begin{subequations}
  \begin{align}
    \SFM{\np{\inf_{i\in I}\fonctionprimal_i}}{\Duality\np{\coupling}}
    &=
      \sup_{i\in I}\SFM{\fonctionprimal_i}{\Duality\np{\coupling}}
      \eqsepv \forall \sequence{\fonctionprimal_i}{i\in I} \subset \barGroup^\PRIMAL
      \eqfinv
    \\
    \SFM{\np{\fonctionprimal\UppGroupOperation\group}}{\Duality\np{\coupling}}
    &=
      \SFM{\fonctionprimal}{\Duality\np{\coupling}}
      \LowGroupOperation \Converse{\group}
      \eqsepv \forall \group \in \barGroup
      \eqfinp 
  \end{align}
\end{subequations}
We also have that (deduced from \cite[Equation~(2.27)]{Martinez-Legaz-Singer:1994}):
\begin{equation}
  \SFMbi{\fonctionprimal}{\Duality\np{\coupling}}
  \leq \fonctionprimal
  \label{eq:bipolar-appendix}
  \eqfinp 
\end{equation}

\newcommand{\noopsort}[1]{} \ifx\undefined\allcaps\def\allcaps#1{#1}\fi

\end{document}